\documentclass[11pt, reqno]{amsart}

\usepackage[section]{chaturvedi-graham}

\newcommand*{\claimproof}[1]{\noindent\textit{Proof of Claim~\textup{\ref{#1}}.}\enspace}
\newcommand*{\claimqed}{\hfill$\lozenge$}
\newcommand*{\coeff}{\beta_3}
\newcommand*{\const}{A}
\newcommand*{\cub}{\f{u}}
\newcommand*{\cubder}{\f{m}}
\newcommand*{\err}{\f{e}}
\newcommand*{\dist}{\f{d}}
\newcommand*{\Cub}{\f{U}}
\newcommand*{\Cubder}{\f{M}}
\newcommand*{\Dist}{\f{D}}
\newcommand*{\uo}[1]{u_{#1}^{\mathrm{out}}}
\newcommand*{\ui}[1]{u_{#1}^{\mathrm{in}}}
\newcommand*{\Ui}[1]{U_{#1}^{\mathrm{in}}}
\newcommand*{\uapp}{u^{\mathrm{app}}}

\title{The inviscid limit of viscous Burgers at \break nondegenerate shock formation }
\author{Sanchit Chaturvedi}
\address{SC: Department of Mathematics, Stanford University, 450 Jane Stanford Way, Building 380, Stanford, CA 94305, USA}
\email{\tt sanchat@stanford.edu}

\author{Cole Graham}
\address{CG: Division of Applied Mathematics, Brown University, 182 George St, Providence, RI 02906, USA}
\email{\tt cole\_graham@brown.edu}

\begin{document}

\begin{abstract}
  We study the vanishing viscosity limit of the one-dimensional Burgers equation near nondegenerate shock formation.
  We develop a matched asymptotic expansion that describes small-viscosity solutions to arbitrary order up to the moment the first shock forms.
  The inner part of this expansion has a novel structure based on a fractional spacetime Taylor series for the inviscid solution.
  We obtain sharp vanishing viscosity rates in a variety of norms, including $L^\infty$.
  Comparable prior results break down in the vicinity of shock formation.
  We partially fill this gap.
\end{abstract}


\maketitle

\section{Introduction}

We study the one-dimensional viscous Burgers equation
\begin{equation}
  \label{eq:Burgers-viscous}
  \partial_t u^\nu = \nu \partial_x^2 u^\nu - u^\nu \partial_x u^\nu, \quad (t, x) \in \R^2
\end{equation}
indexed by the viscosity $\nu > 0$.
We are interested in the vanishing viscosity limit $\nu \to 0^+$.
Formally setting $\nu = 0$, we obtain
\begin{equation}
  \label{eq:Burgers-inviscid}
  \partial_t u^0 = - u^0 \partial_x u^0.
\end{equation}
This is the inviscid Burgers equation, the prototypical hyperbolic conservation law.
Inviscid Burgers famously develops shocks---if we start \eqref{eq:Burgers-inviscid} from generic smooth initial data $\mr{u}$, then the solution will form a shock in finite time.
In contrast, for fixed $\nu > 0$, the dissipation in the viscous model is strong enough to prevent shocks: solutions of \eqref{eq:Burgers-viscous} remain smooth for all time.
There is thus a marked qualitative difference between the viscous and inviscid equations.
In this article, we focus on the moment when this difference first manifests.
We study the vanishing viscosity limit $u^\nu \to u^0$ up to the time $u^0$ first forms a shock.
We develop an asymptotic expansion for $u^\nu$ as $\nu \to 0^+$ and thereby prove sharp vanishing viscosity rates in a variety of norms.

\subsection{Background}
The rich nature of the vanishing viscosity limit has long attracted a great deal of attention.
Here, we merely graze its expansive history.

Despite its name, the Burgers equation \eqref{eq:Burgers-viscous} was introduced by Bateman \cite{Bateman_1915}, who used its traveling wave solutions to illustrate the subtlety of the small-viscosity regime in fluids.
In a similar spirit, Burgers considered \eqref{eq:Burgers-viscous} as a toy model for turbulence and discussed the propagation of shocks in the $\nu \to 0$ limit.

Hopf provided the first thorough mathematical treatment of the Burgers equation.
In \cite{Hopf_1950}, he used the eponymous Cole--Hopf transformation (discovered independently by Cole \cite{Cole_1951}) to convert \eqref{eq:Burgers-viscous} to the linear heat equation.
Taking $\nu \to 0$ in this transformation, Hopf derived a variational formula for a particular weak solution of the inviscid problem \eqref{eq:Burgers-inviscid}.

In general, weak solutions of \eqref{eq:Burgers-inviscid} are not unique.
However, a certain ``entropy condition'' inspired by the second law of thermodynamics selects a unique weak solution, known as the entropy solution.
Remarkably, this entropy solution coincides with the inviscid limit of \eqref{eq:Burgers-viscous} studied by Hopf.
The vanishing viscosity limit is thus deeply intertwined with the well-posedness of the \emph{inviscid} equation \eqref{eq:Burgers-inviscid}.

The Cole--Hopf transformation is powerful, but it has a major limitation: it relies on the precise algebraic structure of the Burgers equation.
In particular, it has no analogue for general viscous scalar conservation laws, which have the form
\begin{equation}
  \label{eq:SCL-viscous}
  \partial_t w^\nu = \nu \partial_x^2 w^\nu - \partial_x(f \circ w^\nu)
\end{equation}
for typically convex $f \colon \R \to \R$.
Nonetheless, the Burgers example suggests that $w^\nu$ converges to the unique entropy solution of the corresponding inviscid law as $\nu \to 0.$
This general phenomenon was confirmed by Ole\u{\i}nik~\cite{Oleinik_1957}, who used a finite difference scheme to overcome the lack of the Cole--Hopf transformation.

The above results on the inviscid limit are purely qualitative.
Kru\v{z}kov~\cite{Kruzkov_1965} established the first \emph{quantitative} rate of convergence: $\|w^\nu - w^0\|_{L_x^1} \lesssim \sqrt{\nu}$.
This rate holds in a wide class of advection--diffusion equations.
However, when $f$ is strictly convex, the nonlinear advection in \eqref{eq:SCL-viscous} \emph{accelerates} the convergence.
Goodman and Xin~\cite{GoodXin_1992} showed that the $L^1$ error in the viscous approximation of piecewise-smooth solutions is $\m{O}_\eps(\nu^{1-\eps})$.
Tang and Teng~\cite{TT_1997} subsequently refined this estimate to $\m{O}(\nu \log \nu^{-1})$.
Using the Cole--Hopf transformation, Wang~\cite{Wang_1998} has shown that this improved rate is sharp for viscous Burgers \eqref{eq:Burgers-viscous}.
In fact, if the initial data is smooth, $\|u^\nu - u^0\|_{L_x^1} \asymp \nu \log \nu^{-1}$ precisely when $u^0$ develops a new shock.
(See Section~\ref{subsec:notation} for our definition of relations such as $\m{O},$ $\lesssim$, and $\asymp$.)
This sharp case partially motivates our detailed study of the vanishing viscosity limit around shock formation.

As these results indicate, the bulk of the vanishing viscosity literature for conservation laws focuses on $L^1$ error.
However, both $u^\nu$ and $u^0$ are typically much more regular (e.g., of bounded variation).
We thus wish to quantify ${u^\nu - u^0}$ in stronger norms that reflect this regularity.
Goodman and Xin pioneered such estimates---they described $u^\nu$ with great precision near isolated shocks in systems of hyperbolic conservation laws \cite{GoodXin_1992}.
Their methods and results are the primary motivation for our work.
In \cite{GoodXin_1992}, Goodman and Xin developed a complete matched asymptotic expansion for the viscous solution $u^\nu$ when $u^0$ has a weak well-developed Lax shock (i.e., a small, nonzero jump discontinuity satisfying certain spectral conditions).
In a tour de force, Yu extended this work to include nonlinear wave phenomena arising from discontinuous initial data \cite{Yu_1999}.

Suppose $u^0$ solves \eqref{eq:Burgers-inviscid} with smooth initial data.
Typically, $u^0$ will develop a shock at a finite time $t_* > 0$.
Fix a small time-step $\tau > 0$.
Before time $t_* - \tau$, the inviscid solution $u^0$ is uniformly smooth.
Thus in the pre-shock period $[0, t_* - \tau]$, the viscous approximation \eqref{eq:Burgers-viscous} can be sharply treated by standard perturbation theory.
In particular, it is straightforward to check that $\|u^\nu - u^0\| \lesssim_\tau \nu$ in a variety of norms.
However, such estimates deteriorate as $\tau \to 0$.

On the other hand, \cite{GoodXin_1992} and \cite{Yu_1999} describe the viscous approximation $u^\nu$ in great detail at times after $t_* + \tau$, provided the shock remains isolated and weak.
Again, however, this description degenerates as $\tau \to 0$.
Therefore, neither approach treats the vanishing viscosity limit throughout the crucial period $[t_* - \tau, t_* + \tau]$ of shock formation.
In this article, we fill the first half of this gap: we characterize $u^\nu$ for $\nu \ll 1$ on the entire period $[0, t_*]$.
We develop an asymptotic description of $u^\nu$ that remains valid up to the first moment of shock formation, and thus prove sharp rates of convergence for $u^\nu - u^0$ in, for example, $L^\infty$.
Moreover, we make no use of the Cole--Hopf transformation, so our approach extends to scalar conservation laws \eqref{eq:SCL-viscous} with strictly convex flux.
We leave the fascinating second stage of shock formation $[t_*, t_* + \tau]$ to future work.
\subsection{Setup}
We now discuss the structure of $u^0$ and $u^\nu$ in greater detail.
Suppose $u^0$ solves \eqref{eq:Burgers-inviscid} with initial condition $u^0(0, \anon) = \mr{u}$.
We can treat the inviscid equation \eqref{eq:Burgers-inviscid} with the method of characteristics so long as $u^0$ remains smooth.
Along characteristics, the slope $\partial_x u^0$ solves the Riccati equation $\dot{y} = -y^2$.
Thus negative slopes successively steepen until they reach $-\infty$ in finite time.
Suppose ${\min \partial_x \mr{u} < 0}$.
Then $u^0$ will first develop a shock on the characteristic corresponding to the steepest point of $\mr{u}$, namely $\mr{x} \coloneqq \argmin \partial_x \mr{u}.$

Because $\partial_x u^0$ is minimized along this critical characteristic, $\partial_x^2 u^0$ always vanishes there.
Hence the local behavior of the first shock is determined by the higher-order structure of $\mr{u}$ at $\mr{x}$.
In this article, we assume that the first shock is \emph{nondegenerate}, meaning $\partial_x^3 \mr{u}(\mr{x}) > 0$.
Precisely, we make the following assumptions on our initial data $\mr{u} \colon \R \to \R$.
\begin{enumerate}[label = \textnormal{(H\arabic*)}, leftmargin = 4em, labelsep = 1em, itemsep= 7pt, topsep = 1ex]
\item
  \label{hyp:smooth}
  $\mr{u}$ is smooth;

\item
  \label{hyp:compact-support}
  There exist $\mr{c} \in \R$ and $\mr{L} > 0$ such that $\mr{u} \equiv \mr{c}$ on $[-\mr{L},\mr{L}]^c$;

\item
  \label{hyp:unique-min}
  $\min \partial_x \mr{u} < 0$ and \;$\mr{x} = \argmin \partial_x \mr{u}$\; is unique;

\item
  \label{hyp:nondegenerate}
  $\partial_x^3 \mr{u}(\mr{x}) > 0$.
\end{enumerate}
Define
\begin{equation*}
  t_0 \coloneqq \left(\min \partial_x \mr{u}\right)^{-1} < 0.
\end{equation*}
We shift time to begin at $t_0 < 0$.
For the remainder of the paper, let $u^0$ denote the unique entropy solution to the inviscid equation \eqref{eq:Burgers-inviscid} with
\begin{equation*}
  u^0(t_0, \anon) = \mr{u}.
\end{equation*}
The method of characteristics shows that $u^0$ is smooth until time $t_* = 0$, when it develops an infinite slope at position
\begin{equation*}
  x_* \coloneqq \mr{x} - \mr{u}(\mr{x}) t_0.
\end{equation*}
After a Galilean transformation of spacetime, we are free to assume that
\begin{enumerate}[resume*]
\item
  \label{hyp:gauge}
  \hfil $\displaystyle\mr{x} = \mr{u}(\mr{x}) = 0,$
\end{enumerate}
so that in addition $x_* = 0.$
Thus, we begin with a smooth initial profile $\mr{u}$ at the negative time $t_0$, and the inviscid solution $u^0$ develops its first shock at the origin in spacetime.
The hypotheses \ref{hyp:smooth}--\ref{hyp:nondegenerate} ensure that this first shock is nondegenerate and isolated.
In particular, $u^0(0, \anon)$ is smooth away from $x = 0$.
We depict this nondegenerate shock formation in Figure~\ref{fig:formation}.
\begin{figure}
  \centering
  \includegraphics[width=0.6\linewidth]{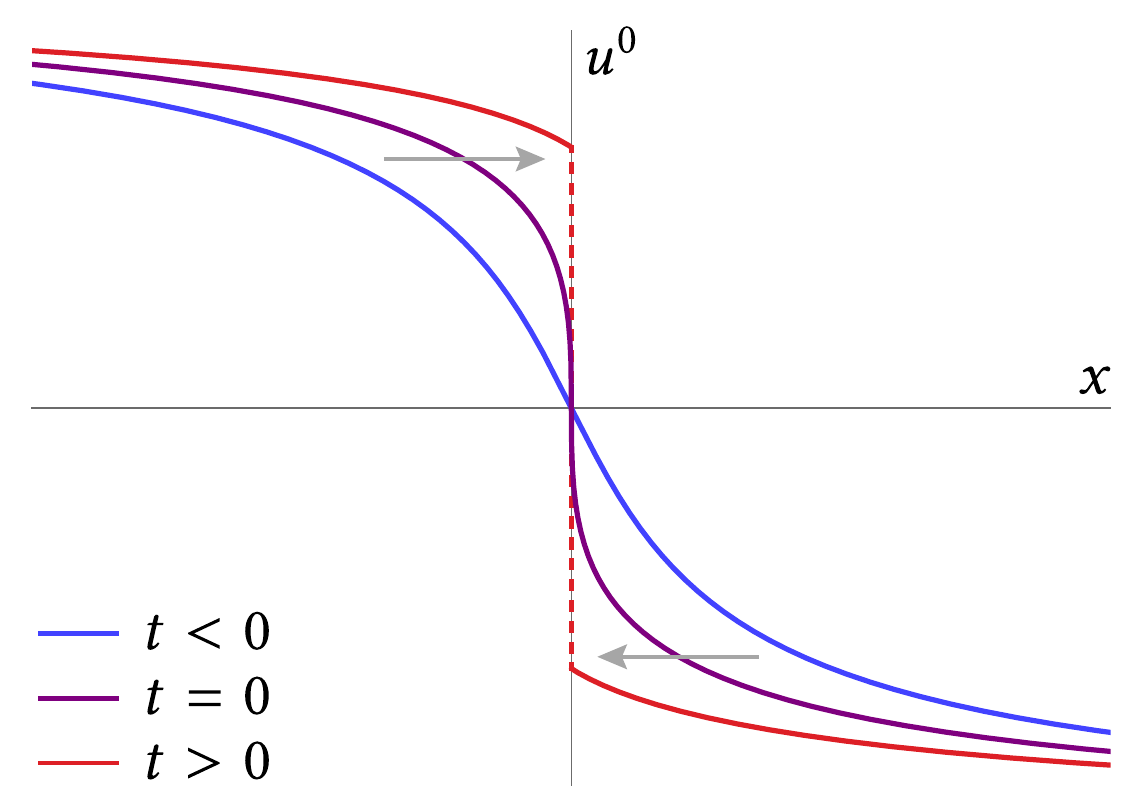}
  \caption{\tbf{Nondegenerate inviscid shock formation near the origin.}
    At times $t < 0$ shortly before shock formation, the solution $u^0(t, \anon)$ (in blue) is smooth.
    At the moment $t = 0$ of shock formation, the solution (in purple) is continuous but not smooth---it has a spatial cusp of regularity $\m{C}^{1/3}$ at $x = 0$.
    After shock formation, the entropy solution (in red) has developed a jump discontinuity---the shock.
  }
  \label{fig:formation}
\end{figure}

We view $u^0$ as the endpoint of a family $(u^\nu)_{\nu \geq 0}$ parameterized by the viscosity.
For simplicity, we use the same initial data for the entire family:
\begin{equation}
  \label{eq:Burgers-viscous-init}
  u^\nu(t_0, \anon) = \mr{u} \ForAll \nu > 0.
\end{equation}
This is not essential---our methods can treat more general initial data that varies smoothly in $\nu$.
However, we do not pursue this direction here.

Our principal contribution is a matched asymptotic expansion linking a traditional ``outer expansion'' with an ``inner expansion'' of a novel form that captures the viscous dynamics of \eqref{eq:Burgers-viscous} near the origin in spacetime, where $u^0$ develops its first shock.
The outer expansion treats the viscous term in \eqref{eq:Burgers-viscous} as a small perturbation and expands in powers of $\nu$:
\begin{equation}
  \label{eq:outer-rough}
  u^\nu \approx \uo{0}+\nu\uo{1}+\nu^2\uo{2}+\ldots \quad \text{as } \nu \to 0.
\end{equation}
The leading term $\uo{0}$ solves the inviscid problem \eqref{eq:Burgers-inviscid}, so in fact $\uo{0} = u^0$.
The corrections $\uo{k}$ for $k \geq 1$ solve linear transport equations forced by $\partial_x^2\uo{k-1}$.

This expansion should be valid wherever the viscous term $\nu \partial_x^2 u^\nu$ is much smaller than the advection $u^\nu \partial_x u^\nu$.
When $\nu \ll 1$, this relation only fails near the origin, where $u^\nu$ develops large derivatives and viscous effects dominate.
Using the method of characteristics, we show that $u^0$ resembles a certain inverse cubic function near the origin.
For instance, at $t = 0$ we have a cubic cusp:
\begin{equation}
  \label{eq:cubic-slice}
  u^0(0, x) \sim -Cx^{1/3} \quad \text{as } x \to 0.
\end{equation}
Using this inverse cubic as an ansatz for $u^\nu$, we can formally check that viscous effects matter when $\abs{t} \lesssim \nu^{1/2}$ and $\abs{x} \lesssim \nu^{3/4}$.
In Section~\ref{sec:overview}, we define a ``distance'' $\dist(t, x) \asymp \abs{t}^{1/2} + \abs{x}^{1/3}$ that quantifies the proximity of $(t, x)$ to the origin $(0,0)$.
With this notation, the viscosity matters where $\dist \lesssim \nu^{1/4}$.
We thus expect the outer expansion \eqref{eq:outer-rough} to hold on the complement $\{\dist \gg \nu^{1/4}\}$.

A different approach becomes necessary when $\dist \lesssim \nu^{1/4}$.
There, \eqref{eq:cubic-slice} suggests that  $u^0 \lesssim \nu^{1/4}$.
This motivates the following ``inner coordinates:''
\begin{equation}
  \label{eq:inner-coords-intro}
  T \coloneqq \nu^{-1/2}t, \quad X \coloneqq \nu^{-3/4}x, \quad U \coloneqq \nu^{-1/4}u.
\end{equation}
We wish to understand the structure of the blown-up solution  $U^\nu(T, X)$ as $\nu \to 0$.

Recall that the outer expansion is valid where $\dist \gg \nu^{1/4}$, which corresponds to large $T$ or $X$.
Thus $U^\nu$ should match the outer expansion at infinity.
Moreover, because the outer expansion holds at the ``macroscopic'' scale $\dist \gtrsim 1$, we are free to restrict our analysis of $U^\nu$ to $\{\dist \ll 1\}$.
Thus it suffices to match $U^\nu$ with the outer expansion on the intermediate zone where $(T, X) \to \infty$ but $(t, x) \to 0$.
That is, we relate the far field of $U^\nu$ to the behavior of the outer expansion near the origin.

In a traditional matched expansion, this correspondence is expressed through Taylor series for the outer expansion~\cite{Fife_1988}.
However, our outer terms are not analytic at the origin: they exhibit fractional behavior like \eqref{eq:cubic-slice}.
To capture this structure, we define a notion of functional homogeneity that is adapted to the scaling in \eqref{eq:inner-coords-intro}; see Definition~\ref{def:homog} below.
We then show that each outer term $\uo{k}$ admits an expansion in homogeneous functions near the origin.
For example,
\begin{equation}
  \label{eq:homog-expansion-intro}
  \uo{0} = u^0 \sim u_{0,0} + u_{0,1} + u_{0,2} + \ldots \quad \text{as } (t, x) \to (0, 0)
\end{equation}
for certain homogeneous functions $u_{0,\ell}(t, x)$ satisfying $|u_{0,\ell}| \leq C_\ell \dist^{\ell + 1}$.
In particular, $u_{0,0}$ is the inverse cubic alluded to above.
\begin{remark}
  Facets of the homogeneous expansion \eqref{eq:homog-expansion-intro} have appeared previously in the literature.
  The leading order $u_{0,0}$ is well known; see, for instance, \mbox{\cite[Proposition~9]{ColGhoMas_2020}}.
  The time-zero slices of the first three terms of \eqref{eq:homog-expansion-intro} were recently constructed in~\mbox{\cite[Lemma~4.11]{BucDriShkVic_2021}}.
\end{remark}
Recall that $U^\nu$ should match $u^0\sim u_{0,0} + u_{0, 1} + \ldots$ where $(T, X) \to \infty$ and $(t, x) \to 0$.
To use this correspondence, we express \eqref{eq:homog-expansion-intro} in the inner coordinates.
Our notion of homogeneity is compatible with the inner scaling, and we can compute $u_{0,\ell}(t, x) \leftrightarrow \nu^{\ell/4} u_{0,\ell}(T, X)$ under \eqref{eq:inner-coords-intro}.
Thus the far field of $U^\nu$ admits an asymptotic expansion in powers of $\nu^{1/4}$.
This motivates the following inner expansion:
\begin{equation}
  \label{eq:inner-rough}
  U^\nu = \Ui{0} + \nu^{1/4} \Ui{1} + \nu^{1/2} \Ui{2} + \ldots \quad \text{as } \nu \to 0.
\end{equation}
We construct functions $\Ui{\ell}(T, X)$ independent of $\nu$ such that \eqref{eq:inner-rough} holds locally uniformly in $(T, X)$.
In fact, \eqref{eq:inner-rough} holds on the entire ``microscopic'' regime $\{\dist \ll 1\}$.
We depict the domains of the inner and outer expansions in Figure~\ref{fig:expansions}.
\begin{figure}
  \centering
  \includegraphics[width=\linewidth]{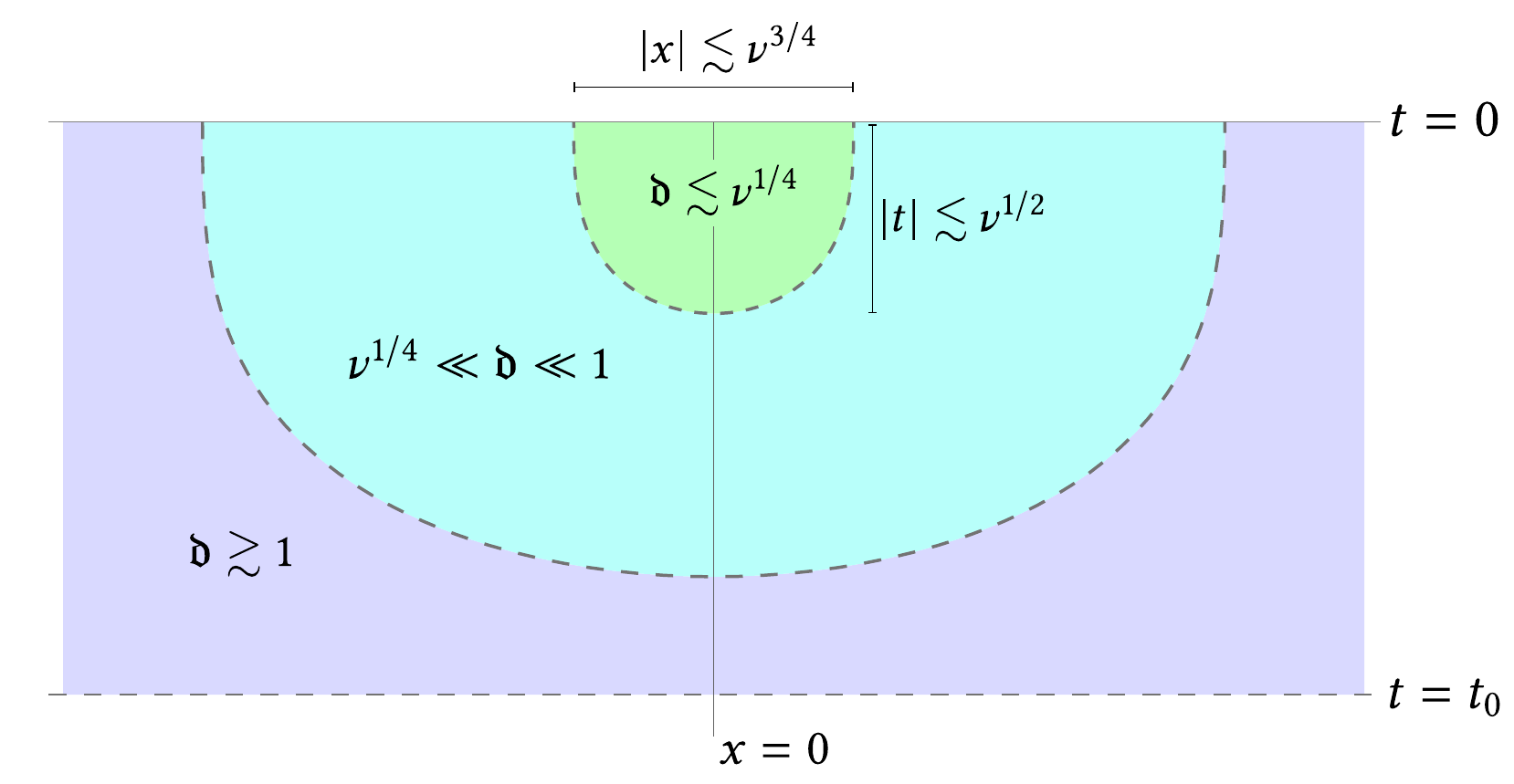}
  \caption{\tbf{Expansion domains.}
    We deploy the outer expansion on the darker blue region $\{\dist \gtrsim 1\}$ and the inner expansion on the green region $\{\dist \lesssim \nu^{1/4}\}.$
    Both are valid on the intermediate turquoise region $\{\nu^{1/4} \ll \dist \ll 1\}$, where we match the two expansions.
    These regions are deliberately imprecise.
    For a concrete choice of inner, outer, and matching regions, see Figure~\ref{fig:IMO} in Section~\ref{sec:overview}.
  }
  \label{fig:expansions}
\end{figure}

The leading term $\Ui{0}$ in \eqref{eq:inner-rough} is an ancient solution of viscous Burgers with unit viscosity.
The higher-order terms solve linear advection--diffusion equations.
We emphasize that these are PDEs, not ODEs.
In contrast to many matched asymptotic expansions, including that of Goodman and Xin \cite{GoodXin_1992}, time retains an essential role in our inner expansion.
\begin{remark}
  The leading term $\Ui{0}$ has a certain universal character.
  It is independent of $\mr{u}$ up to a one-parameter scaling action.
  Moreover, we expect to find $\Ui{0}$ at leading inner order in any strictly convex scalar conservation law \eqref{eq:SCL-viscous}.
  We discuss this universality further in Appendix~\ref{sec:inner-term}.
\end{remark}
Using the matched expansions \eqref{eq:outer-rough} and \eqref{eq:inner-rough}, we construct approximate solutions of \eqref{eq:Burgers-viscous} that incorporate both expansions in their respective domains of validity.
If we include sufficiently many terms from each expansion, standard energy estimates imply that the approximate solution is arbitrarily close to the true solution $u^\nu$.

\subsection{Results}
We now state an informal version of our main result.
For a precise form, see Theorem~\ref{thm:main} below.
\begin{theorem}
  \label{thm:rough}
  Let $u^\nu$ solve \eqref{eq:Burgers-viscous} and \eqref{eq:Burgers-viscous-init} with data $\mr{u}$ satisfying \mbox{\ref{hyp:smooth}--\ref{hyp:nondegenerate}}.
  Given $K \in \Z_{\geq 0},$ there exists an approximate solution $\uapp_{[K]}$ incorporating $K + 1$ terms from each expansion and $C(K, \mr{u}) > 0$ such that for all $\nu \in (0, 1]$,
  \begin{equation*}
    \big\|u^\nu - \uapp_{[K]}\big\|_{L^\infty([t_0, 0) \times \R)} \leq C(K, \mr{u}) \nu^{(K + 2)/5}.
  \end{equation*}
  Moreover, as $\nu \to 0$,
  \begin{equation}
    \label{eq:1/2-regular}
    \sup_{t\in[t_0,0)}\big\|u^\nu(t, \anon)\big\|_{\m{C}^{1/2}} \to \infty \quad\text{while}\quad \sup_{t\in[t_0,0)}\big\|u^\nu(t, \anon) - \uapp_{[K]}(t, \anon)\big\|_{\m{C}^{1/2}} \to 0.
  \end{equation}
\end{theorem}
\begin{remark}
  \label{rem:more-regular}
  The bounds \eqref{eq:1/2-regular} imply that the remainder $u^\nu - \uapp_{[K]}$ is quantitatively more regular than $u^\nu$ itself.
  Thus the approximate solution captures the most irregular part of $u^\nu$.
  In fact, the difference $u^\nu - \uapp_{[K]}$ should become progressively smoother as $K$ increases.
  We expect it to remain uniformly bounded in $\m{C}_x^{(K+2)/3}$ for all $K \in \Z_{\geq 0}$.
  A full proof requires estimates of higher derivatives.
  To simplify this article, we confine our efforts to $H^1$; we therefore only estimate H\"older norms up to $\m{C}^{1/2}$, as in \eqref{eq:1/2-regular}.
  However, our methods readily extend to higher orders of regularity.
\end{remark}
\begin{remark}
  We anticipate that a detailed analysis based solely on the Cole--Hopf transformation would yield similar results.
  However, unlike Cole--Hopf, the approach developed here naturally extends to other strictly convex scalar conservation laws.
\end{remark}
\begin{remark}
  For convenience, we treat smooth initial data, which allows us to state Theorem~\ref{thm:rough} to arbitrary order.
  Naturally, the conclusions of the theorem remain true to some bounded order provided $\mr{u}$ has sufficiently many derivatives.
  For instance, Theorem~\ref{thm:rough} holds for all $K \leq \bar{K}$ if we replace \ref{hyp:smooth} by $\mr{u} \in \mathcal{C}^{L}(\R)$ for $L \geq 3 \bar{K} + 30$.
  We do not attempt to optimize the dependence of $L$ on $\bar{K}$.
\end{remark}
Our inner and outer expansions are relatively explicit, so we can use Theorem~\ref{thm:rough} to resolve concrete questions regarding the vanishing viscosity limit $u^\nu \to u^0$ on $[t_0, 0) \times \R$.
For instance, we obtain the sharp rate of convergence in $L^\infty$.
\begin{corollary}
  \label{cor:inviscid-limit}
  Let $u^\nu$ solve \eqref{eq:Burgers-viscous} and \eqref{eq:Burgers-viscous-init} with data $\mr{u}$ satisfying \mbox{\ref{hyp:smooth}--\ref{hyp:nondegenerate}}.
  Then there exists a constant $C(\mr{u}) \geq 1$ such that for all $\nu \in (0, 1]$,
  \begin{equation}
    \label{eq:inviscid-limit}
    C(\mr{u})^{-1} \nu^{1/4} \leq \big\|u^\nu - u^0\big\|_{L^\infty([t_0, 0) \times \R)} \leq C(\mr{u}) \nu^{1/4}.
  \end{equation}
\end{corollary}
\begin{remark}
  This estimate is saturated near the origin, where $\dist \lesssim \nu^{1/4}$.
  The $\nu^{1/4}$-scale in \eqref{eq:inviscid-limit} reflects the scale of $u$ in \eqref{eq:inner-coords-intro}.
\end{remark}

\subsection{Related works}
Our investigation of shock formation in weakly viscous Burgers joins an exceedingly broad literature on vanishing viscosity and shock formation in conservation laws.
We describe several related efforts, though we do not attempt to be exhaustive.
For an overview of the theory of hyperbolic conservation laws and shocks, we direct the reader to Dafermos~\cite{Dafermos} and Liu~\cite{Liu_2021} as well as the references therein.

\subsubsection*{Vanishing viscosity in 1D}
As the pioneering work of Hopf~\cite{Hopf_1950} and Ole\u{\i}nik~\cite{Oleinik_1957} indicates, vanishing viscosity has long been intertwined with the well-posedness theory of conservation laws.
The behavior of the vanishing viscosity limit is best understood in one spatial dimension.
Ole\u{\i}nik considered strictly convex scalar conservation laws, and showed that the vanishing viscosity limit coincides with the global entropy solution of the inviscid problem.
We note that other regularizations of non-convex laws can yield non-entropy weak solutions in the ``inviscid'' limit~\mbox{\cite{JacMcKShe_1995, BedLeF_2002}.}

Well-posedness for \emph{systems} of hyperbolic conservation laws is much more subtle.
Thanks to the influential work of Glimm~\cite{Glimm_1965}, such systems are known to admit global weak solutions given data that are small in BV.
The relationship between these solutions and vanishing viscosity limits remained open for decades.
This question was finally resolved in the breakthrough paper~\cite{BiaBre_2005} of Bianchini and Bressan, who showed that viscous solutions converge uniquely to Glimm solutions as the viscosity tends to zero.
Their approach requires positive viscosity in each component of the system.
However, many equations of interest like compressible Navier--Stokes involve viscosity in only one component.
The nature of vanishing viscosity limits of Navier--Stokes from general small-BV data remains open.
For recent progress in this direction, see, e.g., \cite{KanVas_2020}.

As noted earlier, the vanishing viscosity limit is typically expressed as convergence in $L^1$.
Stronger pointwise results are known in a handful of cases.
We have already highlighted the asymptotic expansions of Goodman and Xin~\cite{GoodXin_1992} and Yu~\cite{Yu_1999} for hyperbolic conservation laws with weak isolated Lax shocks.
In \cite{Rousset_2003}, Rousset proved the same result under the weaker (sharp) assumption that the viscous shock in the inner expansion is linearly stable.

\subsubsection*{Vanishing viscosity in higher dimensions}
Kru{\v{z}}kov studied scalar conservation laws in multiple spatial dimensions in \cite{Kruzkov_1970}.
Using an $L^1$ modulus of continuity, he showed that viscous solutions converge in $L^1$ to inviscid solutions as $\nu \to 0$.
Much less is known about \emph{systems} in multiple dimensions.
To our knowledge, no general inviscid weak theory is known to be compatible with vanishing viscosity limits.
In fact, it is known that the compressible Euler equations are ill-posed in BV \cite{rauch1986bv}.
Nonetheless, as in one dimension, stronger local results are available for certain shock solutions.
Majda has proven local existence for shocks satisfying certain stability conditions; this seminal work is compiled in \cite{Majda_1983}.
In \cite{GuMeWiZu_2005} and \cite{GuMeWiZu_2006}, Gu\`{e}s, M\'{e}tivier, Williams, and Zumbrun show that such solutions coincide with the vanishing viscosity limit.
In doing so, they prove a full asymptotic expansion \`{a} la Goodman and Xin.

\subsubsection*{Multidimensional shock formation}
Shock formation in multiple spatial dimensions has become a particularly rich subject.
Early works on singularity formation include an indirect proof \cite{Sideris_1985} and a symmetry-reduced construction of non-degenerate shock formation in the compressible Euler equations \cite{Alinhac_1993}.
In a breakthrough, Alinhac established non-degenerate shock formation in a class of quasilinear wave equations including irrotational isentropic Euler~\mbox{\cite{Ali99_2, Ali99_1}}.

Christodoulou revisited this problem in \cite{Chris_2007}.
By combining powerful geometric ideas with robust energy estimates, he established shock formation from small data for irrotational isentropic relativistic Euler.
For a similar treatment of nonrelativistic Euler, see \cite{ChrMia14}.
Christodoulou's methods apply to a class of shock singularities that is more general than the non-degenerate forms treated by Alinhac.
Moreover, on a large open subset of initial data, his approach yields a complete description of the maximal smooth development of the solution in spacetime.
Control of this maximal development enabled Christodoulou to study a closely related system after shock formation~\cite{Chris_2019}.
Recently, Luk and Speck have deployed these geometric methods to prove the stability of planar shocks in the presence of vorticity and entropy up to the first time of singularity \cite{LukSpe_2018, LukSpe_2021}.
For related results including multi-speed problems, see \cite{MiaYu_2017, SpHoLuWo_2016, Speck_2018}.

In a parallel program, Buckmaster, Shkoller, and Vicol have examined the local structure of non-degenerate shock formation from well-prepared initial data~\mbox{\cite{BucShkVic_2019, BucShkVic_2020, BucShkVic_2022}}.
The authors use modulation theory in self-similar variables to study the first point of shock formation in great detail.
They show that the shock profile resembles a self-similar solution of the Burgers equation.
This solution, denoted $u_{0,0}$ in \eqref{eq:homog-expansion-intro}, features prominently in our own work.
Like the present article, these works treat non-degenerate modes of shock formation.
In the opposite direction, Buckmaster and Iyer have constructed degenerate shock formation in 2D polytropic Euler \cite{BucIye_2020}.

\subsubsection*{Post-shock evolution}
In one dimension, small-BV weak solutions are global in time~\cite{Glimm_1965}, so one can study both shock formation and subsequent shock interactions.
In contrast, multidimensional compressible Euler is ill-posed in BV spaces \cite{rauch1986bv}.
Nonetheless, in the physically relevant two and three dimensions, one might hope that solutions can be weakly extended somewhat past their first singularities.
This is known in symmetry-reduced settings: weak solutions have been constructed after shock formation in compressible Euler under spherical~\mbox{\cite{ChrLis16, yin04}} and azimuthal~\cite{BucDriShkVic_2021} symmetry.
In the absence of symmetry, Christodoulou considered a restricted problem wherein the Euler system is forced to remain isentropic and irrotational \cite{Chris_2019}.
In full generality, this ``shock development problem'' for compressible fluids remains open.

\subsubsection*{Singularity formation in related models}
Inviscid Burgers develops shocks in finite time, but they can be prevented through sufficiently strong dissipation or dispersion.
Examples include viscous Burgers and the KdV equation.
In this work, we study the emergence of a singularity as the dissipative coefficient $\nu$ tends to $0$.
In a different direction, one can intrinsically weaken the dissipative or dispersive operator.

The fractional Burgers and KdV equations are essential examples; they feature fractional dissipation and dispersion, respectively~\cite{Whitham}.
Several independent groups~\mbox{\cite{KiNaSh_2008, DongDuLi_2009, AliDroVov_2007}} concurrently established gradient blow-up in fractal Burgers with weak dissipation (and well-posedness under stronger dissipation).
Nonetheless, the precise nature of the singularity remained unclear, and less was known for fractional KdV.
Recently, a flurry of activity rooted in modulation theory has resolved these questions.
Detailed singular descriptions are now known for fractal Burgers \cite{ChiMorPan_2021}, Burgers--Hilbert (a weak case of fractional KdV) \cite{Yang_2020, yang22}, and a broader class of dissipative and dispersive models \cite{OhPas_2021}.

In a somewhat different direction, Collot, Ghoul, and Masmoudi have constructed non-degenerate gradient blow-up in a two-dimensional Burgers equation with solely transverse viscosity \cite{ColGhoMas_2020}.
Like us, \cite{ColGhoMas_2020} makes use of the self-similar Burgers solution $u_{0,0}$.

\subsection{Future directions}
Next, we discuss several problems that may be susceptible to methods developed in this article.

\subsubsection*{Shock development}
As noted earlier, previous results do not provide a uniform pointwise description of $u^\nu$ throughout the time interval $[t_* - \tau, t_* + \tau]$, where $t_*$ denotes the time of shock formation and $\tau$ is a small positive constant.
(In this article, we have shifted time so that $t_* = 0$.)
Our main result addresses the first half of this period: $[t_* - \tau, t_*]$.
However, the structure of $u^\nu$ in the second half $[t_*, t_* + \tau]$ remains open.
In this period, the inviscid solution $u^0$ develops a shock that grows from zero strength.
Presently, it is not known precisely how $u^\nu$ approximates this emergent shock.
It seems likely that a tripartite expansion is required in this setting, corresponding to a region far from the shock, a region close to a \emph{well-developed} shock, and a region close to shock formation.
Thus, a fusion of our approach and that of Goodman and Xin~\cite{GoodXin_1992} seems appropriate.
We leave this intricate case to future study.

\subsubsection*{Systems in 1D}
Hyperbolic systems of conservation laws are notoriously more challenging than their scalar counterparts.
In particular, they resist analysis by the comparison principle, which we use extensively to control the inner expansion.
We believe the machinery developed here may aid the study of shock formation in systems, but new ideas seem to be called for.

\subsubsection*{Higher-dimensional systems}
Taking a broader view, one would like to understand the interaction between viscosity and shock formation in compressible fluids in higher dimensions.
The convergence of compressible Navier--Stokes to Euler is an important test bed.
Once a suitable shock is well-developed, the vanishing viscosity limit is known in detail thanks to \cite{GuMeWiZu_2005}.
However, the behavior around shock formation remains open.

Christodoulou has shown that solutions to an Euler-like PDE can be weakly extended somewhat past the first singularity~\cite{Chris_2019}.
It is desirable on physical grounds to do the same for Euler and to identify the inviscid limit of Navier--Stokes.
As a first step, one could perturb solutions that form planar shocks as in \cite{LukSpe_2018, LukSpe_2021} and incorporate small viscosity.
One could likewise study the setting of Buckmaster, Shkoller, and Vicol~\mbox{\cite{BucShkVic_2019, BucShkVic_2020, BucShkVic_2022}} in the presence of small viscosity.
Their shock formation resembles Burgers, so the approach developed here may serve as a stepping stone in this direction.

\subsection{Organization}
We begin with an overview of our proof strategy in Section~\ref{sec:overview}.
In Section~\ref{sec:outer}, we construct the outer expansion \eqref{eq:outer-rough} and control the homogeneous series \eqref{eq:homog-expansion-intro}.
We treat the inner expansion in Section~\ref{sec:inner}.
In Section~\ref{sec:approx}, we construct approximate solutions and prove the main theorem.
Finally, we collect some interesting calculations concerning the leading inner term $\Ui{0}$ in Appendix~\ref{sec:inner-term}.

\section*{Acknowledgements}
It is our pleasure to thank Jonathan Luk for his encouragement and guidance at every stage.
We are likewise grateful to Tai-Ping Liu, Jonathan Goodman, and Felipe Hern\'{a}ndez for their thoughtful comments.
SC is supported by the National Science Foundation through the grant DMS-2005435.
CG is supported by the NSF Mathematical Sciences Postdoctoral Research Fellowship program through the grant DMS-2103383.

\section{Proof overview}
\label{sec:overview}
In this section, we describe our broad strategy and introduce notation used in the remainder of the paper.
\subsection{Notation}
\label{subsec:notation}
Throughout, constants are implicitly allowed to depend on the initial data, often through $\mr{u}'(0)$ and $\mr{u}'''(0)$.
We say that functions $f$ and $g$ satisfy $f \lesssim g$ if $\abs{f} \leq C \abs{g}$ for some constant $C > 0$ that may depend on $\mr{u}$ but not on the arguments of $f$ and $g$.
We also use the notation $f = \m{O}(g)$ for $f \lesssim g$.
We write $f \asymp g$ if $f \lesssim g$ and $g \lesssim f$.
If the constant $C$ in these relations depends on another parameter $\gamma$, we write $f \lesssim_\gamma g$ or $f = \m{O}_\gamma(g)$.
Finally, we write $f = \smallO(g)$ or $f \ll g$ in the limit $a \to b$ if $f(a)/g(a) \to 0$ as $a \to b$.

\subsection{Inviscid structure}
\label{subsec:hodograph}
To leading order, we know that $u^\nu$ converges to $u^0$.
Moreover, the viscous perturbation should be well behaved wherever $u^0$ is smooth.
Recall that we have shifted time to begin at $t_0 < 0$, so that $u^0$ first forms a shock at the origin in spacetime.
Thus the real challenge arises near the origin, where $u^0$ develops an infinite slope.
As a first step, we analyze the behavior of $u^0$ near the origin.

Because our shock formation is nondegenerate, $u^0(t, x)$ is continuous and decreasing in $x$ near $x = 0$ for each $t \in [t_0, 0]$.
Given $\eps > 0$ and $t \in [t_0, 0]$, let $\m{A}_t(\eps)$ denote the connected component containing $x = 0$ of the open preimage $u^0(t, \anon)^{-1}\big((-\eps, \eps)\big)$.
Define
\begin{equation*}
  \eps_0 \coloneqq \sup\left\{\eps > 0 \mid -\partial_x u^0(0, x) \geq 1\, \text{ for all }\, x \in \m{A}_0(\eps)\right\}.
\end{equation*}
Note that $\eps_0 > 0$ because $-\partial_xu^0(0, 0) = +\infty$.
Then let
\begin{equation}
  \label{eq:hodograph-domain}
  \m{A}_t \coloneqq \m{A}_t(\eps_0), \quad \m{A} \coloneqq \big\{(t, x) \in [t_0, 0] \times \R \mid x \in \m{A}_t\big\}, \And \m{B} \coloneqq (-\eps_0, \eps_0).
\end{equation}
For each $t \in [t_0, 0]$, the function $u^0(t, \anon) \colon \m{A}_t \to \m{B}$ is a continuous bijection.
Let $\omega(t, \anon) \colon \m{B} \to \m{A}_t$ denote its inverse.
If we apply $\partial_t$ to the identity $u^0(t, \omega(t, y)) = y$, the chain rule and \eqref{eq:Burgers-inviscid} yield
\begin{equation}
  \label{eq:hodo-ODE}
  \partial_t \omega(t, y) = y \quad \text{in } (t_0, 0) \times \m{B}.
\end{equation}
This so-called \emph{hodograph transformation} is closely related to the method of characteristics.
Integrating \eqref{eq:hodo-ODE} from $t = t_0$, we obtain
\begin{equation}
  \label{eq:hodograph}
  \omega(t, y) = \mr{\omega}(y) + (t - t_0) y,
\end{equation}
where $\mr{\omega}$ denotes the inverse of $\mr{u} \colon \m{A}_{t_0} \to \m{B}$.

The identity \eqref{eq:hodograph} allows us to control $u^0$ and its derivatives near $x = 0$.
Indeed, \ref{hyp:smooth}--\ref{hyp:gauge} and the chain rule imply that
\begin{equation*}
  \mr{\omega}(0) = 0, \quad \mr{\omega}'(0) = t_0 < 0, \quad \mr{\omega}''(0) = 0, \quad \mr{\omega}'''(0) = t_0^4 \mr{u}'''(0) > 0.
\end{equation*}
Thus by Taylor's theorem,
\begin{equation}
  \label{eq:hodograph-init-Taylor-3}
  \mr{\omega}(y) = t_0 y - \coeff y^3 + \m{O}(y^4) \quad \text{for } \coeff \coloneqq \frac{1}{6} t_0^4 \partial_x^3 \mr{u}(0) > 0.
\end{equation}
Then \eqref{eq:hodograph} yields
\begin{equation}
  \label{eq:hodograph-Taylor}
  \omega(t, y) = ty - \coeff y^3 + \m{O}(y^4).
\end{equation}
The first two terms dominate the behavior of $u^0$ where it is small.
We let $\cub$ denote the unique real solution of
\begin{equation}
  \label{eq:cubic}
  t\cub - \coeff \cub^3 = x.
\end{equation}
Then \eqref{eq:hodograph-Taylor} implies that $u^0 \sim \cub$ as $x \to 0$.

\subsection{Outer expansion}
We can now discuss the outer expansion for $u^\nu$.
Away from the origin, $\partial_x^2 u^0$ is bounded and the viscous term in \eqref{eq:Burgers-viscous} is a small perturbation.
Applying standard perturbation theory, we propose an expansion for $u^\nu$ in powers of $\nu$:
\begin{equation}
  \label{eq:outer-sum}
  u^\nu \sim \sum_{k=0}^\infty \nu^k \uo{k}.
\end{equation}
Recall that $u^\nu(t_0, \anon) = \mr{u}$ for all $\nu > 0$.
If \eqref{eq:outer-sum} holds at $t = t_0$, we must have $\uo{0}(t_0, \anon) = \mr{u}$ and $\uo{k}(t_0, 0) = 0$ for all $k \geq 1$.

Moreover, we can find evolution equations for $\uo{k}$ by formally substituting \eqref{eq:outer-sum} in \eqref{eq:Burgers-viscous} and collecting terms of like order.
At leading (unit) order, $\uo{0}$ solves inviscid Burgers.
Since $\uo{0} = \mr{u}$, the uniqueness of the entropy solution implies that $\uo{0} = u^0$.
For $k \geq 1$, we find
\begin{equation}
  \label{eq:ukout-ev}
  \partial_t \uo{k} = -\partial_x \left(\uo{0}\uo{k}\right) - \sum_{k' = 1}^{k - 1} \uo{k'}\partial_x \uo{k-k'} + \partial_x^2 \uo{k-1}, \quad \uo{k}(t_0, \anon) = 0.
\end{equation}
This is a linear passive-scalar equation with advection $\uo{0}$ and forcing involving the viscous term $\partial_x^2 \uo{k-1}$ and nonlinear cross-terms $\uo{k'}\partial_x \uo{k-k'}$ from earlier in the expansion.

These passive-scalar equations can be integrated quasi-explicitly along the characteristics of $\uo{0}$.
We can thus describe $\uo{k}$ with great precision if we understand the inviscid solution $\uo{0}$ in sufficient detail.

The outer expansion \eqref{eq:outer-sum} should only be valid where the viscous term in \eqref{eq:Burgers-viscous} is much smaller than the other terms, i.e., where
\begin{equation}
  \label{eq:outer-criterion}
  |\nu \partial_x^2 u^\nu| \ll \abs{u^\nu \partial_x u^\nu}.
\end{equation}
This condition fails near the origin in spacetime.
To quantify its failure, we use $\cub$ as a proxy for $u^\nu$.
Define
\begin{equation*}
  \cubder(t, x) \coloneqq \big[\abs{t} + 3 \coeff \cub(t,x)^2\big]^{-1}.
\end{equation*}
Then the chain rule yields
\begin{equation}
  \label{eq:cub-derivs-spatial}
  \partial_x \cub = -\cubder \And \partial_x^2 \cub = -6 \coeff \cub \cubder^3.
\end{equation}
Plugging $\cub$ into \eqref{eq:outer-criterion}, we see that the outer expansion should be valid wherever $|\nu \partial_x^2 \cub| \ll \abs{\cub \partial_x \cub}$.
Using \eqref{eq:cub-derivs-spatial}, this is equivalent to
\begin{equation}
  \label{eq:outer-criterion-cubder}
  \nu \cubder^2 \ll 1.
\end{equation}
For later convenience, we define
\begin{equation*}
  \dist(t, x) \coloneqq \cubder(t, x)^{-1/2} = \big[\abs{t} + 3 \coeff \cub(t, x)^2\big]^{1/2}.
\end{equation*}
We think of $\dist$ as a ``distance to the origin'' that is naturally adapted to our problem.
We note, however, that it is not a metric in the technical sense.
Rewriting \eqref{eq:outer-criterion-cubder}, we see that the outer expansion \eqref{eq:outer-sum} should be valid where
\begin{equation}
  \label{eq:outer-criterion-dist}
  \dist \gg \nu^{1/4}.
\end{equation}
In Lemma~\ref{lem:cubic} below, we show that
\begin{equation*}
  \dist(t, x) \asymp \max\left\{\abs{t}^{1/2}, \, \abs{x}^{1/3}\right\} \quad \text{in } \R_- \times \R.
\end{equation*}
Thus \eqref{eq:outer-criterion-dist} is equivalent to the disjunction
\begin{equation}
  \label{eq:outer-criterion-coords}
  \abs{t} \gg \nu^{1/2} \quad \text{or} \quad \abs{x} \gg \nu^{3/4}.
\end{equation}

For future convenience, we collect these definitions in one place:
\begin{definition}
  \label{def:cubic}
  The profile $\cub$ is the unique real solution of $t\cub - \coeff \cub^3 = x.$
  Its slope $\cubder$ is $-\partial_x \cub  = \big(\abs{t} + 3 \coeff \cub^2\big)^{-1}$ and the distance $\dist$ is $\cubder^{-1/2} = \big(\abs{t} + 3 \coeff \cub^2\big)^{1/2}.$
\end{definition}
\noindent
These functions are ubiquitous in the remainder of the paper.
In the following lemma, we collect some of their properties.
\begin{lemma}
  \label{lem:cubic}
  On $\R_- \times \R$, we have
  \begin{equation}
    \label{eq:cubic-bounds}
    \abs{\cub} \asymp \big(\abs{x}\abs{t}^{-1}\big) \vee \abs{x}^{1/3}, \quad \cubder \asymp \big(\abs{t} + \cub^2\big)^{-1}, \And \dist \asymp \big(\abs{x} \vee \abs{t}^{3/2}\big)^{1/3}.
  \end{equation}
  Moreover, for all $t < 0$ and $x \neq 0$, $\cubder \leq \abs{\cub} \abs{x}^{-1}$.
\end{lemma}
\begin{proof}
  Because $t < 0$, we have
  \begin{equation}
    \label{eq:cubic-abs}
    \abs{\cub}\big(\abs{t} + \coeff \cub^2\big) = \abs{x}.
  \end{equation}
  Suppose $\abs{t} \geq \coeff \cub^2$.
  Then \eqref{eq:cubic-abs} yields
  \begin{equation}
    \label{eq:cubic-small-x}
    \frac{1}{2}\abs{x}\abs{t}^{-1} \leq \abs{\cub} \leq \abs{x}\abs{t}^{-1}.
  \end{equation}
  On the other hand, if $\abs{t} \leq \coeff \cub^2$, then \eqref{eq:cubic-abs} implies that
  \begin{equation}
    \label{eq:cubic-large-x}
    (2\coeff)^{-1/3} \abs{x}^{1/3} \leq \abs{\cub} \leq \coeff^{-1/3} \abs{x}^{1/3}.
  \end{equation}
  Together, \eqref{eq:cubic-small-x} and \eqref{eq:cubic-large-x} imply the first part of \eqref{eq:cubic-bounds}.
  
  Next, $\cubder \asymp \big(\abs{t} + \cub^2\big)^{-1}$ is immediate from Definition~\ref{def:cubic}.
  It follows that
  \begin{equation*}
    \dist = \cubder^{-1/2} \asymp \big(\abs{t} + \cub^2\big)^{1/2}.
  \end{equation*}
  Hence if $\abs{t} \geq \coeff \cub^2$, we obtain $\dist \asymp \abs{t}^{1/2}$.
  On the other hand, if $\abs{t} \leq \coeff \cub^2$, \eqref{eq:cubic-large-x} yields $\dist \asymp \abs{\cub} \asymp \abs{x}^{1/3}$.
  This completes the proof of \eqref{eq:cubic-bounds}.

  Finally, the definition of $\cubder$ and \eqref{eq:cubic-abs} imply that
  \begin{equation*}
    \cubder = \frac{1}{\abs{t} + 3 \coeff \cub^2} \leq \frac{\abs{\cub}}{\abs{\cub}\big(\abs{t} + \coeff \cub^2\big)} = \frac{\abs{\cub}}{\abs{x}}.\qedhere
  \end{equation*}
\end{proof}

\subsection{Inner expansion}

When $\dist \lesssim \nu^{1/4}$, we must reckon with the viscosity in \eqref{eq:Burgers-viscous}.
To focus on this region, we perform a blow-up about the origin by introducing the ``inner coordinates'' \eqref{eq:inner-coords-intro}, which we recall here:
\begin{equation}
  \label{eq:inner-var}
  T \coloneqq \nu^{-1/2} t, \quad  X \coloneqq \nu^{-3/4} x, \And U \coloneqq \nu^{-1/4} u.
\end{equation}
These scaling exponents are motivated by \eqref{eq:outer-criterion-coords}.
The blown-up solution
\begin{equation*}
  U^\nu(T, X) \coloneqq \nu^{-1/4} u^\nu(\nu^{1/2} T, \nu^{3/4}X)
\end{equation*}
satisfies viscous Burgers with \emph{unit} viscosity in the blown-up domain $(\nu^{-1/2}t_0, 0) \times \R$:
\begin{equation}
  \label{eq:viscous-Burgers-unit}
  \partial_T U^\nu = - U^\nu \partial_X U^\nu + \partial_X^2 U^\nu.
\end{equation}
We observe that this domain approaches the past half-plane $\R_- \times \R$ as $\nu \to 0$.
We postulate that $(U^\nu)_{\nu > 0}$ converges locally uniformly in the $(T, X)$ coordinates to a special solution $\Ui{0}$ of \eqref{eq:viscous-Burgers-unit} defined on $\R_- \times \R$.

Now consider the regime $\nu^{1/4} \ll \dist \ll 1$.
In this region, $(t, x) \to (0, 0)$ while $(T, X) \to \infty$ as $\nu \to 0$.
Because $\dist \gg \nu^{1/4}$, the outer expansion is valid, and $u^\nu \sim \uo{0} = u^0$.
On the other hand, $(t, x) \to 0$ implies that $u^0 \sim \cub$.
Thus $u^\nu \sim \cub$ in this regime.
We expect $\Ui{0}$ to reflect this behavior.
If
\begin{equation*}
  \Cub(T, X) \coloneqq \nu^{-1/4} \cub(\nu^{1/2} T, \nu^{3/4}X) = \cub(T, X)
\end{equation*}
denotes the blown-up inverse cubic, we should have $\Ui{0} \sim \Cub$ at infinity.
That is, matching with the outer expansion imposes a far-field condition on $\Ui{0}$.

In Section~\ref{sec:inner}, we construct an ancient solution $\Ui{0}$ of \eqref{eq:viscous-Burgers-unit} in $\R_- \times \R$ that resembles $\Cub$ at infinity.
See Proposition~\ref{prop:Ul} for details.
This solution admits an explicit integral representation via the Cole--Hopf transformation.
We discuss its structure further in Appendix~\ref{sec:inner-term}.

The special solution $\Ui{0}$ only involves the \emph{leading} order $\cub$ of $\uo{0}$ near the origin.
The full inner expansion must include higher-order information.
To express these higher orders, we introduce a notion of functional homogeneity that respects the natural scaling of the problem.
\begin{definition}
  \label{def:homog}
  Given $r \in \R$, a function $f \colon \R_- \times \R \to \R$ is \emph{$r$-homogeneous} if
  \begin{equation*}
    f\big(\lambda^2 t, \lambda^3 x\big) = \lambda^r f(t, x)
  \end{equation*}
  for all $t \in \R_-$, $x \in \R$, and $\lambda \in \R$.
\end{definition}
\noindent
We note that this definition strays from standard terminology, as we scale spacetime by $\big(\lambda^2 t, \lambda^3 x\big)$ rather than $(\lambda t, \lambda x)$.
\begin{example}
  Using Definition~\ref{def:cubic}, we can check that $\cub$ and $\dist$ are $1$-homogeneous while $\cubder$ is $-2$-homogeneous.
\end{example}
We show that $\uo{0}$ admits a homogeneous expansion about the origin:
\begin{equation}
  \label{eq:u0-homog-sum}
  \uo{0}(t, x) \sim \sum_{\ell \geq 0} u_{0,\ell}(t, x) \quad \text{as } (t, x) \to (0, 0)
\end{equation}
for certain $(\ell + 1)$-homogeneous function $u_{0,\ell}$.
In particular, $u_{0,0} = \cub$.
For a precise statement, see Lemma~\ref{lem:u0-out-homog}.
Due to their increasing degrees of homogeneity, the terms in \eqref{eq:u0-homog-sum} vanish to successively higher order at the origin.
Quantitatively,
\begin{equation}
  \label{eq:homog-bound}
  \abs{u_{0,\ell}} \lesssim_\ell \dist^{\ell + 1},
\end{equation}
recalling that the distance $\dist$ from Definition~\ref{def:cubic} vanishes precisely at $(t, x) = (0, 0)$.

The change-of-variables $(t, x) = (\nu^{1/2}T, \nu^{3/4}X)$ in \eqref{eq:inner-var} exhibits the scaling in Definition~\ref{def:homog} with $\lambda = \nu^{1/4}$.
By $(\ell + 1)$-homogeneity, $u_{0,\ell}(t, x) = \nu^{(\ell + 1)/4}u_{0, \ell}(T, X)$.
Accounting for the scaling $U = \nu^{-1/4} u$ in \eqref{eq:inner-var}, we see that $u_{0,\ell}(t, x)$ corresponds to a term of order $\nu^{\ell/4}$ in the inner coordinates.

Recall that $\uo{0}$ prescribes the far field of $U^\nu$ in the inner coordinates.
We have just argued that this field admits an expansion in powers of $\nu^{1/4}$ corresponding to \eqref{eq:u0-homog-sum}.
Extending this structure to the entire inner region $\{\dist \ll 1\}$, we are led to the inner expansion:
\begin{equation}
  \label{eq:inner-sum}
  U^\nu \sim \sum_{\ell \geq 0} \nu^{\ell/4} \Ui{\ell}.
\end{equation}
Plugging \eqref{eq:inner-sum} into \eqref{eq:viscous-Burgers-unit} and collecting like terms, we obtain the evolution equations
\begin{equation}
  \label{eq:Ul-ev-overview}
  \partial_T \Ui{\ell} = -\frac{1}{2}\sum_{\ell' + \ell'' = \ell} \partial_X(\Ui{\ell'} \Ui{\ell''}) + \partial_X^2 \Ui{\ell}.
\end{equation}
We now return to the original coordinates; define
\begin{equation*}
  \ui{\ell}(t, x) \coloneqq \nu^{(\ell + 1)/4} \Ui{\ell}\big(\nu^{-1/2}t, \nu^{-3/4}x\big).
\end{equation*}
Then we find an expansion
\begin{equation}
  \label{eq:inner-sum-outer-coords}
  u^\nu \sim \sum_{\ell \geq 0} \ui{\ell}
\end{equation}
that is valid when $\dist \ll 1$.
Matching between $u^0$ and \eqref{eq:inner-sum-outer-coords} imposes the condition $\ui{\ell} \sim u_{0,\ell}$ where $\nu^{1/4} \ll \dist \ll 1$.
By \eqref{eq:homog-bound}, $\ui{\ell}$ is roughly of size $\dist^{\ell + 1}$.
Therefore the series \eqref{eq:inner-sum-outer-coords} represents a true hierarchy of scales when $\dist \ll 1$.

We have now described two expansions for $u^\nu$: an outer expansion \eqref{eq:outer-sum} on ${\{\dist \gg \nu^{1/4}\}}$ and an inner expansion \eqref{eq:inner-sum-outer-coords} on $\{\dist \ll 1\}$.
If both are valid, the two expansions must be ``compatible'' in some sense on the overlap region ${\{\nu^{1/4} \ll \dist \ll 1\}}$.
To express this compatibility, we extend the homogeneous expansion \eqref{eq:u0-homog-sum} to all $k \geq 0$:
\begin{equation}
  \label{eq:outer-homog-expansion}
  \uo{k} \sim \sum_{\ell \geq 0} u_{k,\ell} \quad \text{when } \dist \ll 1
\end{equation}
for certain $(-4k+\ell+1)$-homogeneous functions $u_{k,\ell}$.
We then construct our inner expansion so that
\begin{equation}
  \label{eq:inner-homog-expansion}
  \ui{\ell} \sim \sum_{k \geq 0} \nu^k u_{k,\ell} \quad \text{when } \dist \gg \nu^{1/4}.
\end{equation}
For precise statements, see Propositions~\ref{prop:uk-homogeneous} and \ref{prop:Ul} below.

We can visualize our two expansions via the grid of terms $(\nu^k u_{k,\ell})_{k,\ell \geq 0}$ in Figure~\ref{fig:grid}.
The outer expansion $\uo{k}$ sums the \emph{rows} of the grid and the inner expansion $\ui{\ell}$ sums the \emph{columns}.
\begin{figure}
  \centering
  \begin{tikzpicture}[scale = 1]
    \tikzstyle{every node}=[font=\small]

    \def\k{3}
    \def\Xeps{0.6}
    \def\Yeps{0.6}

    \fill [green, opacity = 0.2, rounded corners=5pt] (- 0.7, -2 + 0.5) rectangle (\k + 1 + 0.5, -2 - 0.5);
    \fill [blue, opacity = 0.1, rounded corners=5pt] (1 - 0.5, 0.5) rectangle (1 + 0.5, -\k - 1 - 0.5);

    \foreach \x in {0,1,...,\k}
    {\node at (\x, 0) {$ u_{0, \x}$};
      \node at (\x, -1) {$\nu u_{1, \x}$};
      \foreach \y in {2,...,\k}
      {\node at (\x, -\y) {$\nu^{\y} u_{\y, \x}$};}
      \node at (\x, -\k-0.9) {$\vdots$};}
    
    \foreach \y in {0,...,\k}
    {\node at (\k + 1, -\y) {$\cdots$};}

    \node at (\k + 1, -\k - 0.8) {$\ddots$};

    \node[green!40!black] at (-1.4, -2) {$\nu^2 \uo{2}$};
    \node[blue!60!black] at (1, 1) {$\ui{1}$};
  \end{tikzpicture}
  \caption{
    \tbf{Asymptotic structure of the inner and outer expansions:} $\nu^2\uo{2}$ is asymptotically equivalent to the sum of the green row, while $\ui{1}$ is equivalent to the sum of the blue column.
  }
  \label{fig:grid}
\end{figure}
This figure explains how the outer and inner expansions \eqref{eq:outer-sum} and \eqref{eq:inner-sum-outer-coords} can be simultaneously valid in the intermediate region $\nu^{1/4} \ll \dist \ll 1$.
There, $u^\nu$ is asymptotically equivalent to the sum of \emph{all} terms in Figure~\ref{fig:grid}.
The two asymptotic expansions simply express two orders of summation.
We emphasize, however, that these asymptotic identities are only simultaneously valid when $\nu^{1/4} \ll \dist \ll 1$.

\subsection{An approximate solution}
\label{subsec:approx}
Next, we combine our two compatible expansions to construct approximate solutions of \eqref{eq:Burgers-viscous}.
Given $\alpha \in (0, 1)$, define the inner region
\begin{equation*}
  I \coloneqq \big\{(t, x) \in [t_0, 0) \times \R \mid \dist(t, x) <  \nu^{\alpha}\big\},
\end{equation*}
the outer region
\begin{equation*}
  O \coloneqq \big\{(t, x) \in [t_0, 0) \times \R \mid \dist(t, x) > 2\nu^{\alpha}\big\},
\end{equation*}
and the matching zone
\begin{equation*}
  M \coloneqq \big\{(t, x) \in [t_0, 0) \times \R \mid\nu^{\alpha}\leq \dist(t, x) \leq 2\nu^{\alpha}\big\}.
\end{equation*}
These regions are chosen so that the effect of viscosity on $u^\nu$ is negligible in $O$ but relevant in $I$.
As we shall see in Section~\ref{sec:approx}, the optimal choice for our scheme is $\al = 1/5$.
We depict these regions in Figure~\ref{fig:IMO}.
\begin{figure}
  \centering
  \includegraphics[width=\linewidth]{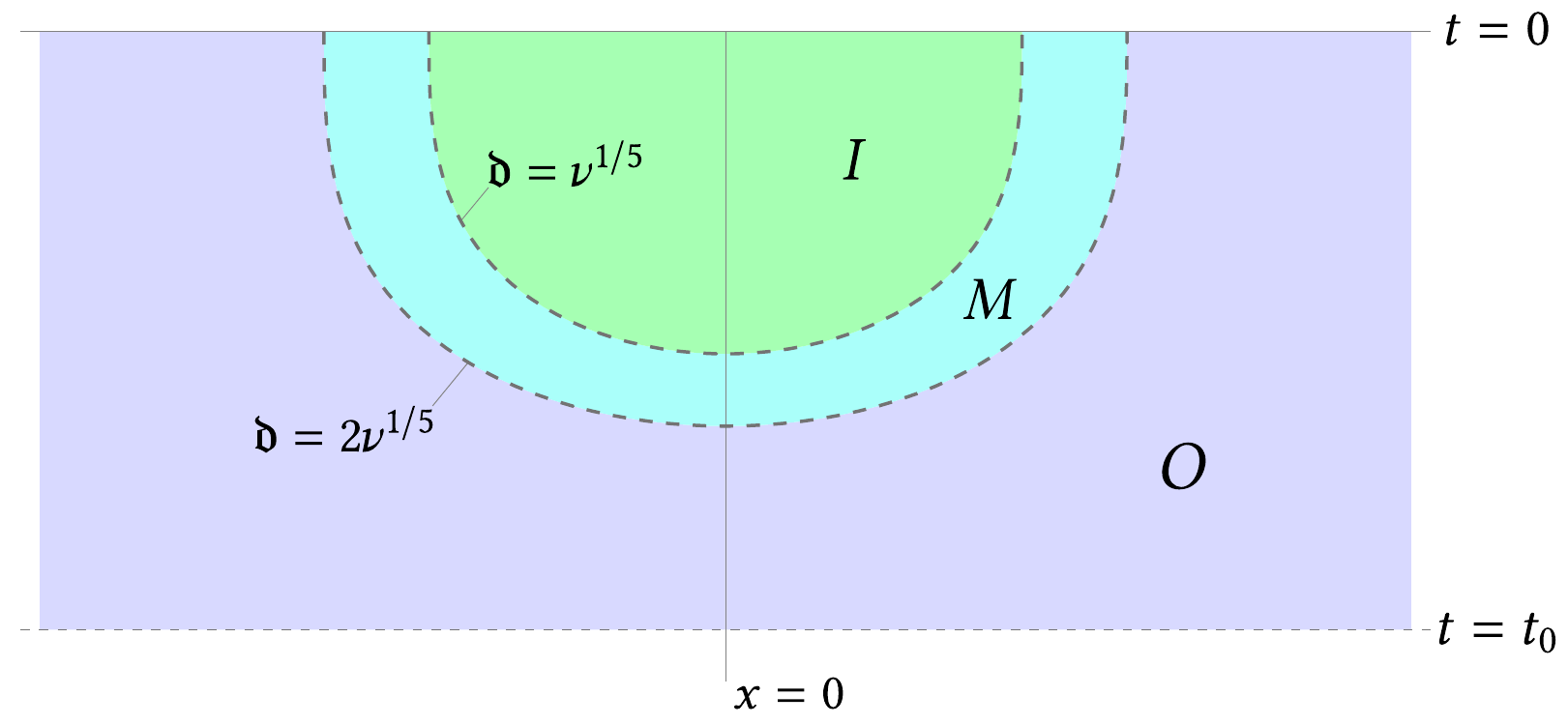}
  \caption{\tbf{The inner, outer, and matching zones.}
    These regions are precisely delineated versions of those in Figure~\ref{fig:expansions}.
  }
  \label{fig:IMO}
\end{figure}

Now let $\vartheta \in \m{C}^\infty(\R_{\geq 0})$ satisfy $0 \leq \vartheta \leq 1$, $\vartheta|_{[0, 1]} \equiv 1$, and $\vartheta|_{[2, \infty)} \equiv 0$.
We define the smooth cutoff
\begin{equation}\label{eq:theta-def}
  \theta(t, x) \coloneqq \vartheta\big(\dist(t, x)\nu^{-\alpha}\big),
\end{equation}
which interpolates between $1$ on $I$ and $0$ on $O$.

Fix $K \in \Z_{\geq 0}$ and define the partial sums
\begin{equation*}
  \uo{[K]} \coloneqq \sum_{k = 0}^K \nu^k \uo{k} \And \ui{[K]} \coloneqq \sum_{\ell = 0}^K \ui{\ell}.
\end{equation*}
We define an approximate solution $\uapp_{[K]}$ of \eqref{eq:Burgers-viscous} by
\begin{equation*}
  \uapp_{[K]} \coloneqq \theta \ui{[K]} + (1 - \theta) \uo{[K]}.
\end{equation*}
The difference $v \coloneqq u^\nu - \uapp_{[K]}$ satisfies a PDE of the form
\begin{equation*}
  \partial_t v = -\partial_x(\uapp_{[K]} v) - v \partial_x v + \nu\partial_x^2 v + E, \quad v(t_0, \anon) = 0,
\end{equation*}
for some forcing $E$ satisfying $\norm{E}_{L_t^\infty H_x^1} \lesssim \nu^{(2K-3)/10}$.
Exploiting the small but positive diffusion, we close an estimate on $v$ in $L_t^\infty H_x^1$ and conclude that
\begin{equation*}
  \norm{v}_{L_t^\infty H_x^1} \lesssim \nu^{(K-6)/5}.
\end{equation*}
Choosing $K$ large, we see that $v$ is small in $L_t^\infty H_x^1$, and hence in $L_{t, x}^\infty$.
Thus $u^\nu \approx \uapp$, as desired.
This analysis can be carried out in arbitrarily high Sobolev norms, so we obtain a complete description of $u^\nu$ to all orders.

\section{The outer expansion}
\label{sec:outer}
In this section, we study the terms in \eqref{eq:outer-sum}.
We begin with the observation that $\uo{k}$ is trivial outside a compact set.
Recall $\mr{c}$ from \ref{hyp:compact-support}.
\begin{lemma}
  \label{lem:compact}
  There exists $L > 0$ such that on $[t_0, 0) \times [-L, L]$, $\uo{0} = \mr{c}$ and $\uo{k} = 0$ for all $k \geq 1$.
\end{lemma}
\begin{proof}
  By \ref{hyp:compact-support}, $\mr{u}(x) = \mr{c}$ for all $\abs{x} \geq \mr{L}$.
  Define the region
  \begin{equation*}
    \m{D} \coloneqq \big\{(t, x) \in [t_0, 0) \times \R \mid |x - \mr{c}(t - t_0)| \geq \mr{L}\big\}.
  \end{equation*}
  Using the method of characteristics, we see that $\uo{0} = \mr{c}$ on $\m{D}$.

  We claim that $\uo{k} = 0$ on $\m{D}$ for all $k \geq 1$.
  We show this by induction.
  Fix $k \geq 1$ and suppose we have shown the claim for all $1 \leq k' < k$.
  Note that this assumption is vacuous in the base case $k = 1$.
  Using \eqref{eq:ukout-ev}, we see that
  \begin{equation*}
    \partial_t \uo{k} = - \partial_x\big(\uo{0} \uo{k}\big) \quad \text{in } \m{D}.
  \end{equation*}
  Now $\uo{k}$ starts from $0$ and $\m{D}$ is a union of characteristics of $\uo{0}$.
  Hence $\uo{k} = 0$ in $\m{D}$, as desired.
  The claim follows from induction.
  To obtain the lemma, we take $L \coloneqq |\mr{c}t_0| + \mr{L}$. 
\end{proof}
\subsection{Homogeneous algebra}
In the remainder of the section, we focus on the behavior of the outer expansion \eqref{eq:outer-sum} near $x = 0$.
We use the inverse cubic $\cub$ and its slope $\cubder$ from Definition~\ref{def:cubic} as building blocks.
It is therefore essential that we understand the derivatives of these functions in some detail.

Recall that the functions $\cub$ and $\cubder$ are homogeneous of degrees $1$ and $-2$, respectively.
In light of the scaling in Definition~\ref{def:homog}, the derivatives $\partial_t$ and $\partial_x$ reduce homogeneity by $2$ and $3,$ respectively.
In particular, we can use Definition~\ref{def:cubic} and the chain rule to compute
\begin{equation}
  \label{eq:derivs}
  \partial_x \cub = -\cubder, \quad \partial_t \cub = \cub \cubder, \quad \partial_x \cubder = 6 \coeff \cub \cubder^3, \quad \partial_t \cubder = \cubder^2(1 - 6 \coeff \cub^2 \cubder).
\end{equation}
It follows that the space of polynomials in $\cub$ and $\cubder$ is closed under differentiation.

Let $\m{P}$ denote the $\R$-algebra of $0$-homogeneous real polynomials in $\cub$ and $\cubder$.
By $0$-homogeneity, $\m{P}$ is generated as an algebra by $\cub^2 \cubder$.
That is, $\m{P} = \R\big[\cub^2 \cubder\big]$.
Given $a, b \in \Z_{\geq 0},$ we then define
\begin{equation*}
  \m{P}(a, b) = \cub^a \cubder^b \m{P},
\end{equation*}
which is a real vector space of $(a-2b)$-homogeneous functions.
Using \eqref{eq:derivs}, we see that
\begin{equation}
  \label{eq:deriv-map}
  \partial_x \colon \m{P}(a, b) \to \m{P}(a - 1, b + 1) \And \partial_t \colon \m{P}(a, b) \to \m{P}(a, b + 1)
\end{equation}
when $a \geq 1$ and $b \geq 0$.

We would like to extend this pattern to $a \leq 0$, but something new happens: $\partial_x \cubder \in \m{P}(1, 3)$.
Thus $\partial_x \colon \m{P}(0, 1) \to \m{P}(1,3)$.
In light of \eqref{eq:deriv-map}, we therefore \emph{define} $\m{P}(-1, 2) \coloneqq \m{P}(1, 3)$.
More generally, we extend our notation to negative indices as follows:
\begin{equation*}
  \m{P}(-a, b) \coloneqq \m{P}(-a + 2 \ceil{a/2}, b + \ceil{a/2}) \For a,b \in \Z_{\geq 0}.
\end{equation*}
Thus we have formally replaced $\cub^{-a}$ by $\cub^{-a} (\cub^2 \cubder)^{\ceil{a/2}}$ so that the result is still a polynomial in $\cub$ and $\cubder$.
With this convention, \eqref{eq:deriv-map} holds for all $a \in \Z$ and $b \in \Z_{\geq 0}$.

We now introduce some bounds adapted to this algebra.
We first note that homogeneous functions are determined by their values on the quarter-sphere
\begin{equation*}
  \s{Q} \coloneqq \big\{\abs{t}^2 + \abs{x}^2 = 1,\, t < 0,\, x \geq 0\big\}.
\end{equation*}
Because
\begin{equation*}
  0 < \inf_{\s{Q}} \dist < \sup_{\s{Q}} \dist < \infty
\end{equation*}
and $\dist$ is $1$-homogeneous, we have
\begin{equation}
  \label{eq:dist-bd}
  \abs{f} \leq C(f) \dist^r \quad \text{on } \R_- \times \R
\end{equation}
for every $r$-homogeneous function $f$ that is bounded on $\s{Q}$.
In particular, \eqref{eq:dist-bd} holds if $f$ is continuous in $(-\infty,0] \times \R \setminus \{(0,0)\}$.
For this reason, we frequently express estimates in terms of powers of $\dist$.

Next, we introduce an $(a-2b)$-homogeneous function $\err(a, b)$ to control polynomials in $\cub$ and $\cubder$.
For $a, b \in \Z$, let
\begin{equation}\label{eq:poly-err}
  \err(a, b) \coloneqq
  \begin{cases}
    \abs{\cub}^a \cubder^b & \text{if } a \geq 0,\\
    \cubder^{b + \abs{a}/2} & \text{if } a < 0.
  \end{cases}
\end{equation}
If $p \in \m{P}(a, b)$, then
\begin{equation}
  \label{eq:err-bd}
  \abs{p} \leq C(p) \err(a, b) \quad \text{on } \R_- \times \R
\end{equation}
for some $C(p) > 0$.
The functions $\err(a, b)$ are refined versions of $\dist^{a - 2b}$ that track the order of vanishing at $x = 0$ when $t < 0$.
As it happens, this vanishing is not relevant when $a < 0$, so then we simply let $\err(a, b) = \dist^{a - 2b}$.

\subsection{The inviscid solution}

We can now justify the expansion \eqref{eq:u0-homog-sum} for $\uo{0}$ near the origin.
In fact, we will verify the expansion wherever $\abs{x} \ll 1$.
We recall the identity \eqref{eq:hodograph} for the local spatial inverse $\omega$ of $u^0 = \uo{0}$.
Because the inverse $\mr{\omega}$ of $\mr{u}$ is smooth on the domain $\m{B}$ defined in \eqref{eq:hodograph-domain}, it admits a Taylor expansion.
Refining \eqref{eq:hodograph-init-Taylor-3}, we have
\begin{equation*}
  \mr{\omega}(y) = t_0 y - \coeff y^3 + \sum_{m=4}^M \beta_m y^m + \m{O}_M\big(y^{M + 1}\big) \quad \text{as } y \to 0
\end{equation*}
for each $M \in \Z_{\geq 4}$ and some coefficients $\beta_m \in \R$.
Using \eqref{eq:hodograph}, we find
\begin{equation}
  \label{eq:inverse-power-series}
  \omega(t, y) = t y - \coeff y^3 + \sum_{m=4}^M \beta_m y^m + \m{O}_M\big(y^{M + 1}\big) \quad \text{as } y \to 0
\end{equation}
for each $M \in \Z_{\geq 4}$.
We also permit $M = 3$ in \eqref{eq:inverse-power-series} under the convention that the empty sum is zero.

As noted earlier, this expansion suggests that $\uo{0} \sim \cub$ where both are small.
\begin{lemma}
  \label{lem:sing-IFT}
  We have $\uo{0} = \cub + \m{O}\big(\abs{x}^{1/3}|\cub|\big)$ on $[t_0, 0) \times [-1, 1]$.
\end{lemma}
\begin{remark}
  The rectangle $[t_0, 0) \times [-1, 1]$ is simply a convenient bounded domain about $\{x = 0\}$.
\end{remark}
\begin{proof}
  Recall \eqref{eq:hodograph-domain} and $\uo{0} = u^0$.
  By construction, $\uo{0} \colon \m{A} \to \m{B}$ and $\uo{0}(t, \anon)$ has inverse $\omega(t, \anon)$.
  We initially focus on $(t, x) \in \m{A}$.
  
  For $y\in\m{B}$, \eqref{eq:inverse-power-series} implies that we can bound $\omega(t, y)$ from above and below by multiples of $ty - \coeff y^3$.
  It follows that
  \begin{equation}
    \label{eq:cub-compare}
    \uo{0} \asymp \cub \quad \text{on } \m{A}.
  \end{equation}
  Now, $\abs{\cub}$ is increasing in $t$, so
  \begin{equation}
    \label{eq:1/3-bd}
    |\uo{0}(t, x)| \lesssim \abs{\cub(t, x)} \leq \abs{\cub(0, x)} \asymp \abs{x}^{1/3}.
  \end{equation}

  Because $\uo{0}$ and $\cub$ are spatial bijections near $x=0$, there exists a unique shift $h(t, x)$ such that
  \begin{equation}
    \label{eq:distortion}
    \cub(t, x + h(t, x)) = \uo{0}(t, x)
  \end{equation}
  for sufficiently small $x$.
  That is, for $\abs{x} \leq c_0$ for some $c_0 > 0$.
  We will control this ``distortion'' $h$.
  In the following calculations, we assume $(t, x) \in \m{A} \cap \{\abs{x} \leq c_0\}$ unless stated otherwise.

  By \eqref{eq:inverse-power-series},
  \begin{equation}
    \label{eq:power-quartic}
    \abs{\omega(t, y) - (t y - \coeff y^3)} \lesssim \abs{y}^4 \ForAll y \in \m{B}.
  \end{equation}
  Since $x = \omega(t, \uo{0}(t, x))$, \eqref{eq:1/3-bd} and \eqref{eq:power-quartic} yield
  \begin{equation}
    \label{eq:diff-4/3}
    \abs{x - \big[t \uo{0}(t, x) - \coeff \uo{0}(t, x)^3\big]} \lesssim \uo{0}(t, x)^4 \lesssim  \abs{x}^{4/3}.
  \end{equation}
  Now Definition~\ref{def:cubic}, \eqref{eq:distortion}, and \eqref{eq:diff-4/3} imply
  \begin{equation}
    \label{eq:small-distortion}
    \abs{h} = \abs{x - (x + h)} = \abs{x - \big[t \cub(t, x + h) - \coeff \cub(t, x + h)^3\big]} \lesssim \abs{x}^{4/3}.
  \end{equation}
  By reducing $c_0$, we can assume that $\abs{h} \leq \abs{x}/2$ when $\abs{x} \leq c_0.$
  Then $\abs{z} \asymp \abs{x}$ for all $z$ between $x$ and $x + h$.
  By Lemma~\ref{lem:cubic},
  \begin{equation*}
    \abs{\partial_x \cub(t, z)} = \cubder(t, z) \leq \frac{\abs{\cub(t, z)}}{\abs{z}} \lesssim \frac{\cub(t, x)}{\abs{x}}
  \end{equation*}
  for such $z$.
  Hence by the mean value theorem and \eqref{eq:small-distortion},
  \begin{equation*}
    \abs{\cub(t, x + h) - \cub(t, x)} \lesssim \frac{\abs{\cub(t, x)}}{\abs{x}} \abs{h} \lesssim \abs{x}^{1/3} \abs{\cub(t, x)}.
  \end{equation*}
  Now \eqref{eq:distortion} yields
  \begin{equation}
    \label{eq:sing-IFT}
    \abs{\uo{0} - \cub} \lesssim \abs{x}^{1/3}|\cub| \quad \text{on } \m{A} \cap \{\abs{x} \leq c_0\}.
  \end{equation}
  Finally, $\uo{0} - \cub$ is bounded on the rectangle $[t_0, 0) \times [-1, 1]$.
  Since $\abs{x}^{1/3} \abs{\cub} \gtrsim 1$ when $\abs{x} \geq c_0,$ \eqref{eq:sing-IFT} holds on the full domain $[t_0, 0) \times [-1, 1]$.
\end{proof}
This lemma quickly yields stronger results.
\begin{corollary}
  \label{cor:u01}
  We have $\uo{0} = \cub + \beta_4 \cub^4 \cubder + \m{O}\big(\abs{\cub}^5 \cubder\big)$ on $[t_0, 0) \times [-1, 1]$.
\end{corollary}
\begin{proof}
  Let $v \coloneqq \uo{0} - \cub$.
  Plugging $\uo{0}$ into \eqref{eq:inverse-power-series} and using \eqref{eq:cub-compare}, we find
  \begin{equation*}
    x - \big[t \uo{0} - \coeff (\uo{0})^3\big] \lesssim \cub^4.
  \end{equation*}
  Replacing $\uo{0}$ by $\cub + v$ and using Definition~\ref{def:cubic}, this implies that
  \begin{equation}
    \label{eq:cubic-expand}
    \left[\abs{t} + \coeff \big(3\cub^2 + 3 \cub v + v^2\big)\right]v \lesssim \cub^4.
  \end{equation}
  By Lemma~\ref{lem:sing-IFT}, we can assume that $\abs{v} \leq \abs{\cub}/2$ provided $\abs{x} \leq c_0$ for some $c_0 > 0$.
  Hence
  \begin{equation*}
    \cub^2 \leq 3\cub^2 + 3 \cub v + v^2 \leq 5 \cub^2.
  \end{equation*}
  Then Lemma~\ref{lem:cubic} yields
  \begin{equation*}
    \abs{t} + \coeff \big(3\cub^2 + 3 \cub v + v^2\big) \asymp \cubder^{-1}.
  \end{equation*}
  Rearranging \eqref{eq:cubic-expand}, we obtain
  \begin{equation}
    \label{eq:diff-quad}
    v \lesssim \cub^4\cubder \lesssim \cub^2.
  \end{equation}
  Keeping one more term in \eqref{eq:inverse-power-series} and following the same reasoning, we find
  \begin{equation*}
    \left[\abs{t} + \coeff (3\cub^2 + 3\cub v + v^3)\right]v = \beta_4 (\cub + v)^4 + \m{O}\big(\abs{\cub}^5\big).
  \end{equation*}
  Finally, \eqref{eq:diff-quad} allows us to simplify:
  \begin{equation*}
    \cubder^{-1} v = \big(\abs{t} + 3 \coeff \cub^2\big) v = \beta_4 \cub^4 + \m{O}\big(\abs{\cub}^5\big).
  \end{equation*}
  The corollary follows for $\abs{x} \leq c_0$.
  Again, because $\abs{\cub} \gtrsim 1$ when $\abs{x} \geq c_0$, the bound extends to the full rectangle $[t_0, 0) \times [-1, 1]$.
\end{proof}
We can iterate this argument to show that $\uo{0}$ admits an expansion about $x = 0$ of the form
\begin{equation*}
  \uo{0} \sim \sum_{\ell = 0}^\infty u_{0, \ell}
\end{equation*}
for certain $(\ell + 1)$-homogeneous functions $u_{0,\ell}$ that are polynomials in $\cub$ and $\cubder$.
For instance, $u_{0, 0} = \cub$ and $u_{0, 1} = \beta_4 \cub^4 \cubder$.
Using identities like
\begin{equation*}
  \partial_x \uo{0}(t, x) = \partial_y \omega(t, \uo{0}(t, x))^{-1},
\end{equation*}
we can in turn derive analogous expansions for the derivatives of $\uo{0}$.

We denote the partial homogeneous sum by
\begin{equation*}
  u_{0, [\ell]} \coloneqq \sum_{\ell' = 0}^\ell u_{0, \ell'}.
\end{equation*}
Then variations on the proof of Corollary~\ref{cor:u01} yield:
\begin{lemma}
  \label{lem:u0-out-homog}
  For each $\ell \geq 1$, we have $u_{0, \ell} \in \m{P}(\ell + 3, 1)$.
  Moreover, for each $i, \ell \in \Z_{\geq 0}$,
  \begin{equation*}
    \abs{\partial_x^i\big(\uo{0} - u_{0, [\ell]}\big)} \lesssim_{i, \ell} \err(\ell - i + 4, i + 1) \quad \text{in } [t_0, 0) \times [-1, 1].
  \end{equation*}
\end{lemma}
In particular, we obtain direct control on $\uo{0}$:
\begin{corollary}
  \label{cor:u0-out-der}
  For each $i\in \Z_{\geq 0}$, we have $|\partial_x^i\uo{0}| \lesssim_i \err(- i +1, i)$ in $[t_0, 0) \times \R.$
\end{corollary}
\begin{proof}
  We use Lemma~\ref{lem:u0-out-homog} with $\ell = 0$, noting that $u_{0, [0]} = u_{0, 0 } = \cub$:
  \begin{equation*}
    \abs{\partial_x^i \uo{0}} \lesssim_{i} \abs{\partial_x^i \cub} + \err(\ell - i + 4, i + 1) \quad \text{in } [t_0, 0) \times [-1, 1].
  \end{equation*}
  By \eqref{eq:deriv-map}, $\partial_x^i \cub \in \m{P}(-i + 1, i)$.
  Thus by \eqref{eq:err-bd}, $|\partial_x^i \cub| \lesssim_i \err(-i + 1, i)$.
  Also,
  \begin{equation*}
    \err(- i + 4, i + 1) \lesssim \err(-i + 1, i)
  \end{equation*}
  in $[t_0, 0) \times [-1, 1]$, so
  \begin{equation}
    \label{eq:u0-out-der}
    \abs{\partial_x^i \uo{0}} \lesssim_{i} \err(-i + 1, i) \quad \text{in } [t_0, 0) \times [-1, 1].
  \end{equation}
  By Lemma~\ref{lem:compact}, $\uo{0}$ is constant outside a compact set.
  Hence when $i \geq 1$, $\partial_x^i \uo{0}$ is compactly supported and \eqref{eq:u0-out-der} holds in the entire space $[t_0, 0) \times \R$.
  When $i = 0$, $\err(1, 0) = \abs{\cub}$ grows at infinity, so again \eqref{eq:u0-out-der} holds in the entire space.
\end{proof}
We occasionally require a more precise bound on $\partial_x \uo{0}$.
\begin{corollary}
  \label{cor:u0-out-x}
  For all $t \in [t_0, 0)$,
  \begin{equation}
    \label{eq:u00-x}
    \norm{\partial_x \cub(t, \anon)}_{L^\infty} = \frac{1}{\abs{t}}.
  \end{equation}
  Moreover,
  \begin{equation}
    \label{eq:u0-out-x}
    \norm{\partial_x\uo{0}(t, \anon)}_{L^\infty} = \frac{1}{\abs{t}}+\m{O}\left(\abs{t}^{-1/2}\right).
  \end{equation}
\end{corollary}
\begin{proof}
  The first bound \eqref{eq:u00-x} follows from
  \begin{equation*}
    \abs{\partial_x \cub(t, x)} = \cubder(t, x) = \big[\abs{t} + 3\coeff \cub^2(t, x)\big]^{-1}
  \end{equation*}
  and $\cub(t, 0) = 0$.
  Applying Lemma~\ref{lem:u0-out-homog} with $i = 1$ and $\ell = 0$, we find
  \begin{equation}
    \label{eq:u0-out-precise}
    \abs{\partial_x \uo{0} - \partial_x \cub} \lesssim \err(3, 2) \lesssim \dist^{-1} \lesssim \abs{t}^{-1/2} \quad \text{in } [t_0, 0) \times [-1,1].
  \end{equation}
  Since $\uo{0}$ is smooth away from the origin, this bound actually holds on the whole space.
  Then \eqref{eq:u0-out-x} follows from \eqref{eq:u00-x}.
\end{proof}
\subsection{First viscous corrector}
Next, we consider the leading viscous corrector $\uo{1}$, which solves
\begin{equation}
  \label{eq:u1-out}
  \partial_t \uo{1} = -\partial_x\big(\uo{0}\uo{1}\big) + \partial_x^2 \uo{0}, \quad \uo{1}(t_0, \anon) = 0.
\end{equation}
We claim that $\uo{1}$ also admits a homogeneous expansion
\begin{equation}
  \label{eq:u1-out-expansion}
  \uo{1} \sim \sum_{\ell = 0}^\infty u_{1, \ell}
\end{equation}
about $x =0 $ with $(\ell - 3)$-homogeneous terms $u_{1, \ell}$.
To explain this degree of homogeneity, we note that $u_{1,\ell}$ roughly corresponds to $\int \!\partial_x^2 u_{0,\ell} \ds t$.
Each spatial derivative reduces homogeneity by $3$, while the time integral increases it by $2$.
Since $u_{0,\ell}$ is $(\ell + 1)$-homogeneous, it is natural to expect $u_{1,\ell}$ to be homogeneous of degree $(\ell + 1) - 3 - 3 + 2 = \ell - 3$.

The justification of \eqref{eq:u1-out-expansion} is somewhat more involved, as we do not have an implicit equation for $\uo{1}$.
Instead, we must proceed from the PDE \eqref{eq:u1-out}.
If we formally plug \eqref{eq:u1-out-expansion} into \eqref{eq:u1-out} and collect terms of like homogeneity, we arrive at a putative PDE for $u_{1, \ell}$:
\begin{equation}
  \label{eq:u-ell}
  \partial_t u_{1, \ell} = - \partial_x(\cub u_{1, \ell}) - \sum_{\ell' = 1}^{\ell} \partial_x(u_{0, \ell'} \, u_{1, \ell - \ell'}) + \partial_x^2 u_{0, \ell}.
\end{equation}
This equation has the form
\begin{equation}
  \label{eq:homog-PDE}
  \partial_t u_{1, \ell} = - \partial_x(\cub u_{1, \ell}) + F_\ell
\end{equation}
for an $(\ell-5)$-homogeneous forcing $F_\ell$ determined by the expansion for $\uo{0}$ and by earlier terms $\{u_{1,\ell'}\}_{\ell' < \ell}$.
We will prove the following by induction:
\begin{lemma}
  \label{lem:forcing-poly}
  For each $\ell \in \Z_{\geq 0},$ we have $F_\ell \in \m{P}(\ell + 1, 3)$.
\end{lemma}
Before proving the lemma, we discuss our approach to \eqref{eq:homog-PDE}.
The key is a change of coordinates from $(t, x)$ to $(t, \cub)$.
We refer to the latter as ``dynamic coordinates.''
We denote their partial derivatives by $\bar{\partial}_t$ and $\bar{\partial}_{\cub}$.
Because $\cub$ solves \eqref{eq:Burgers-inviscid}, the chain rule yields
\begin{equation*}
  \partial_t = (\partial_tt) \bar{\partial}_t + (\partial_t \cub) \bar{\partial}_\cub = \bar{\partial}_t - (\cub \partial_x \cub) \bar{\partial}_\cub \And \partial_x = (\partial_x \cub) \bar{\partial}_\cub.
\end{equation*}
Rearranging, we see that
\begin{equation}
  \label{eq:dynamic-deriv}
  \bar{\partial}_t = \partial_t + \cub \partial_x \And \bar{\partial}_\cub = -\cubder^{-1} \partial_x.
\end{equation}
Thus \eqref{eq:homog-PDE} becomes
\begin{equation}
  \label{eq:ODE-PDE}
  \bar\partial_t u_{1, \ell} = \cubder u_{1, \ell} + F_\ell.
\end{equation}
Recall that $\cubder^{-1} = -t + 3 \coeff \cub^2$.
It follows that $\bar\partial_t(\cubder^{-1}) = -1$.
Dividing \eqref{eq:ODE-PDE} by $\cubder$ and rearranging, we find
\begin{equation*}
  \bar\partial_t (\cubder^{-1} u_{1, \ell}) = \cubder^{-1} F_\ell.
\end{equation*}
We integrate this in $t$, holding $\cub$ constant.
By Lemma~\ref{lem:forcing-poly}, $\cubder^{-1} F_\ell \in \m{P}(\ell + 1, 2)$.
Crucially, $\cubder^2 \sim t^{-2}$ is integrable as $t \to -\infty$.
We can compute
\begin{equation*}
  \int_{-\infty}^t \cubder(s, \cub)^{a} \d s = \frac{1}{a - 1} \cubder(t, \cub)^{a - 1}.
\end{equation*}
In light of Lemma~\ref{lem:forcing-poly}, we obtain a special solution of \eqref{eq:homog-PDE}:
\begin{equation}
  \label{eq:special}
  u_{1, \ell}^F \coloneqq \cubder \int_{-\infty}^t \cubder(s, \cub)^{-1} F_\ell(s, \cub) \d s \in \m{P}(\ell + 1, 2).
\end{equation}

We emphasize that the three powers of $\cubder$ in $F_\ell \in \m{P}(\ell + 1, 3)$ are essential.
The above argument breaks down if a term in $F_\ell$ only includes two powers of $\cubder$, due to a logarithmic divergence in the integral.
For this reason, we keep careful track of powers of $\cubder$ throughout the section.

We have made progress, but we have so far ignored the initial data $\uo{1}(t_0, \anon) = 0$.
We wish to construct a solution of \eqref{eq:homog-PDE} that also reflects this condition.
Any two solutions of \eqref{eq:homog-PDE} differ by a solution $v$ of the \emph{unforced} equation
\begin{equation}
  \label{eq:unforced}
  \partial_t v = -\partial_x(\cub v).
\end{equation}
In the dynamic coordinates $(t, \cub)$, this becomes
\begin{equation*}
  \bar\partial_t v = \cubder v.
\end{equation*}
Arguing as above, we find the general solution: $v = f(\cub) \cubder$ for $f$ smooth.
The $(\ell-3)$-homogeneous solutions are all multiples of
\begin{equation*}
  v_\ell \coloneqq \cub^{\ell - 1} \cubder.
\end{equation*}
Thus all homogeneous solutions of \eqref{eq:homog-PDE} have the form
\begin{equation}
  \label{eq:u1-structure}
  u_{1, \ell} = u_{1, \ell}^F + C_\ell v_\ell \in \m{P}(\ell - 1, 1).
\end{equation}
We can now prove Lemma~\ref{lem:forcing-poly}.
\begin{proof}[Proof of Lemma~\textup{\ref{lem:forcing-poly}}]
  We proceed by induction.
  In the base case $\ell = 0$, we have
  \begin{equation*}
    F_0 = \partial_x^2 \cub = -6 \coeff \cub \cubder^3 \in \m{P}(1, 3)
  \end{equation*}
  as claimed.
  Next, fix $\ell \in \N$ and suppose $F_{\ell'} \in \m{P}(\ell' + 1, 3)$ for all $\ell' < \ell$.
  Recall that
  \begin{equation}
    \label{eq:forcing-structure}
    F_\ell = -\sum_{\ell' = 1}^\ell \partial_x(u_{0, \ell'} u_{1, \ell - \ell'}) + \partial_x^2 u_{0, \ell}.
  \end{equation}
  We must therefore control the structure of $u_{0, \ell'}$ and $u_{1, \ell- \ell'}$.

  Since $\ell' \geq 1$, Lemma~\ref{lem:u0-out-homog} yields $u_{0, \ell'} \in \m{P}(\ell' + 3, 1)$.
  Next, we construct $u_{1,\ell-\ell'}$ using the method described above, so
  \begin{equation*}
    u_{1,\ell-\ell'} = \cubder \int_{-\infty}^t \cubder^{-1} F_{\ell-\ell'} \d s + C v_{\ell-\ell'}
  \end{equation*}
  for some $C \in \R$.
  Using the inductive hypothesis that $F_{\ell-\ell'} \in \m{P}(\ell - \ell' + 1, 3)$, our calculations above show that $u_{1,\ell-\ell'} \in \m{P}(\ell - \ell' - 1, 1)$, as in \eqref{eq:u1-structure}.
  Combining these observations about $u_{0,\ell'}$ and $u_{1,\ell- \ell'}$, \eqref{eq:deriv-map} implies that $\partial_x(u_{0, \ell'} u_{1, \ell - \ell'}) \in \m{P}(\ell + 1, 3)$.
  Similarly, \eqref{eq:deriv-map} yields $\partial_x^2 u_{0, \ell} \in \m{P}(\ell + 1, 3).$
  By \eqref{eq:forcing-structure}, we have $F_\ell \in \m{P}(\ell + 1, 3)$, as desired.
  The lemma now follows from induction.
\end{proof}
To uniquely determine $u_{1,\ell}$, we must choose the free parameter $C_\ell$ in \eqref{eq:u1-structure}.
When $\ell = 0$, the unforced solution $v_\ell = \cub^{-1} \cubder$ is not smooth on $\R_- \times \R$; it is singular at $x = 0$.
Our initial data is smooth, so the expansion \eqref{eq:u1-out-expansion} should not contain such a term.
Thus smoothness demands the choice $C_0 = 0$, and we let $u_{1, 0} = u_{1, 0}^F$.
This is the \emph{unique} smooth homogeneous solution of
\begin{equation*}
  \partial_t u_{1, 0} = - \partial_x(\cub u_{1, 0}) + F_0 = - \partial_x(\cub u_{1, 0}) + \partial_x^2 \cub.
\end{equation*}

In contrast, $v_\ell$ is smooth when $\ell \geq 1$.
Thus \eqref{eq:u-ell} admits a $1$-parameter family of smooth homogeneous solutions.
To select one, we try to match the initial data for $\uo{1}$ near $x = 0$ as closely as possible.
Since $\uo{1}(t_0, \anon) = 0$, we want the formal homogeneous sum \eqref{eq:u1-out-expansion} to vanish to all orders in $x$ at $(t_0, 0)$.
We choose $C_\ell$ in \eqref{eq:u1-structure} to arrange this.

At fixed $t < 0$, we can explicitly compute
\begin{equation}
  \label{eq:cubic-x-0}
  \cub(t, x) \sim x t^{-1} \And m \to \abs{t}^{-1} \quad \text{as } x \to 0.
\end{equation}
So
\begin{equation}
  \label{eq:vl-asymp}
  v_\ell(t_0, x) = \cub(t_0, x)^{\ell - 1} \cubder(t_0, x) \sim (-1)^{\ell - 1} \abs{t_0}^{-\ell} x^{\ell - 1}  \quad \text{as } x \to 0.
\end{equation}
We wish to approximate $\uo{1}$ by $\sum_{\ell' \geq 0} u_{1, \ell'}$ when $\abs{x} \ll 1$, so we choose the coefficient $C_\ell$ of $v_\ell$ to cancel the $x^{\ell - 1}$ term in the Taylor series of $\sum_{\ell' \geq 0} u_{1, \ell'}(t_0, \anon)$ around $x = 0$.
Define the operator
\begin{equation*}
  T_\ell[f] \coloneqq \frac{1}{(\ell - 1)!} \partial_x^{\ell - 1}f(t_0, 0).
\end{equation*}
Then $T_\ell$ extracts the coefficient of $x^{\ell - 1}$ in the Taylor series of $f$ at time $t_0$.
By \eqref{eq:special} and \eqref{eq:vl-asymp}, $u_{1,\ell}^F$ is divisible by $\cub^{\ell + 1}$ and $\cub^{\ell + 1} \asymp_t x^{\ell + 1}$ near $x = 0$.
It follows that $T_\ell[u_{1,\ell}^F] = 0$.
On the other hand, \eqref{eq:vl-asymp} implies that $T_{\ell}[v_\ell] = (-1)^{\ell-1}\abs{t_0}^{\ell}$.

Applying $T_\ell$ to the full homogeneous expansion, we find
\begin{equation*}
  T_\ell\left[\sum_{\ell' \geq 0} u_{1, \ell'}\right] = C_\ell (-1)^{\ell-1}\abs{t_0}^{\ell} + T_\ell \left[\sum_{0 \leq \ell' < \ell} u_{1, \ell'}\right].
\end{equation*}
We inductively choose
\begin{equation*}
  C_\ell = (-1)^{\ell} \abs{t_0}^{-\ell} T_\ell \left[\sum_{0 \leq \ell' < \ell} u_{1, \ell'}\right]
\end{equation*}
so that
\begin{equation*}
  T_\ell\left[\sum_{\ell' \geq 0} u_{1, \ell'}\right] = 0.
\end{equation*}
Then we have fixed $C_\ell$, and thus $u_{1, \ell}$, for all $\ell \in \Z_{\geq 0}$.
Given $\ell \in \Z_{\geq 0}$, define the partial sum
\begin{equation*}
  u_{1,[\ell]} \coloneqq \sum_{\ell' = 0}^\ell u_{1, \ell'}.
\end{equation*}
Our construction ensures that
\begin{equation}
  \label{eq:initial-match}
  \abs{u_{1,[\ell]}(t_0, x)} \lesssim_\ell \abs{x}^{\ell}.
\end{equation}
We will show that $u_{1, [\ell]}$ approximates $\uo{1}$ near $x = 0$:
\begin{lemma}
  \label{lem:u1-homogeneous}
  For each $i, \ell \in \Z_{\geq 0}$, $u_{1, \ell} \in \m{P}(\ell - 1, 1)$ and
  \begin{equation}
    \label{eq:u1-homogeneous}
    \abs{\partial_x^i(\uo{1} - u_{1, [\ell]})} \lesssim_{i, \ell} \err(\ell - i, i + 1) \quad \text{in } [t_0, 0) \times [-1, 1].
  \end{equation}
\end{lemma}

\begin{proof}
  Fix $\ell\in \Z_{\geq 0}$.
  The difference $w \coloneqq \uo{1} - u_{1, [\ell]}$ solves
  \begin{equation}
    \label{eq:diff-ev}
    \partial_t w = - \partial_x\left[\uo{0}\left(w + u_{1, [\ell]}\right)\right] + \partial_x^2 \uo{0} - \partial_t u_{1, [\ell]}.
  \end{equation}
  Using \eqref{eq:u-ell}, we write
  \begin{equation*}
    \partial_t u_{1, [\ell]} = -\sum_{0 \leq \ell' + \ell'' \leq \ell} \partial_x(u_{0,\ell'} u_{1, \ell''}) + \partial_x^2 u_{0, [\ell]}.
  \end{equation*}
  Rearranging \eqref{eq:diff-ev}, we obtain
  \begin{equation}
    \label{eq:diff-ev-simple}
    \partial_t w = -\partial_x(\uo{0} w) + G_\ell
  \end{equation}
  for
  \begin{equation*}
    G_\ell \coloneqq - \partial_x\left[(\uo{0} - u_{0, [\ell]}) u_{1, [\ell]}\right] - \sum_{\substack{1 \leq \ell', \ell'' \leq \ell\\ \ell' + \ell'' > \ell}} \partial_x(u_{0,\ell'} u_{1, \ell''}) + \partial_x^2(\uo{0} - u_{0, [\ell]}).
  \end{equation*}
  Recall from \eqref{eq:u1-structure} that $u_{1, \ell'} \in \m{P}(\ell' - 1, 1)$.
  Using \eqref{eq:deriv-map}, \eqref{eq:err-bd}, and Lemma~\ref{lem:u0-out-homog}, a calculation shows that
  \begin{equation}
    \label{eq:G}
    \abs{\partial_x^i G_\ell} \lesssim_{i, \ell} \err(\ell - i + 2, i + 3) \ForAll i \in \Z_{\geq 0}.
  \end{equation}
  
  We now prove \eqref{eq:u1-homogeneous} by induction.
  Fix $i \in \Z_{\geq 0}$ and suppose we have shown \eqref{eq:u1-homogeneous} for all $i' < i$, noting that this assumption is vacuous in the base case $i = 0$.
  Differentiating \eqref{eq:diff-ev-simple}, the derivative $z \coloneqq \partial_x^i w$ satisfies
  \begin{equation}
    \label{eq:diff-deriv-ev}
    \partial_t z = -\uo{0} \partial_x z - (i + 1)(\partial_x \uo{0}) z + G_{i,\ell}
  \end{equation}
  for
  \begin{equation*}
    G_{i, \ell} \coloneqq - \sum_{i'=2}^{i + 1} \binom{i + 1}{i'}\partial_x^{i'}\uo{0} \partial_x^{i+1-i'}w + \partial_x^i G_\ell.
  \end{equation*}
  Note that $i + 1 - i' < i$ in the sum, so $G_{i, \ell}$ only involves lower derivatives of $w$.
  Using \eqref{eq:u1-homogeneous} for $i + 1 - i'$ (the inductive hypothesis) as well as Corollary~\ref{cor:u0-out-der} and \eqref{eq:G}, we see that
  \begin{equation}
    \label{eq:G-deriv}
    \abs{G_{i, \ell}} \lesssim_{i, \ell} \err(\ell - i + 2, i + 3).
  \end{equation}
  To integrate the advection equation \eqref{eq:diff-deriv-ev}, we introduce dynamic coordinates $(t, \uo{0})$.
  These are valid on $\m{A}$ defined in \eqref{eq:hodograph-domain}.
  Let $\ti{\partial}_t$ denote partial differentiation in $t$ with $\uo{0}$ held constant.
  Noting that $\partial_x \uo{0} < 0$ on $\m{A}$, \eqref{eq:diff-deriv-ev} becomes
  \begin{equation}
    \label{eq:deriv-dynamic}
    \ti\partial_t z = (i + 1)\abs{\partial_x\uo{0}} z + G_{i, \ell}.
  \end{equation}
  
  Before proceeding, we relate $\uo{0}$ to $\cub$.
  By Lemma~\ref{lem:u0-out-homog},
  \begin{equation*}
    \abs{\uo{0} - \cub} \lesssim \cub^4 \cubder \lesssim \cub^2.
  \end{equation*}
  Moreover, because $\partial_x \uo{0} < 0$ on $\m{A}$, $\partial_x \cub < 0$, and $\uo{0} = \cub = 0$ at $x = 0$, we have
  \begin{equation*}
    \op{sgn}\uo{0} = \op{sgn} \cub = \op{sgn} x \quad \text{on } \m{A}.
  \end{equation*}
  It follows that
  \begin{equation}
    \label{eq:solution-equiv}
    \uo{0} \asymp \cub \quad \text{on } \m{A}.
  \end{equation}
  
  We show a similar relationship for the spatial derivative.
  Let ${m_0 \coloneqq -\partial_x \uo{0} > 0}$ on $\m{A}$.
  Recall from Section~\ref{subsec:hodograph} that $\omega(t,\uo{0}(t, x)) = x$.
  Differentiating, we find $m_0 \partial_y \omega = 1.$
  Now, the Taylor expansion \eqref{eq:hodograph-Taylor} for $\omega$ is valid in $\m{C}^1$, so we can differentiate it to obtain
  \begin{equation*}
    m_0(t, \uo{0})^{-1} = \abs{t} + 3 \coeff (\uo{0})^2 + \m{O}\big(|\uo{0}|^3\big).
  \end{equation*}
  When $\uo{0}$ is small, Lemma~\ref{lem:cubic} and \eqref{eq:solution-equiv} yield
  \begin{equation}
    \label{eq:slope-solution}
    m_0(t, \uo{0}) \asymp \left[\abs{t} + (\uo{0})^2\right]^{-1} \asymp \left[\abs{t} + \cub^2\right]^{-1} \asymp \cubder.
  \end{equation}
  On the other hand, wherever $\uo{0} \gtrsim 1$ on $\m{A}$, both $m_0$ and $\cubder$ are bounded away from $0$ and $\infty$, so again $m_0 \asymp \cubder$.
  We therefore have
  \begin{equation}
    \label{eq:dynamic-equiv}
    \uo{0} \asymp \cub \And m_0 \asymp \cubder \quad \text{on } \m{A}.
  \end{equation}

  We now tackle \eqref{eq:deriv-dynamic}.
  Because $\uo{0}$ satisfies inviscid Burgers, we can compute $\ti\partial_t m_0 = m_0^2.$
  Hence
  \begin{equation*}
    \ti\partial_t m_0^{-(i + 1)} = -(i + 1) m_0^{-i}.
  \end{equation*}
  We can thus use $m_0^{-(i + 1)}$ as an integrating factor in \eqref{eq:deriv-dynamic}.
  We obtain a special solution:
  \begin{equation}
    \label{eq:deriv-special}
    z^F \coloneqq m_0^{i + 1} \int_{t_0}^t m_0^{-(i + 1)} G_{i, \ell} \d s.
  \end{equation}
  We emphasize that $\uo{0}$ is held constant in this integral.
  Recalling \eqref{eq:poly-err}, we write
  \begin{equation*}
    \err(\ell - i + 2, i + 3) = \abs{\cub}^\alpha \cubder^\beta
  \end{equation*}
  for integers $\alpha \geq 0$ and $\beta \geq i + 3$ depending on $i$ and $\ell$.
  Then \eqref{eq:G-deriv}, \eqref{eq:slope-solution}, and \eqref{eq:dynamic-equiv} yield
  \begin{equation*}
    \big|m_0^{-(i + 1)} G_{i, \ell}\big| \lesssim_{i,\ell} \abs{\uo{0}}^\al \left[\abs{t} + (\uo{0})^2\right]^{-(\beta-i-1)}.
  \end{equation*}
  We integrate the right side in $t$ while holding $\uo{0}$ fixed.
  Because $\beta - i - 1 \geq 2$, we do not pick up a logarithm in $t$; \eqref{eq:deriv-special} implies:
  \begin{equation*}
    \abs{z^F} \lesssim_{i,\ell} \abs{\uo{0}}^\al m_0^{i + 1} \left[\abs{t} + (\uo{0})^2\right]^{-(\beta-i-2)}.
  \end{equation*}
  Using \eqref{eq:poly-err}, \eqref{eq:slope-solution} and \eqref{eq:dynamic-equiv} again, this becomes
  \begin{equation}
    \label{eq:deriv-special-bd}
    \abs{z^F} \lesssim_{i,\ell} \abs{\uo{0}}^\al m_0^{\beta-1} \lesssim_{i,\ell} \abs{\cub}^\al \cubder^{\beta - 1} = \err(\ell - i + 2, i + 2).
  \end{equation}
  
  Next, we incorporate the initial data
  \begin{equation*}
    z(t_0, \anon) = \partial_x^i w(t_0, \anon) = -\partial_x^i u_{1,[\ell]}(t_0,\anon).
  \end{equation*}
  Here we have used $\uo{1}(t_0, \anon) = 0$.
  Recall the inverse $\mr{\omega}$ of $\mr{u}$ and define
  \begin{equation}
    \label{eq:unforced-multiple}
    f(y) \coloneqq -\abs{\partial_y\mr{\omega}(y)}^{i + 1} \left[z^F\big(t_0, \mr{\omega}(y)\big) + \partial_x^i u_{1,[\ell]}\big(t_0, \mr{\omega}(y)\big)\right] \For y \in \m{B}.
  \end{equation}
  By the chain rule, $\big[(\partial_y\mr{\omega})\circ\mr{u}\big]^{-1} = \partial_x\mr{u} = m_0(t_0,\anon)$.
  Hence we can substitute
  \begin{equation*}
    y = \mr{u}(x) = \uo{0}(t_0, x)
  \end{equation*}
  in \eqref{eq:unforced-multiple} to obtain
  \begin{equation*}
    f\big(\uo{0}(t_0, x)\big) \abs{m_0(t_0, x)}^{i + 1} = -z^F(t_0, x) - \partial_x^i u_{1, [\ell]}(t_0, x) \ForAll x \in \m{A}_{t_0}.
  \end{equation*}
  We then define
  \begin{equation}
    \label{eq:unforced-def}
    v \coloneqq f(\uo{0}) m_0^{i + 1},
  \end{equation}
  which solves the unforced equation
  \begin{equation}
    \label{eq:unforced-ev}
    \ti{\partial}_t v = (i + 1) m_0 v, \quad v(t_0, \anon) = -z^F(t_0, \anon) - \partial_x^i u_{1, [\ell]}(t_0, \anon).
  \end{equation}

  Now, \eqref{eq:cubic-x-0} and \eqref{eq:deriv-special-bd} imply that
  \begin{equation*}
    \abs{z^F(t_0, x)} \lesssim_{i, \ell}  \err(\ell - i + 2, i + 2)(t_0, x) \lesssim_{i, \ell} \abs{x}^{(\ell - i + 2) \vee 0}.
  \end{equation*}
  Also, \eqref{eq:deriv-map} and \eqref{eq:initial-match} yield
  \begin{equation*}
    \abs{\partial_x^i u_{1,[\ell]}(t_0, x)} \lesssim_{i, \ell} \abs{x}^{(\ell - i) \vee 0} \For x \in \m{A}_{t_0}.
  \end{equation*}
  Therefore,
  \begin{equation}
    \label{eq:forcing-x-0}
    \abs{z^F + \partial_x^i u_{1, [\ell]}} \lesssim_{i,\ell}\abs{x}^{(\ell - i) \vee 0}
  \end{equation}
  when $t = t_0$.
  Now, $\partial_y\mr{\omega} \asymp 1$ on $\m{B}$.
  Thus \eqref{eq:unforced-multiple} and \eqref{eq:forcing-x-0} yield
  \begin{equation*}
    \abs{f(y)} \lesssim_{i,\ell} \abs{y}^{(\ell - i)\vee 0}.
  \end{equation*}
  Hence \eqref{eq:dynamic-equiv} and \eqref{eq:unforced-def} imply that
  \begin{equation}
    \label{eq:deriv-unforced-bd}
    \abs{v} \lesssim_{i,\ell} \err\big((\ell-i) \vee 0, i + 1\big).
  \end{equation}
  Recall that $z^F$ solves \eqref{eq:deriv-dynamic}.
  In light of \eqref{eq:unforced-ev}, we see that both $z$ and $z^F + v$ solve \eqref{eq:deriv-dynamic} with initial data $-\partial_x^iu_{1,[\ell]}(t_0, \anon)$.
  By uniqueness, we have $z = z^F + v$.
  Finally, \eqref{eq:deriv-special-bd}, \eqref{eq:deriv-unforced-bd}, and $\cub^2 \cubder \lesssim 1$ yield
  \begin{equation*}
    \abs{\partial_x^i w} = \abs{z} \lesssim_{i, \ell} \err(\ell - i + 2, i + 2) + \err\big((\ell-i) \vee 0, i + 1\big) \lesssim_{i, \ell} \err(\ell - i, i + 1) \quad \text{on } \m{A}.
  \end{equation*}
  Because $\uo{1}$ and $u_{1,[\ell]}$ are smooth away from the origin in spacetime, this bound also holds in $[t_0, 0) \times [-1, 1]$.
  This proves \eqref{eq:u1-homogeneous} for $i$.
  The lemma follows from induction.
\end{proof}
\begin{remark}
  There is another possible approach to Lemma~\ref{lem:u1-homogeneous} due to the following curious observation: $-\partial_x^2\uo{0}\big(\partial_x \uo{0}\big)^{-2}$ solves the PDE in \eqref{eq:u1-out}.
  It does not satisfy the initial condition, but this can be fixed by adding an appropriate solution $f(\uo{0}) \partial_x \uo{0}$ of the unforced equation $\partial_t v = -\partial_x(\uo{0}v)$.
  We proceed much as in \eqref{eq:unforced-multiple}.
  Recalling the inverse $\mr{\omega}$ of $\mr{u}$, let
  \begin{equation*}
    f \coloneqq \frac{(\partial_x^2 \mr{u})\circ \mr{\omega}}{(\partial_x \mr{u})^3\circ \mr{\omega}}.
  \end{equation*}
  Then
  \begin{equation*}
    \uo{1} = -\frac{\partial_x^2\uo{0}}{(\partial_x \uo{0})^2} + f(\uo{0}) \partial_x \uo{0}
  \end{equation*}
  satisfies \eqref{eq:u1-out}.
  Using Lemma~\ref{lem:u0-out-homog}, we can construct the homogeneous expansion for $\uo{1}$ and prove Lemma~\ref{lem:u1-homogeneous} by elementary series expansion.
  We believe that this exact formula is an artifact of the Cole--Hopf transformation.
  More general scalar conservation laws do not seem to admit such identities, so we adopt a more robust strategy above.
\end{remark}

\subsection{Higher viscous correctors}
To close, we consider the higher-order correctors $\uo{k}$ for $k \geq 2$.
Each satisfies \eqref{eq:ukout-ev}, which we write as
\begin{equation*}
  \partial_t \uo{k} = -\frac{1}{2}\sum_{k' + k'' = k} \partial_x(\uo{k'}\uo{k''}) + \partial_x^2 \uo{k - 1}.
\end{equation*}
We expect each $\uo{k}$ to admit a homogeneous expansion of the form
\begin{equation*}
  \uo{k} \sim \sum_{\ell = 0}^\infty u_{k, \ell}
\end{equation*}
about $x = 0$.
After formal substitution, we can check that the terms in this series should satisfy
\begin{equation}
  \label{eq:ukl-ev}
  \partial_t u_{k, \ell} = - \frac{1}{2}\sum_{\substack{k' + k'' = k\\ \ell' + \ell'' = \ell}} \partial_x(u_{k', \ell'} u_{k'', \ell''}) + \partial_x^2 u_{k-1, \ell}.
\end{equation}
Structurally, this is very similar to the equation for $u_{1, \ell}$.
Define $u_{k, [\ell]} \coloneqq \sum_{\ell' = 0}^\ell u_{k, \ell'}$.
\begin{proposition}
  \label{prop:uk-homogeneous}
  For all $i, k, \ell \in \Z_{\geq 0},$ $u_{k, \ell} \in \m{P}(- 4k + \ell + 3, 1)$ and
  \begin{equation*}
    \abs{\partial_x^i(\uo{k} - u_{k, [\ell]})} \lesssim_{i,k,\ell} \err(-4k + \ell - i + 4, i + 1) \quad \text{in } [t_0, 0) \times [-1, 1].
  \end{equation*}
\end{proposition}
\begin{proof}
  The proof of Lemma~\ref{lem:u1-homogeneous} readily adapts to this setting and the proposition follows from induction on $k$.
  We omit the repeated details.
\end{proof}
In particular, we can control the size of $\uo{k}$.
\begin{corollary}
  \label{cor:uk-der}
  For all $i, k \in \Z_{\geq 0}$, $|\partial_x^i(\uo{k})| \lesssim_{i,k} \dist^{-3i-4k+1}.$
\end{corollary}
\begin{proof}
  This bound is a consequence of \eqref{eq:deriv-map}, \eqref{eq:dist-bd}, Proposition~\ref{prop:uk-homogeneous}, and the triangle inequality.
  The argument is similar to the proof of Corollary~\ref{cor:u0-out-der}; we omit the details.
\end{proof}
\begin{remark}
  We work with smooth data for convenience, but Proposition~\ref{prop:uk-homogeneous} is true, up to a point, under bounded regularity.
  Precisely, one can check that the proposition holds for the values $i,k,\ell \in \Z_{\geq 0}$ when $\mr{u} \in \m{C}^{i + 2k + \ell + 4}$.
  Thus, the construction of the outer expansion demands ever greater regularity as $k$ and $\ell$ increase.
  The dependence on $\ell$ is clear, since $\ell$ corresponds to the Taylor series of $\mr{u}$.
  The regularity also depends on $k$ because $\uo{k}$ is forced by $\partial_x^2 \uo{k-1}$.
  We thus need two additional derivatives as $k$ increments.
  This is not surprising: we are attempting to approximate a second-order equation by first-order equations, so we lose derivatives.
\end{remark}
\begin{remark}
  \label{rem:indep}
  For simplicity, we have assumed that $u^\nu(t_0, \anon) = \mr{u}$ is independent of $\nu$.
  However, we could naturally allow $u^\nu(t_0, \anon)$ to vary smoothly in $\nu$.
  Then $\uo{k}$ would have nontrivial initial data corresponding to the $\nu^k$ term in $u^\nu(t_0, \anon)$ expanded in $\nu$ about $\nu = 0$.
  This data influences the structure of the homogeneous components $u_{k,\ell}$.

  Indeed, we can adapt the construction of $u_{1,\ell}$ to form $u_{k, \ell}$.
  Let $u_{k,\ell}^F$ denote the particular solution of \eqref{eq:ukl-ev} constructed as in \eqref{eq:special}.
  Recalling \eqref{eq:u1-structure}, we can write
  \begin{equation*}
    u_{k,\ell} = u_{k,\ell}^F + C_{k,\ell} \cub^{-4k + \ell + 3}\cubder.
  \end{equation*}
  Here $\cub^{-4k + \ell + 3}\cubder$ solves the unforced problem \eqref{eq:unforced} and $C_{k,\ell} \in \R$ is chosen so that $u_{k,\ell}(t_0, \anon)$ agrees with the initial data $\uo{k}(t_0, \anon)$ to spatial order $-4k + \ell + 3$ at $x = 0$.
  
  This selection principle for $C_{k,\ell}$ is nonsensical when $-4k + \ell + 3 < 0$, since the initial data is smooth.
  For such indices, we must therefore take $C_{k,\ell} = 0$ and $u_{k,\ell} = u_{k,\ell}^F$.
  Thus $u_{k,\ell}$ is independent of the initial data if and only if $-4k + \ell + 3 < 0$, i.e., if $\ell \leq 4(k - 1)$.
  These terms are \emph{universal}, depending only on $\coeff$.
  We have already observed this phenomenon when $k = 1$ and $\ell = 0$.
  As we shall see, the distinction between universal and data-dependent homogeneous components has consequences for the uniqueness of terms in the inner expansion.
\end{remark}

\section{The inner expansion}\label{sec:inner}
We now turn to the inner expansion.
Recall the inner variables
\begin{equation*}
  T \coloneqq \nu^{-1/2} t \And  X \coloneqq \nu^{-3/4} x
\end{equation*}
defined in \eqref{eq:inner-var}.
This is simply a $\lambda$-rescaling $(T, X) \mapsto (\lambda^2 T, \lambda^3 X) = (t, x)$ with $\lambda = \nu^{1/4}$.
Thus if $f$ is $\ell$-homogeneous, we have
\begin{equation*}
  f(t, x) = \nu^{\ell/4} f(T, X).
\end{equation*}
In particular, we make use of the following rescaled homogeneous functions:
\begin{align}
  \Cub(T, X) &\coloneqq \nu^{-1/4} \cub(\nu^{1/2} T, \nu^{3/4}X),\label{eq:Cub}\\
  \Cubder(T, X) &\coloneqq \nu^{1/2} \cubder(\nu^{1/2} T, \nu^{3/4}X),\nonumber\\
  \Dist(T, X) &\coloneqq \nu^{-1/4} \dist(\nu^{1/2} T, \nu^{3/4}X).\label{eq:Dist}
\end{align}
These functions are simply $\cub$, $\cubder$, and $\dist$ evaluated at $(T, X)$ rather than $(t, x)$.
We employ distinct notation to indicate that these functions are implicitly evaluated in the inner coordinates.

As noted in Section~\ref{sec:overview}, the blow-up $U^\nu(T, X) \coloneqq \nu^{-1/4} u^\nu(\nu^{1/2} T, \nu^{3/4}X)$ solves viscous Burgers \eqref{eq:viscous-Burgers-unit} with \emph{unit} viscosity in the blown-up domain $(\nu^{-1/2}t_0, 0) \times \R$.
Recalling \eqref{eq:inner-sum}, we expect $U^\nu$ to admit an expansion of the form
\begin{equation*}
  U^\nu \sim \sum_{\ell = 0}^\infty \nu^{\ell/4} \Ui{\ell} \quad \text{ when } \nu \ll 1 \text{ and } \Dist \ll \nu^{-1/4}.
\end{equation*}
Each term $\Ui{\ell}$ is defined on $\R_- \times \R$.
Plugging this sum into \eqref{eq:viscous-Burgers-unit}, we formally obtain the evolution equations \eqref{eq:Ul-ev-overview} for $\Ui{\ell}$.
For convenience, we restate them here:
\begin{equation}
  \label{eq:Ul-ev}
  \partial_T \Ui{\ell} = -\frac{1}{2}\sum_{\ell' + \ell'' = \ell} \partial_X(\Ui{\ell'} \Ui{\ell''}) + \partial_X^2 \Ui{\ell}.
\end{equation}
Naturally, the leading term satisfies viscous Burgers:
\begin{equation}
  \label{eq:U0-ev}
  \partial_T \Ui{0} = - \Ui{0} \partial_X \Ui{0} + \partial_X^2 \Ui{0}.
\end{equation}
When $\ell \geq 1,$ \eqref{eq:Ul-ev} is linear and can be written as a passive scalar equation with diffusion:
\begin{equation*}
  \partial_T \Ui{\ell} = -\partial_X(\Ui{0}\Ui{\ell})-\sum_{\ell' = 1}^{\ell - 1} \Ui{\ell'} \partial_X\Ui{\ell - \ell'} + \partial_X^2 \Ui{\ell}.
\end{equation*}
We emphasize that the sum only involves $\Ui{\ell'}$ for $0 < \ell' < \ell$, so it is simply a forcing term.

Each $\Ui{\ell}$ is an eternal solution of its respective equation, so it does not have ``initial data'' per se.
Rather, we choose the far-field behavior of $\Ui{\ell}$ to match the outer expansion \eqref{eq:outer-sum}.
We do so through the homogeneous components $u_{k, \ell}$ in \eqref{eq:outer-homog-expansion}.
Recall that $u_{k, \ell}$ is $(\ell - 4k + 1)$-homogeneous.
Thus in the inner coordinates, the term $\nu^k u_{k, \ell}$ in the outer expansion becomes
\begin{equation*}
  \nu^k u_{k, \ell}(\nu^{1/2}T, \nu^{3/4}X) = \nu^{k + (\ell - 4k + 1)/4}u_{k, \ell}(T, X) = \nu^{(\ell + 1)/4} u_{k, \ell}(T, X).
\end{equation*}
On the other hand,
\begin{equation*}
  u^\nu(\nu^{1/2}T, \nu^{3/4}X) = \nu^{1/4} U^\nu(T, X) \sim \sum_{\ell \geq 0} \nu^{(\ell + 1)/4} \Ui{\ell}(T, X).
\end{equation*}
Hence $u_{k, \ell}$ should contribute to $\Ui{\ell}$.
To keep the inner notation separate, we define
\begin{equation*}
  U_{k, \ell}(T, X) \coloneqq \nu^{k - (\ell + 1)/4} u_{k, \ell}(\nu^{1/2}T, \nu^{3/4}X).
\end{equation*}
Then $\Ui{\ell}$ should somehow incorporate the homogeneous components $\{U_{k,\ell}\}_{k \geq 0}$.
For fixed $k \geq 0$, $U_{k + 1, \ell}$ is $(\ell - 4k - 3)$-homogeneous, so it decays faster than $U_{k, \ell}$ at infinity.
We therefore conjecture that $\Ui{\ell}$ admits an asymptotic expansion about infinity of the form
\begin{equation}
  \label{eq:inner-at-infinity}
  \Ui{\ell} \sim \sum_{k = 0}^\infty U_{k, \ell} \quad \text{as } \Dist \to \infty \text{ in } \R_- \times \R.
\end{equation}
We introduce the corresponding partial sum
\begin{equation*}
  U_{[k], \ell} = \sum_{k'=0}^k U_{k', \ell}.
\end{equation*}
Using \eqref{eq:ukl-ev}, we can check that $U_{[k],\ell}$ satisfies
\begin{equation}
  \label{eq:Ukl-sum-ev}
  \partial_T U_{[k],\ell} = -\frac{1}{2} \sum_{\substack{k' + k'' \leq k \\ \ell' + \ell'' = \ell}} \partial_X(U_{k',\ell'}U_{k'',\ell''}) + \partial_X^2 U_{[k-1],\ell}
\end{equation}
under the convention that $U_{[-1],\ell} = 0$ when $k = 0$.

To prove the expansion \eqref{eq:inner-at-infinity}, we argue that $U_{[k],\ell}$ is a good approximate solution of \eqref{eq:Ul-ev} at infinity.
More concretely, we study the evolution equation for the difference $\Ui{\ell} - U_{[k],\ell}$.
We construct super- and subsolutions of this equation that are appropriately small, and conclude that there exists a solution $\Ui{\ell}$ of \eqref{eq:Ul-ev} close to $U_{[k],\ell}$ at infinity.

Throughout, we find it technically convenient to construct one pair $\pm W_{\text{early}}$ of super- and subsolutions on the deep past $\{T < T_{\text{early}}\}$, and another pair $\pm W_{\text{late}}$ on $\{T_{\text{early}} \leq T < 0\}$, for some $T_{\text{early}} \ll -1$.
The early pair $\pm W_{\text{early}}$ yields the solution $\Ui{\ell}$ through diagonalization.
However, our construction of $W_{\text{early}}$ involves a Riccati equation that blows up at finite negative time.
Before this blow-up, we hand our estimates off to $\pm W_{\text{late}}$, which ensure that $\Ui{\ell}$ behaves as expected on the final bounded time interval.

The evolution equations for the differences $\Ui{\ell} - U_{[k],\ell}$ have a common form.
We can therefore construct our super- and subsolutions with one lemma.
In the following, let $\m{Q} \coloneqq (-1, 0) \times (-1, 1) \subset \R_- \times \R$ denote a rectangle of spacetime near the origin and let $\partial^* \m{Q} \coloneqq [-1, 0) \times \{\pm 1\}$ denote the portion of the boundary that parabolically influences the complement $\m{Q}^c \coloneqq (\R_- \times \R) \setminus \m{Q}$.
\begin{lemma}
  \label{lem:inner-super-sub}
  Let $i,k, \ell \in \Z_{\geq 0}$ satisfy $3i + 4k \geq \ell$ and fix $\const > 0$.
  In the following, we refer to functions $b,c,G \colon \R_- \times \R \to \R$ and $g,h\colon \R \to \R$.
  \begin{enumerate}[label = \textup{(\roman*)}, itemsep = 2pt, leftmargin = 18pt]
  \item
    \label{item:inner-super-sub-early}
    Consider the equation
    \begin{equation}
      \label{eq:inner-super-sub-early}
      \begin{aligned}
        \hspace{10pt}\partial_T V = -\Cub \partial_X V - (i + 1) (\partial_X \Cub) V + \big[b + &g(V)\big]\partial_X V\\
        &+ \big[c + h(V)\big]V + \partial_X^2 V + G.
      \end{aligned}
    \end{equation}
    Suppose
    \begin{equation}
      \begin{gathered}
        \abs{b} \leq \const \Dist^{-3}, \quad \abs{c} \leq \const \Dist^{-6}, \quad \abs{G} \leq \const \Dist^{-(3i + 4k - \ell + 5)},\\
        \text{and } \quad \abs{g(s)} + \abs{h(s)} \leq \const \abs{s}.
      \end{gathered}
      \label{eq:early-cond}
    \end{equation}
    Then there exist constants $C_{\mathrm{early}} > 0$ and $T_{\mathrm{early}} < 0$ as well as a smooth function $W_{\mathrm{early}} \colon \{T < T_{\mathrm{early}}\} \to \R_+$ such that $W_{\mathrm{early}}$ (resp. $-W_{\mathrm{early}}$) is a supersolution (resp. subsolution) of \eqref{eq:inner-super-sub-early} and
    \begin{equation*}
      W_{\mathrm{early}} \leq C_{\mathrm{early}} \Dist^{-(3i + 4k - \ell + 3)}.
    \end{equation*}
    The constants $C_{\mathrm{early}}$ and $T_{\mathrm{early}}$ and the function $W_{\mathrm{early}}$ depend on $i, k, \ell$, and $\const$.

  \item
    \label{item:inner-super-sub-late}
    Consider the equation
    \begin{equation}
      \label{eq:inner-super-sub-late}
      \partial_T V = -\Cub \partial_X V - (i + 1) (\partial_X \Cub) V + b\partial_X V + cV + \partial_X^2 V + G.
    \end{equation}
    Suppose
    \begin{equation}
      \label{eq:late-cond}
      \abs{b} + \abs{c} \leq \const \And \abs{G} \leq \const \Dist^{-(3i + 4k - \ell + 5)} \quad \text{in } \{-\const < T < 0\} \setminus \m{Q}.
    \end{equation}
    Then there exist $C_{\mathrm{late}} > 0$ and smooth $W_{\mathrm{late}} \colon \{-\const < T < 0\} \setminus \m{Q} \to \R_+$ such that $W_{\mathrm{late}}$ (resp. $-W_{\mathrm{late}}$) is a supersolution (resp. subsolution) of \eqref{eq:inner-super-sub-late},
    \begin{equation*}
      W_{\mathrm{late}} > \const \Dist^{-(3i + 4k - \ell + 3)} \quad \text{on } \{T = -\const\} \cup \partial^* \m{Q},
    \end{equation*}
    and
    \begin{equation*}
      W_{\mathrm{late}} \leq C_{\mathrm{late}} \Dist^{-(3i + 4k - \ell + 3)} \quad \text{on } \{-\const < T < 0\} \setminus \m{Q}.
    \end{equation*}
    Both $C_{\mathrm{late}}$ and $W_{\mathrm{late}}$ depend on $i, k, \ell$, and $\const$.
  \end{enumerate}
\end{lemma}
Before proving this lemma, we use it to construct and control $\Ui{\ell}$.
\begin{proposition}
  \label{prop:Ul}
  For each $\ell \in \Z_{\geq 0}$, there exists a unique solution $\Ui{\ell}$ of \eqref{eq:Ul-ev} on $\R^2$ such that for all $i,k\in \Z_{\geq 0}$,
  \begin{equation}
    \label{eq:Ul-est}
    \abs{\partial_X^i (\Ui{\ell} - U_{[k], \ell})} \lesssim_{i,k,\ell} \Dist^{-(3i+4k-\ell+3)} \quad \text{on } \m{Q}^c.
  \end{equation}
  Moreover, $\Ui{\ell}$ is the unique solution of \eqref{eq:Ul-ev} satisfying \eqref{eq:Ul-est} for \emph{some} $k \geq \ell/4$ and each $i \in \{0, 1\}$.
\end{proposition}
\begin{remark}
  The leading term $\Ui{0}$ plays a particularly important role.
  We discuss its structure in more detail in Appendix~\ref{sec:inner-term}.
\end{remark}
\begin{remark}
  According to Proposition~\ref{prop:Ul}, once \eqref{eq:Ul-est} holds for some $k_0 \geq \ell/4$, it holds for \emph{every} $k \in \Z_{\geq 0}$.
  That is, if $\Ui{\ell}$ agrees with the homogeneous expansion \eqref{eq:inner-homog-expansion} to a certain finite order, it automatically agrees to every order.
  This indicates that the homogeneous series is rigid: if the first $\ceil{\ell/4}$ terms of \eqref{eq:inner-homog-expansion} are set, the higher-order terms are pre-determined.
  
  Remark~\ref{rem:indep} explains this rigidity.
  There, we observed that $u_{k,\ell}$, and hence $U_{k,\ell}$, is independent of the initial data when $k \geq \ell/4 + 1$.
  Such terms are \emph{universal}.
  Suppose \eqref{eq:Ul-est} holds for some $k_0 \geq \ell/4$.
  The homogeneous components of \eqref{eq:inner-homog-expansion} not represented in \eqref{eq:Ul-est} have the form $u_{k,\ell}$ for $k \geq k_0 + 1 \geq \ell/4 + 1$.
  Thus every homogeneous component ``omitted'' from \eqref{eq:Ul-est} is universal.
  It follows that the homogeneous expansion \eqref{eq:inner-homog-expansion} is indeed pre-determined past order $\ceil{\ell/4}$.
  
  Reasoning in the opposite direction, we see that the uniqueness condition in Proposition~\ref{prop:Ul} is sharp.
  If $k_0 < \ell /4$, there exists $k_0 < k < \ell/4 + 1$ such that $u_{k,\ell}$ is \emph{not} universal.
  Different choices of the initial data $u^\nu(t_0, \anon)$ can lead to different homogeneous components $u_{k,\ell}$.
  Thus there exist homogeneous expansions \eqref{eq:inner-homog-expansion} that agree up to order $k_0$ but differ thereafter.
  Each expansion corresponds to a distinct solution $\Ui{\ell}$ of \eqref{eq:Ul-ev}, so the estimate \eqref{eq:Ul-est} with $k_0 < \ell/4$ does \emph{not} ensure the uniqueness of $\Ui{\ell}$.
\end{remark}
\begin{corollary}
  \label{cor:Ul-size}
  For each $i,\ell \in \Z_{\geq 0}$, we have $\abs{\partial_X^i \Ui{\ell}} \lesssim_{i,\ell} \Dist^{-3i + \ell + 1}$ on $\m{Q}^c$ and $\abs{\partial_X^i \Ui{\ell}} \lesssim_{i,\ell} 1$ on $\m{Q}$.
\end{corollary}
\begin{proof}
  Fix $i,\ell \in \Z_{\geq 0}$ and recall that $U_{[0],\ell} = U_{0,\ell}$.
  Because $\partial_X^i U_{0,\ell}$ is $(\ell + 1 - 3i)$-homogeneous, we have $\abs{\partial_X^i U_{[0],\ell}} \lesssim_{i,\ell} \Dist^{-3i + \ell + 1}$ on $\R_- \times \R$.
  On $\m{Q}^c$, this dominates the error in \eqref{eq:Ul-est} with $k = 0$.
  Thus the first part of the corollary follows from Proposition~\ref{prop:Ul} and the triangle inequality.
  Then the parabolic Harnack inequality yields $\sup_{\m{Q}} |\partial_X^i \Ui{\ell}| \lesssim_{i,\ell}.$
\end{proof}
\begin{proof}[Proof of Proposition~\textup{\ref{prop:Ul}}]
  We begin by recording some helpful estimates on the approximate solutions $U_{[k], \ell}$.
  Recall that $U_{[k],\ell} = \sum_{0\leq k'\leq k}U_{k',\ell}$ and $U_{k',\ell}$ is bounded away from the origin and $(-4k'+\ell+1)$-homogeneous.
  Each spatial derivative of a homogeneous function reduces its homogeneity by $3$.
  Adapting \eqref{eq:dist-bd}, each $r$-homogeneous function that is bounded away from the origin is bounded by a multiple of $\Dist^r$.
  From these observations, we can readily deduce the following bounds on $\m{Q}^c$:
  \begin{equation}
    \label{eq:Ukl-est}
    \begin{aligned}
      \abs{\partial_X^i U_{k,\ell}} &\lesssim_{i,k,\ell} \Dist^{-(3i + 4k - \ell - 1)},\\
      \abs{\partial_X^i U_{[k],\ell}} &\lesssim_{i,k,\ell} \Dist^{-(3i - \ell - 1)},\\
      \abs{\partial_X^i (U_{[k],0} - \Cub)} &\lesssim_{i,k} \Dist^{-(3i + 3)}.
    \end{aligned}
  \end{equation}

  We now make a series of interlocking inductive arguments to prove the existence portion of the proposition for all $i,k,\ell \in \Z_{\geq 0}$.
  In broad strokes, we fix $k$ and induct twice: first on $i$, then on $\ell$.
  Throughout, we take some care to distinguish between nonlinear and linear evolution equations.
  In particular, we treat the $\ell = 0$ and $\ell \geq 1$ cases separately.
  \medskip

  First fix $k \geq 0$ and consider $\ell = 0$.
  Then \eqref{eq:Ul-ev} is \eqref{eq:U0-ev}, i.e., viscous Burgers.
  Formally suppose $U$ solves \eqref{eq:U0-ev}.
  We show that $V \coloneqq U - U_{[k],0}$ solves an equation of the form \eqref{eq:inner-super-sub-early}.
  Using the super- and subsolutions from Lemma~\ref{lem:inner-super-sub}, we employ a diagonalization argument to \emph{construct} a particular solution $V^*$ of this equation that is suitably small.
  Then $\Ui{0} \coloneqq U_{[k],0} +  V^*$ will be the desired solution of \eqref{eq:U0-ev}.
  
  Recall that $U_{[k],0}$ solves \eqref{eq:Ukl-sum-ev}, which we write as
  \begin{equation*}
    \partial_T U_{[k],0} = -\frac{1}{2} \partial_X\big(U_{[k],0}^2\big) + \partial_X^2 U_{[k],0} - G_{0,0}
  \end{equation*}
  for
  \begin{equation}
    \label{eq:G0-def}
    G_{0,0} \coloneqq -\frac{1}{2} \sum_{\substack{k',k'' \leq k \\ k' + k'' > k}}\partial_X(U_{k',0}U_{k'',0}) + \partial_X^2 U_{k,0}.
  \end{equation}
  Then if $U_0$ solves \eqref{eq:U0-ev}, $V_0 \coloneqq U - U_{[k], 0}$ solves
  \begin{equation}
    \label{eq:U0-diff-ev-prelim}
    \partial_T V_0 = -\partial_X (U_{[k],0} V_0) - \frac{1}{2} \partial_X\big(V_0^2\big) + \partial_X^2 V_0 + G_{0,0}.
  \end{equation}
  Now, for each $n \in \N$, let $V_0^{(n)}$ denote the solution of \eqref{eq:U0-diff-ev-prelim} on $(-n, 0) \times \R$ with initial data
  \begin{equation*}
    V_0^{(n)}(-n, \anon) = 0.
  \end{equation*}
  We inductively prove the following claim:
  \begin{claim}
    \label{claim:U0-early}
    For all $i \in \Z_{\geq 0}$, there exist constants $C_{i,0}^{\mathrm{early}} > 0$ and $T_{i,0}^{\mathrm{early}} < 0$ such that
    \begin{equation}
      \label{eq:U0-deriv-est}
      \abs{\partial_X^{i} V_0^{(n)}(T, X)} \leq C_{i,0}^{\mathrm{early}} \Dist(T, X)^{-(3i+4k+3)}
    \end{equation}
    for all $n \in \N$, $T \in \big(-n, T_{i,0}^{\mathrm{early}} \big)$, and $X \in \R$.
  \end{claim}
  \begin{remark}
    This estimate is vacuous when $n \leq \big|T_{i,0}^{\mathrm{early}}\big|$.
    We elide this point below.
  \end{remark}
  \noindent
  \claimproof{claim:U0-early}We begin with the $i = 0$ case.
  Recall that $U_{[k],0} \sim U_{0,0} = \Cub$ at infinity.
  Extracting this leading order, we write \eqref{eq:U0-diff-ev-prelim} as
  \begin{equation}
    \label{eq:U0-diff-ev}
    \begin{aligned}
      \partial_T V_0 = -\Cub\partial_XV_0 - (\partial_X\Cub) V_0 - (U_{[k],0} - &\Cub + V_0) \partial_X V_0\\
      &- \partial_X(U_{[k],0} - \Cub)V_0 + \partial_X^2 V_0 + G_{0,0}.
    \end{aligned}
  \end{equation}
  Using \eqref{eq:Ukl-est}, we can check that
  \begin{align*}
    \abs{U_{[k],0} - \Cub} &\lesssim_k \Dist^{-3},\\
    \abs{\partial_X(U_{[k],0} - \Cub)} &\lesssim_k \Dist^{-6},\\
    \abs{G_{0,0}} &\lesssim_k \Dist^{-(4k+5)}.
  \end{align*}
  Therefore \eqref{eq:U0-diff-ev} satisfies the hypotheses \eqref{eq:early-cond} in Lemma~\ref{lem:inner-super-sub}\ref{item:inner-super-sub-early} with
  \begin{gather*}
    b = -(U_{[k],0} - \Cub), \quad c = -\partial_X(U_{[k],0} - \Cub), \quad G = G_{0,0}, \quad g(s) = -s, \quad h = 0.
  \end{gather*}
  The lemma yields a signed super/subsolution $\pm W_{0,0}^{\mathrm{early}}$ of \eqref{eq:U0-diff-ev} such that
  \begin{equation*}
    W_{0,0}^{\mathrm{early}} \leq C_{0,0}^{\mathrm{early}} \Dist^{-(4k+3)} \quad \text{on } \big\{T < T_{0,0}^{\mathrm{early}}\big\}
  \end{equation*}
  for some $C_{0,0}^{\mathrm{early}} > 0$ and $T_{0,0}^{\mathrm{early}} < 0$.
  Because $W_{0,0}^{\mathrm{early}} > 0$,
  \begin{equation*}
    -W_{0,0}^{\mathrm{early}}(-n, \anon) < 0 = V_0^{(n)}(-n, \anon) < W_{0,0}^{\mathrm{early}}(-n, \anon).
  \end{equation*}
  Thus by the comparison principle,
  \begin{equation*}
    \big|V_0^{(n)}(T, X)\big| \leq W_{0,0}^{\mathrm{early}}(T, X) \leq C_{0,0}^{\mathrm{early}} \Dist^{-(4k+3)}(T, X)
  \end{equation*}
  for all $n\in \N$ and $(T, X) \in \big(-n, T_{0,0}^{\mathrm{early}}\big) \times \R$.
  This establishes Claim~\ref{claim:U0-early} for $i = 0$.

  Next, consider $i \geq 1$ and suppose we have verified the claim for all $i' < i$.
  We show it for $i$ as well.
  Differentiating \eqref{eq:U0-diff-ev-prelim}, $Z \coloneqq \partial_X^i V_0^{(n)}$ satisfies
  \begin{equation}
    \label{eq:U0-deriv-ev-prelim}
    \begin{aligned}
      \partial_T Z = -U_{[k],0}\partial_X Z - (i + 1) (\partial_XU_{[k],0})Z - V_0^{(n)}\partial_X Z - (i+&1)(\partial_X V_0^{(n)})Z\\
      &+ \partial_X^2 Z + G_{i,0}.
    \end{aligned}
  \end{equation}
  for
  \begin{align*}
    G_{i,0} \coloneqq - \sum_{i' = 0}^{i-1} \binom{i + 1}{i'} \big(\partial_X^{i-i'+1} U_{[k],0}\big)\big(\partial_X^{i'}V_0^{(n)}\big) - \frac{1}{2}\sum_{i' = 2}^{i-1} \binom{i + 1}{i'} \big(\partial_X^{i-i'+1} V_0^{(n)}\big)&\big(\partial_X^{i'}V_0^{(n)}\big)\\
                                                                                                                                                                                                                                  &+ \partial_X^i G_{0,0}.
  \end{align*}
  The sums only involve derivatives of $V_0^{(n)}$ of order less than $i$; they can be controlled using \eqref{eq:U0-deriv-est} for $i' < i$.
  Combining \eqref{eq:Ukl-est}, \eqref{eq:G0-def}, and \eqref{eq:U0-deriv-est}, we can check that
  \begin{equation}
    \label{eq:G0-est}
    \abs{G_{i,0}} \lesssim_{i, k} \Dist^{-(3i + 4k + 5)}.
  \end{equation}
  Extracting the leading part of $U_{[k], 0}$, we can write \eqref{eq:U0-deriv-ev-prelim} as
  \begin{equation}
    \label{eq:U0-deriv-ev}
    \begin{aligned}
      \partial_T Z = -\Cub \partial_X Z - (i + 1) &(\partial_X \Cub) Z -\left(U_{[k],0} - \Cub + V_0^{(n)}\right)\partial_X Z \\
      &- (i+1)\partial_X\left(U_{[k],0} - \Cub + V_0^{(n)}\right) Z + \partial_X^2 Z + G_{i,0}.
    \end{aligned}
  \end{equation}

  We claim that this equation satisfies the assumptions \eqref{eq:early-cond} in Lemma~\ref{lem:inner-super-sub}\ref{item:inner-super-sub-early}.
  When $i = 1$, we have $\partial_X V_0^{(n)} = Z$, so
  \begin{equation}
    \label{eq:U0-diff-1-deriv-ev}
    \begin{aligned}
      \partial_T Z = -\Cub \partial_X Z - 2(\partial_X \Cub) Z &-\left(U_{[k],0} - \Cub + V_0^{(n)}\right)\partial_X Z \\
      &- 2\left[\partial_X\left(U_{[k],0} - \Cub\right) + Z\right] Z + \partial_X^2 Z + G_{1,0}.
    \end{aligned}
  \end{equation}
  Using \eqref{eq:Ukl-est}, Claim~\ref{claim:U0-early} for $i = 0$, and \eqref{eq:G0-est}, we have
  \begin{align*}
    \abs{U_{[k],0} - \Cub + V_0^{(n)}} &\lesssim_{k} \Dist^{-3},\\
    \abs{\partial_X\left(U_{[k],0} - \Cub\right)} &\lesssim_k \Dist^{-6},\\
    \abs{G_{1,0}} &\lesssim_k \Dist^{-(4k + 8)}.
  \end{align*}
  Thus \eqref{eq:U0-diff-1-deriv-ev} satisfies \eqref{eq:early-cond} with $b = -\big(U_{[k],0} - \Cub + V_0^{(n)}\big),$ $c = -2\partial_X\left(U_{[k],0} - \Cub\right)$, $G = G_{1,0}$, $g = 0$, and $h(s) = -2s$.
  As when $i = 0$, the $i = 1$ case of Claim~\ref{claim:U0-early} follows from Lemma~\ref{lem:inner-super-sub}\ref{item:inner-super-sub-early}.

  Finally, if $i \geq 2$, \eqref{eq:U0-deriv-ev} is linear in $Z$.
  Using \eqref{eq:Ukl-est} and Claim~\ref{claim:U0-early} for $i' < i$, we have
  \begin{align*}
    \abs{U_{[k],0} - \Cub + V_0^{(n)}} &\lesssim_{k} \Dist^{-3},\\
    \abs{\partial_X\left(U_{[k],0} - \Cub + V_0^{(n)}\right)} &\lesssim_k \Dist^{-6},\\
    \abs{G_{1,0}} &\lesssim_k \Dist^{-(3i + 4k + 5)}.
  \end{align*}
  Taking $b = -\left(U_{[k],0} - \Cub + V_0^{(n)}\right)$, $c = - (i+1)\partial_X\left(U_{[k],0} - \Cub + V_0^{(n)}\right)$, $G = G_{i,0}$, and $g = h = 0$, we see that \eqref{eq:U0-diff-1-deriv-ev} satisfies \eqref{eq:early-cond}.
  As above, Claim~\ref{claim:U0-early} for $i$ follows from Lemma~\ref{lem:inner-super-sub}\ref{item:inner-super-sub-early}.
  By induction on $i$, we obtain the full claim.
  \claimqed

  Next, we take $n \to \infty$ and extract a convergent subsequence.
  Without loss of generality, we may assume that the sequence $\big(T_{i,0}^{\mathrm{early}}\big)_{i \geq 0}$ is decreasing.
  Now, Claim~\ref{claim:U0-early} implies that the sequence $\big(V_0^{(n)}\big)_{n \in \N}$ is bounded in $\m{C}_T^2 \cap \m{C}_X^3$ on each compact $K \subset \big(-\infty, T_{3,k,0}^{\mathrm{early}}\big) \times \R$.
  By Arzel\`{a}--Ascoli, we can extract a subsequence that converges in $\m{C}_T^1(K) \cap \m{C}_X^2(K)$ as $n \to \infty$.
  It follows that the limit satisfies \eqref{eq:U0-diff-ev-prelim} in $K^\circ$.
  Diagonalizing along an exhaustion of $\big(-\infty, T_{3,k,0}^{\mathrm{early}}\big) \times \R$, we can produce a single subsequence that converges locally uniformly to a limit $V_0^*$ solving \eqref{eq:U0-diff-ev-prelim} on $\big(-\infty, T_{3,k,0}\big) \times \R$.
  Because each term in the subsequence satisfies \eqref{eq:U0-deriv-est}, the limit satisfies the following for all $i \in \Z_{\geq 0}$:
  \begin{equation}
    \label{eq:U0-est-early}
    \abs{\partial_X^i V_0^*} \leq C_{i,0}^{\mathrm{early}} \Dist^{-(3i + 4k + 3)} \quad \text{on } \big(-\infty, T_{i,0}^{\mathrm{early}}\big) \times \R.
  \end{equation}
  
  We define $\Ui{0} \coloneqq U_{[k],0} + V_0^*$, which solves viscous Burgers \eqref{eq:U0-ev}.
  We claim that viscous Burgers is globally well posed for initial data that grow slower than $\abs{X}$ at infinity.
  Indeed, under the Cole--Hopf transform, such functions correspond to data for the heat equation that grow slower than $\exp\big(\abs{X}^2\big)$.
  Using Lemma~\ref{lem:cubic}, \eqref{eq:Ukl-est}, \eqref{eq:U0-est-early}, and the triangle inequality, we can check that $\abs{\Ui{0}(T, X)} \lesssim_T \abs{X}^{1/3}$ for $T < T_{0,0}^{\mathrm{early}}$.
  Therefore $\Ui{\ell}$ has a smooth extension to $\R^2$.
  In particular, $\abs{\partial_X^i \Ui{0}} \lesssim_{i} 1$ on the precompact set $\m{Q}$.
  We also extend the definition of $V_0^* = \Ui{0} - U_{[k],0}$ to $\R_- \times \R$.
  We don't go further because $U_{[k],0}$ is singular at the origin.

  To complete the existence portion of Proposition~\ref{prop:Ul} in the case $\ell = 0$, we must extend the bounds \eqref{eq:U0-est-early} to the domain $\m{Q}^c = (\R_- \times \R) \setminus \big([-1, 0) \times [-1, 1]\big).$
  We wish to use super- and subsolutions from Lemma~\ref{lem:inner-super-sub}\ref{item:inner-super-sub-late}, but there is one difficulty: the coefficients $b$ and $c$ in this part of the lemma must be bounded.
  Because \eqref{eq:U0-diff-ev-prelim} includes the nonlinear advection $\frac{1}{2}\partial_X\big(V_0^*\big)=V_0^*\partial_X V_0^*$, we first show that $V_0^*$ is bounded.
  To do so, we compare $\Ui{0}$ with a smooth cubic.
  Let $\bar\Cub(T, X) \coloneqq \Cub(T - 1, X)$ denote a temporal shift of the cubic profile $\Cub$ and define $\bar{V}_0 \coloneqq \Ui{0} - \bar{\Cub}$.
  This modified difference satisfies
  \begin{equation}
    \label{eq:shift-diff}
    \partial_T \bar{V}_0 = -\partial_X\big(\bar{\Cub} \,\bar{V}_0\big) - \bar{V}_0 \partial_X \bar{V}_0 + \partial_X^2 \bar{V}_0 + \partial_X^2 \bar{\Cub}.
  \end{equation}
  It is easy to check that $\bar{\Cub} - U_{[k],0}$ is bounded at time $\bar{T} \coloneqq T_{3,k,0}^{\mathrm{early}}$, and \eqref{eq:U0-est-early} implies the same for $\Ui{0} - U_{[k],0}$.
  By the triangle inequality,
  \begin{equation*}
    \big\|\bar{V}_0(\bar{T}, \anon)\big\|_{\m{C}^0} = \big\|\big(\Ui{0} - U_{[k],0}\big)(\bar{T},\anon) - \big(\bar{\Cub} - U_{[k],0}\big)(\bar{T},\anon)\big\|_{\m{C}^0} < \infty.
  \end{equation*}
  Now $\partial_X \bar{\Cub}$ and $\partial_X^2 \bar{\Cub}$ are uniformly bounded on $\R_- \times \R$, so $C \e^{\lambda (T - \bar{T})}$ is a supersolution of \eqref{eq:shift-diff} provided $C$ and $\lambda$ are sufficiently large.
  Likewise, $-C \e^{\lambda (T - \bar{T})}$ is a subsolution.
  Taking $C \geq \big\|\bar{V}_0(\bar{T}, \anon)\big\|_{\m{C}^0}$, the comparison principle implies that
  \begin{equation*}
    \big\|\bar{V}_0\big\|_{\m{C}^0((\bar{T}, 0) \times \R)} < \infty.
  \end{equation*}
  Again using the boundedness of $\bar\Cub - U_{[k], 0}$ in $\m{Q}^c$, the triangle inequality allows us to conclude that
  \begin{equation*}
    \big\|V_0^*\big\|_{\m{C}^0(\m{Q}^c)} = \big\|\bar{V}_0  + U_{[k],0} - \bar\Cub\big\|_{\m{C}^0(\m{Q}^c)}\leq \const_{0,0}
  \end{equation*}
  for some $\const_{0,0} > 0$.
  Moreover, standard parabolic estimates for \eqref{eq:shift-diff} imply that $\partial_X^i \bar{V}_0$ is bounded for each $i \in \N$.
  Since the derivatives of $\bar\Cub - U_{[k], 0}$ are bounded on $\m{Q}^c$, we see that there exists $\const_{i,0} > 0$ such that
  \begin{equation}
    \label{eq:U0-diff-bounded}
    \abs{\partial_X^i V_0^*} \leq \const_{i,0} \quad \text{on } \m{Q}^c
  \end{equation}
  for each $i \in \Z_{\geq 0}$.

  With these a priori estimates, we can extend the bounds \eqref{eq:U0-est-early} up to time $T = 0$.
  \begin{claim}
    \label{claim:U0-est}
    For each $i \in \Z_{\geq 0}$, there exists a constant $C_{i,0} > 0$ such that
    \begin{equation}
      \label{eq:U0-est}
      \abs{\partial_X^{i} (\Ui{0} - U_{[k],0})} \leq C_{i,0} \Dist^{-(3i+4k+3)} \quad \text{on } \m{Q}^c.
    \end{equation}
  \end{claim}
  \claimproof{claim:U0-est}We again proceed by induction.
  Fix $i \in \Z_{\geq 0}$ and suppose the claim holds for all $i' < i$, noting that this assumption is vacuous in the base case $i = 0$.
  Repeating our calculation for $\partial_X^i V_0^{(n)}$, we see that $Z \coloneqq \partial_X^i V_0^*$ satisfies
  \begin{equation}
    \label{eq:U0-limit-deriv-ev}
    \begin{aligned}
      \partial_T Z = -\Cub \partial_X Z - (i + 1) (\partial_X &\Cub) Z - (U_{[k],0} - \Cub + V_0^*) \partial_X Z\\
      - &(i + 1) \partial_X(U_{[k],0} - \Cub + V_0^*) Z + \partial_X^2 Z + G_{i,0}^*
    \end{aligned}
  \end{equation}
  for
  \begin{align*}
    G_{i,0}^* \coloneqq - \sum_{i' = 0}^{i-1} \binom{i + 1}{i'} \big(\partial_X^{i-i'+1} U_{[k],0}\big)\big(\partial_X^{i'}V_0^{*}\big) - \frac{1}{2}\sum_{i' = 2}^{i-1} \binom{i + 1}{i'} \big(\partial_X^{i-i'+1} V_0^{*}\big)\big(&\partial_X^{i'}V_0^*\big)\\
                                                                                                                                                                                                                                     &+ \partial_X^i G_{0,0}.
  \end{align*}
  By \eqref{eq:Ukl-est}, \eqref{eq:G0-def}, and \eqref{eq:U0-est} for $i' < i$, we have $\big|G_{i,0}^*\big| \lesssim_{i,k} \Dist^{-(3i+4k+5)}$.
  Likewise, \eqref{eq:Ukl-est} and \eqref{eq:U0-diff-bounded} imply that
  \begin{equation*}
    b \coloneqq - (U_{[k],0} - \Cub + V_0^*) \And c \coloneqq - (i + 1) \partial_X(U_{[k],0} - \Cub + V_0^*)
  \end{equation*}
  are bounded on $\m{Q}^c.$
  Thus \eqref{eq:U0-limit-deriv-ev} satisfies \eqref{eq:late-cond} in Lemma~\ref{lem:inner-super-sub}\ref{item:inner-super-sub-late}.
  Furthermore, \eqref{eq:U0-est-early} and \eqref{eq:U0-diff-bounded} imply the existence of $\const > \big|T_{i,0}^{\mathrm{early}}\big|$ such that
  \begin{equation}
    \label{eq:border-bd}
    \abs{Z} \leq \const \Dist^{-(3i+4k+3)} \quad \text{on } \{T = -\const\} \cup \partial^*\m{Q}.
  \end{equation}

  By Lemma~\ref{lem:inner-super-sub}\ref{item:inner-super-sub-late}, there exists $C_{i,0}^{\mathrm{late}} > 0$ and a signed super/subsolution $\pm W_{i,0}^{\mathrm{late}}$ of \eqref{eq:U0-limit-deriv-ev} on $\{-\const < T < 0\} \setminus \m{Q}$ such that $W_{i,0}^{\mathrm{late}} > \const \Dist^{-(3i + 4k + 3)}$ on $\{T = -\const\} \cup \partial^*\m{Q}$ and $W_{i,0}^{\mathrm{late}} \leq C_{i,0}^{\mathrm{late}} \Dist^{-(3i + 4k + 3)}$.
  By \eqref{eq:border-bd}, $\abs{Z} \leq W_{i,0}^{\mathrm{late}}$ on the parabolic boundary of $\{-\const < T < 0\} \setminus \m{Q}$.
  By the comparison principle (see, e.g., \cite[Corollary~2.5]{Lieberman}),
  \begin{equation*}
    \abs{Z} \leq W_{i,0}^{\mathrm{late}} \leq C_{i,0}^{\mathrm{late}} \Dist^{-(3i + 4k + 3)} \quad \text{on } \{-\const < T < 0\} \setminus \m{Q}.
  \end{equation*}
  In combination with \eqref{eq:U0-est-early}, we obtain Claim~\ref{claim:U0-est} for $i$.
  The full claim now follows by induction.
  \claimqed

  We have constructed $\Ui{0}$ satisfying the bounds in Proposition~\ref{prop:Ul}.
  This concludes the existence proof for $\ell = 0$.
  \medskip

  Next, we tackle $\ell \geq 1$.
  \begin{claim}
    \label{claim:Vl}
    For each $\ell \in \Z_{\geq 0}$ and $k \geq \ell/4$, there exists a solution $\Ui{\ell}$ of \eqref{eq:Ul-ev} such that for each $i \in \Z_{\geq 0}$,
    \begin{equation}
      \label{eq:Vl-est}
      \abs{\partial_X^i(\Ui{\ell} - U_{[k],\ell})} \lesssim_{i,k,\ell} \Dist^{-(3i + 4k - \ell + 3)} \quad \text{on } \m{Q}^c.
    \end{equation}
  \end{claim}
  \claimproof{claim:Vl}
  We verified the case $\ell = 0$ above.
  Fix $\ell \in \N$ and $k \geq \ell/4$.
  Suppose we have verified Claim~\ref{claim:Vl} for all $0 \leq \ell' <\ell$.
  Because $k \geq \ell/4 \geq \ell'/4$, we obtain solutions $\Ui{\ell'}$ of \eqref{eq:Ul-ev} satisfying \eqref{eq:Vl-est} for each $i \in \Z_{\geq 0}$ and $0 \leq \ell' < \ell$.
  We wish to verify the claim for $\ell$.

  To begin, we write \eqref{eq:Ukl-sum-ev} as
  \begin{equation*}
    \partial_T U_{[k],\ell} = -\frac{1}{2} \sum_{\ell' + \ell'' = \ell} \partial_X(U_{[k],\ell'}U_{[k],\ell''}) + \partial_X^2 U_{[k],\ell} - G_{0,\ell}^{(1)}
  \end{equation*}
  with
  \begin{equation}
    \label{eq:Gl1}
    G_{0,\ell}^{(1)} \coloneqq -\frac{1}{2} \sum_{\substack{k',k'' \leq k \\ k' + k'' > k\\ \ell' + \ell'' = \ell}}\partial_X(U_{k',\ell'}U_{k'',\ell''}) + \partial_X^2 U_{k,\ell}.
  \end{equation}
  Let $V_{\ell'} \coloneqq \Ui{\ell'} - U_{[k],\ell'}$ for each $0 \leq \ell' < \ell$.
  Suppose we have a solution $U_\ell$ of \eqref{eq:Ul-ev} and define $V_\ell \coloneqq U_\ell - U_{[k], \ell}$.

  Using \eqref{eq:Ul-ev}, elementary manipulations yield
  \begin{equation*}
    \partial_T V_\ell = -\sum_{\ell' + \ell'' = \ell} \partial_X(U_{[k],\ell'} V_{\ell''}) - \frac{1}{2} \sum_{\ell' + \ell'' = \ell} \partial_X(V_{\ell'}V_{\ell''}) + \partial_X^2 V_{\ell} + G_{0,\ell}^{(1)}.
  \end{equation*}
  We group the advection terms as follows:
  \begin{equation*}
    -\sum_{\ell' + \ell'' = \ell} \partial_X(U_{[k],\ell'} V_{\ell''}) - \frac{1}{2}\sum_{\ell' + \ell'' = \ell} \partial_X(V_{\ell'}V_{\ell''}) = -\partial_X(\Ui{0} V_\ell) + G_{0,\ell}^{(2)}
  \end{equation*}
  for
  \begin{equation}
    \label{eq:Gl2}
    G_{0,\ell}^{(2)} \coloneqq -\sum_{\ell'=0}^{\ell - 1} \partial_X(U_{[k],\ell - \ell'} V_{\ell'}) - \frac{1}{2} \sum_{\ell' = 1}^{\ell - 1} \partial_X(V_{\ell'}V_{\ell - \ell'}).
  \end{equation}

  Then if $G_{0,\ell} \coloneqq G_{0,\ell}^{(1)} + G_{0,\ell}^{(2)}$, we find
  \begin{equation}
    \label{eq:Vl-ev}
    \partial_T V_\ell = -\partial_X(\Ui{0} V_\ell) + \partial_X^2 V_\ell + G_{0,\ell}.
  \end{equation}
  Let $V_\ell^{(n)}$ denote the solution of \eqref{eq:Vl-ev} on $(-n, 0) \times \R$ with initial data
  \begin{equation*}
    V_\ell^{(n)}(-n, \anon) = 0.
  \end{equation*}
  We prove the following by induction:
  \begin{claim}
    \label{claim:Vl-deriv-est-early}
    For each $i \in \Z_{\geq 0}$, there exist constants $C_{i,\ell}^{\mathrm{early}} > 0$ and $T_{i,\ell}^{\mathrm{early}} < 0$ such that for all $n \in \N$,
    \begin{equation}
      \label{eq:Vl-deriv-est-early}
      \abs{\partial_X^i V_\ell^{(n)}} \leq C_{i,\ell}^{\mathrm{early}} \Dist^{-(3i + 4k - \ell + 3)} \quad \text{on } \big(-\infty, T_{i,\ell}^{\mathrm{early}}\big) \times \R.
    \end{equation}
  \end{claim}
  \claimproof{claim:Vl-deriv-est-early}
  Fix $i \in \Z_{\geq 0}$ and suppose we have verified Claim~\ref{claim:Vl-deriv-est-early} for all $i' < i$, noting that this hypothesis is vacuous in the base case $i = 0$.
  We wish to establish the claim for $i$.
  Differentiating \eqref{eq:Vl-ev}, we see that $Z \coloneqq \partial_X^i V_\ell^{(n)}$ satisfies
  \begin{equation}
    \label{eq:Vl-ev-prelim}
    \partial_T Z = -\Ui{0} \partial_X Z - (i + 1) (\partial_X \Ui{0}) Z + \partial_X^2 Z + G_{i,\ell}.
  \end{equation}
  for
  \begin{equation}
    \label{eq:Gikl}
    G_{i,\ell} \coloneqq - \sum_{i'=0}^{i-1} \binom{i+1}{i'} \big(\partial_X^{i-i'+1}\Ui{0}\big) \big(\partial_X^{i'}V_\ell^{(n)}\big) + \partial_X^i G_{0,\ell}.
  \end{equation}

  We claim that
  \begin{equation}
    \label{eq:Gikl-est}
    \abs{G_{i,\ell}} \lesssim_{i,k,l} \Dist^{-(3i + 4k - \ell + 5)}.
  \end{equation}
  To see this, recall from \eqref{eq:Gl1} and \eqref{eq:Gl2} that $G_{0,\ell}$ in \eqref{eq:Vl-ev} is composed of various terms involving homogeneous functions and $V_{\ell'}$ for $\ell' < \ell$.
  Using \eqref{eq:Ukl-est} and \eqref{eq:Vl-est}, a calculation shows that
  \begin{equation}
    \label{eq:Gl-est}
    \abs{\partial_X^i G_{0,\ell}} \lesssim_{i,k,\ell} \Dist^{-(3i + 4k - \ell + 5)}.
  \end{equation}
  Next, the sum in \eqref{eq:Gikl} involves $\Ui{0}$ and lower-order derivatives of $V_\ell^{(n)}$, which are controlled by Claim~\ref{claim:Vl-deriv-est-early} for $i' < i$.
  Combining the $\ell = 0$ case of \eqref{eq:Ul-est} and \eqref{eq:Ukl-est}, we have
  \begin{equation}
    \label{eq:U0-deriv-bd}
    \abs{\partial_X^{i-i'+1} \Ui{0}} \lesssim_{i,k} \Dist^{-(3i - 3i' + 2)} \quad \text{on } \m{Q}^c.
  \end{equation}
  Also, \eqref{eq:Vl-deriv-est-early} for $i' < i$ yields
  \begin{equation}
    \label{eq:Vl-deriv-bd}
    \abs{\partial_X^{i'}V_\ell^{(n)}} \lesssim_{i,k,\ell} \Dist^{-(3i' + 4k - \ell + 3)}.
  \end{equation}
  Combining \eqref{eq:Gikl} and \eqref{eq:Gl-est}--\eqref{eq:Vl-deriv-bd}, we obtain \eqref{eq:Gikl-est}.
  
  We now extract the leading part of $\Ui{0}$ in \eqref{eq:Vl-ev-prelim}:
  \begin{equation}
    \label{eq:Vl-ev-standard}
    \begin{aligned}
      \partial_T Z = -\Cub \partial_X Z - (i+1)(\partial_X \Cub) Z - &(\Ui{0} - \Cub) \partial_X Z\\
      &- (i+1) \partial_X(\Ui{0} - \Cub) Z + \partial_X^2 Z + G_{i,\ell}.
    \end{aligned}
  \end{equation}
  Combining \eqref{eq:Ukl-est} and \eqref{eq:U0-est}, we find
  \begin{equation}
    \label{eq:U0in-diff}
    \abs{\Ui{0} - \Cub} \lesssim_{k} \Dist^{-3} \And \abs{\partial_X(\Ui{0} - \Cub)} \lesssim_k \Dist^{-6} \quad \text{on } \m{Q}^c.
  \end{equation}
  By \eqref{eq:Gikl-est}, \eqref{eq:Vl-ev-standard} satisfies \eqref{eq:early-cond} in Lemma~\ref{lem:inner-super-sub}\ref{item:inner-super-sub-early}.
  By the lemma, there exist constants $C_{i,\ell}^{\mathrm{early}} > 0$ and $T_{i,\ell}^{\mathrm{early}} < 0$ and a signed super/subsolution $\pm W_{i,\ell}^{\mathrm{early}}$ of \eqref{eq:Vl-ev-standard} such that $W_{i,\ell}^{\mathrm{early}} \leq C_{i,\ell}^{\mathrm{early}} \Dist^{-(3i + 4k - \ell + 3)}$ on $\big(-\infty, T_{i,\ell}^{\mathrm{early}}\big) \times \R$.
  Since $Z$ starts from $0$ at time $-n$, the comparison principle implies \eqref{eq:Vl-deriv-est-early} for $i$.
  Now Claim~\ref{claim:Vl-deriv-est-early} follows by induction.
  \claimqed

  Assume without loss of generality that the sequence $(T_{i,\ell}^{\mathrm{early}})_{i \geq 0}$ is decreasing.
  Employing the diagonalization argument from the $\ell = 0$ case, we can extract a subsequence of $(V_\ell^{(n)})_{n \geq 1}$ converging locally uniformly as $n \to \infty$ to a solution $V_\ell^*$ of \eqref{eq:Vl-ev} on $\big(-\infty, T_{3,k,\ell}^{\mathrm{early}}\big) \times \R$ such that for all $i \in \Z_{\geq 0}$,
  \begin{equation}
    \label{eq:Vl-limit-est-early}
    \abs{\partial_X^i V_\ell^*} \leq C_{i,\ell}^{\mathrm{early}} \Dist^{-(3i + 4k - \ell + 3)} \quad \text{on } \big(-\infty, T_{i,\ell}^{\mathrm{early}}\big) \times \R.
  \end{equation}
  Moreover, the linear equation \eqref{eq:Ul-ev} is well posed on the space of functions of polynomial growth, so we can extend $\Ui{\ell} \coloneqq U_{[k],\ell} + V_\ell^*$ to a solution defined on $\R \times \R$.
  We can likewise extend $V_\ell^* \coloneqq \Ui{\ell} - U_{[k],\ell}$ to $\R_- \times \R$.

  It remains to extend the estimates \eqref{eq:Vl-limit-est-early} up to time $0$ in $\m{Q}^c$.
  We will inductively prove the following:
  \begin{claim}
    \label{claim:Vl-deriv-est-late}
    For each $i \in \Z_{\geq 0}$, there exists $C_{i,\ell} > 0$ such that
    \begin{equation}
      \label{eq:Vl-deriv-est-late}
      \abs{\partial_X^i V_\ell^*} \leq C_{i,\ell} \Dist^{-(3i + 4k - \ell + 3)} \quad \text{on } \m{Q}^c.
    \end{equation}
  \end{claim}
  \claimproof{claim:Vl-deriv-est-late}Fix $i \in \Z_{\geq 0}$ and suppose we have verified Claim~\ref{claim:Vl-deriv-est-late} for all $i' < i$.
  As usual, this assumption is vacuous in the base case $i = 0$.
  We wish to establish the claim for $i$.
  The derivative $Z \coloneqq \partial_X^i V_\ell^*$ solves
  \begin{equation}
    \label{eq:Vl-limit-ev-standard}
    \begin{aligned}
      \partial_T Z = -\Cub \partial_X Z - (i+1)(\partial_X \Cub) Z - &(\Ui{0} - \Cub) \partial_X Z\\
      &- (i+1) \partial_X(\Ui{0} - \Cub) Z + \partial_X^2 Z + G_{i,\ell}^*
    \end{aligned}
  \end{equation}
  for
  \begin{equation*}
    G_{i,\ell}^* \coloneqq - \sum_{i'=0}^{i-1} \binom{i+1}{i'} \big(\partial_X^{i-i'+1}\Ui{0}\big)\big(\partial_X^{i'}V_\ell^*\big) + \partial_X^i G_{0,\ell}.
  \end{equation*}
  Combining \eqref{eq:Ukl-est}, \eqref{eq:U0-est}, \eqref{eq:Gl-est}, and \eqref{eq:Vl-deriv-est-late} for $i' < i$, we find
  \begin{equation*}
    \abs{G_{i,\ell}^*} \lesssim_{i,k,\ell} \Dist^{-(3i + 4k - \ell + 5)} \quad \text{on } \m{Q}^c.
  \end{equation*}
  Together with \eqref{eq:U0in-diff}, this shows that \eqref{eq:Vl-limit-ev-standard} satisfies \eqref{eq:late-cond} in Lemma~\ref{lem:inner-super-sub}\ref{item:inner-super-sub-late}.
  
  Now, $\Ui{\ell}$ is smooth in $\R\times \R$ while $U_{[k], \ell}$ is smooth locally uniformly in $\big((-\infty, 0] \times \R\big) \setminus \{(0, 0)\}$, so $V_\ell^*$ is uniformly smooth on $\partial^* \m{Q} = [-1,0) \times \{\pm 1\}$.
  By \eqref{eq:Vl-limit-est-early}, there exists $\const \in \big(\big|T_{i,\ell}^{\mathrm{early}}\big|,\infty\big)$ such that
  \begin{equation*}
    \abs{Z} \leq \const \Dist^{-(3i + 4k - \ell + 3)} \quad \text{on } \big(\{-\const\} \times \R\big) \cup \partial^* \m{Q}.
  \end{equation*}
  By Lemma~\ref{lem:inner-super-sub}\ref{item:inner-super-sub-late}, there exist $C_{i,\ell}^{\mathrm{late}} > 0$ and a signed super/subsolution $\pm W_{i,\ell}^{\mathrm{late}}$ of \eqref{eq:Vl-ev-standard} on $\{-\const < T < 0\} \setminus \m{Q}$ such that
  \begin{equation*}
    W_{i,\ell}^{\mathrm{late}} > \const \Dist^{-(3i + 4k - \ell + 3)} \quad \text{on } \big(\{-\const\} \times \R\big) \cup \partial^* \m{Q}
  \end{equation*}
  and
  \begin{equation*}
    W_{i,\ell}^{\mathrm{late}} < C_{i,\ell}^{\mathrm{late}} \Dist^{-(3i + 4k - \ell + 3)} \quad \text{on } \{-\const < T < 0\} \setminus \m{Q}.
  \end{equation*}
  By the comparison principle in $\{-\const < T < 0\} \setminus \m{Q}$,
  \begin{equation*}
    \abs{Z} \leq  W_{i,\ell}^{\mathrm{late}} < C_{i,\ell}^{\mathrm{late}} \Dist^{-(3i + 4k - \ell + 3)} \quad \text{on } \{-\const < T < 0\} \setminus \m{Q}. 
  \end{equation*}
  In light of \eqref{eq:Vl-limit-est-early}, we have proven Claim~\ref{claim:Vl-deriv-est-late} for $i$.
  The full claim then follows from induction on $\ell$.
  \claimqed
  \smallskip

  \noindent
  In turn, we have now verified Claim~\ref{claim:Vl} for $\ell$.
  The full form of Claim~\ref{claim:Vl} follows from induction.
  \claimqed
  \medskip

  For each pair $(k,\ell)$ with $k \geq \ell/4$, we have constructed a solution $\Ui{\ell}$ of \eqref{eq:Ul-ev} satisfying \eqref{eq:Ul-est} for all $i \in \Z_{\geq 0}$.
  We now show that this solution is unique.  

  First consider $\ell = 0$.
  Suppose $U^{(1)}$, $U^{(2)}$ satisfy \eqref{eq:U0-ev} and \eqref{eq:Ul-est} with $k_1,k_2\geq 0$ respectively for each $i \in \{0, 1\}$.
  The difference $V \coloneqq U^{(1)} - U^{(2)}$ satisfies
  \begin{equation}
    \label{eq:unique0-ev}
    \partial_T V = -\partial_X(\braket{U}V) + \partial_X^2 V,
  \end{equation}
  where $\braket{U} \coloneqq (U^{(1)}+ U^{(2)})/2$.
  Using \eqref{eq:Ul-est} for $i = 0$ and $k = k_1, k_2$ as well as \eqref{eq:Ukl-est},
  \begin{equation*}
    \abs{V} \leq \big|U^{(1)} - U_{[k_1],0}\big| + \big|U_{[k_1],0} - U_{[k_2],0}\big| + \big|U^{(2)} - U_{[k_2],0}\big| \lesssim_{k_1,k_2} \Dist^{-3} \quad \text{in } \m{Q}^c.
  \end{equation*}
  Thus
  \begin{equation*}
    \varphi(T) \coloneqq \norm{V(T, \anon)}_{L^2(\R)}^2 \lesssim_{k_1,k_2} \int_{\R} \Dist(T, X)^{-6} \d X.
  \end{equation*}
  Using Lemma~\ref{lem:cubic}, we find
  \begin{equation}
    \label{eq:unique-a-priori}
    \varphi(T) \lesssim_{k_1,k_2} \int_{\R} \Dist(T, X)^{-6} \d X \lesssim \abs{T}^{-3/2}.
  \end{equation}
  On the other hand, we can use \eqref{eq:unique0-ev} to control the growth of $\varphi$.
  Multiplying \eqref{eq:unique0-ev} by $V$ and integrating by parts, we see that
  \begin{equation*}
    \dot \varphi(T) \leq \norm{\partial_X \braket{U}(T, \anon)}_{L^\infty(\R)} \varphi(T).
  \end{equation*}
  Fix $T < 0$.
  By Gr\"onwall,
  \begin{equation}
    \label{eq:unique-Gronwall}
    \varphi(T) \leq \varphi(-N) \exp\left(\int_{-N}^{T} \norm{\partial_X \braket{U}(S, \anon)}_{L^\infty(\R)} \d S\right)
  \end{equation}
  for $N > \abs{T}$.
  Now, \eqref{eq:Ul-est} with $i=1$, \eqref{eq:Ukl-est}, and Lemma~\ref{lem:cubic} imply that
  \begin{equation*}
    \big\|\partial_X\big(U^{(j)} - \Cub\big)(S,\anon)\big\|_{L^\infty(\R)} \lesssim_{k_j} \norm{\Dist(S, \anon)^{-6}}_{L^\infty(\R)} \lesssim \abs{S}^{-3}
  \end{equation*}
  for each $j \in \{1, 2\}$.
  By Corollary~\ref{cor:u0-out-x}, we find
  \begin{equation*}
    \norm{\partial_X \braket{U}(S, \anon)}_\infty = \norm{\partial_X \Cub(S, \anon)}_{L^\infty(\R)} + \m{O}_{k_1,k_2}\big(\abs{S}^{-3}\big) = \abs{S}^{-1} + \m{O}_{k_1,k_2}\big(\abs{S}^{-3}\big).
  \end{equation*}
  Thus \eqref{eq:unique-a-priori} and \eqref{eq:unique-Gronwall} yield
  \begin{equation*}
    \varphi(T) \lesssim_{k_1,k_2} N^{-3/2} \exp\big(\log N - \log \abs{T} + \m{O}_{k_1,k_2}(1)\big) \lesssim_{k_1,k_2,T} N^{-1/2}.
  \end{equation*}
  The left side is independent of $N$, so taking $N \to -\infty$, we see that $\varphi(T) = 0$.
  That is, $U^{(1)} = U^{(2)}$.
  
  Finally, take $\ell \geq 1$.
  Suppose $U^{(1)}$, $U^{(2)}$ satisfy \eqref{eq:Ul-ev} and \eqref{eq:Ul-est} with ${k_1,k_2\geq \ell/4}$ respectively and $i \in \{0, 1\}$.
  The difference $V \coloneqq U^{(1)} - U^{(2)}$ satisfies
  \begin{equation*}
    \partial_T V = -\partial_X(\Ui{0}V) + \partial_X^2 V.
  \end{equation*}
  Using \eqref{eq:Ul-est} for $i = 0$ and $k = k_1, k_2$ as well as \eqref{eq:Ukl-est}, the triangle inequality yields
  \begin{align*}
    \abs{V} &\leq \big|U^{(1)} - U_{[k_1],\ell}\big| + \big|U_{[k_1],\ell} - U_{[k_2],\ell}\big| + \big|U^{(2)} - U_{[k_2],\ell}\big|\\
            &\lesssim_{k_1,k_2,\ell} \Dist^{-(4k_1 - \ell) - 3} + \Dist^{-[4(k_1\wedge k_2) - \ell] - 3} + \Dist^{-(4k_2 - \ell) - 3}
  \end{align*}
  in $\m{Q}^c$.
  Because $k_1 \wedge k_2 \geq \ell/4$, we have $\abs{V} \lesssim_{k_1,k_2,\ell} \Dist^{-3}$.
  As in \eqref{eq:unique-a-priori}, this implies that
  \begin{equation*}
    \varphi(T) \coloneqq \norm{V(T, \anon)}_{L^2(\R)}^2 \lesssim_{k_1,k_2,\ell} \abs{T}^{-3/2}.
  \end{equation*}
  Using \eqref{eq:Ul-est} with  $i = 1$ and $\ell = k = 0$ as well as Corollary~\ref{cor:u0-out-x}, we have
  \begin{equation*}
    \norm{\partial_X \Ui{0}(S, \anon)}_{L^\infty(\R)} = \abs{S}^{-1} + \m{O}\big(\abs{S}^{-3}\big).
  \end{equation*}
  Arguing as above, we see that $V = 0$.
  Thus there is a unique solution $\Ui{\ell}$ of \eqref{eq:Ul-ev} satisfying \eqref{eq:Ul-est} for each $i \in \{0, 1\}$ and some $k \geq \ell/4$.

  One minor technical point remains.
  Our proof of \eqref{eq:Ul-est} assumed that $k \geq \ell/4$.
  We must thus verify \eqref{eq:Ul-est} for $k < \ell/4$.
  Let $k_\ell \coloneqq \ceil{\ell/4}$, so that $\Ui{\ell}$ satisfies \eqref{eq:Ul-est} with $k_\ell$.
  Take $k < k_\ell$.
  Using \eqref{eq:Ukl-est}, we can easily check that the difference $U_{[k],\ell} - U_{[k_\ell],\ell}$ is dominated by the first non-shared term: $U_{k+1,\ell}$.
  Precisely,
  \begin{equation*}
    \abs{\partial_X^i(U_{[k],\ell} - U_{[k_\ell],\ell})} \lesssim_{i,\ell} \Dist^{-(3i + 4k - \ell + 3)} \quad \text{on } \m{Q}^c.
  \end{equation*}
  Thus by \eqref{eq:Ul-est} for $k_\ell$ and the triangle inequality,
  \begin{equation*}
    \abs{\partial_X^i(\Ui{\ell} - U_{[k],\ell})} \lesssim_{i,\ell} \Dist^{-(3i + 4k - \ell + 3)} \quad \text{in } \m{Q}^c
  \end{equation*}
  as desired.
  This concludes the proof of Proposition~\ref{prop:Ul}.
\end{proof}
We now construct the super- and subsolutions used in the previous proof.
\begin{proof}[Proof of Lemma~\textup{\ref{lem:inner-super-sub}}]
  \tbf{Part~\ref{item:inner-super-sub-early}.}~Fix $\const > 0$ and $i, k, \ell \in \Z_{\geq 0}$ with $k \geq \ell/4.$
  We consider the equation \eqref{eq:inner-super-sub-early} under the condition \eqref{eq:early-cond}.
  We construct a positive supersolution $W$ of the form
  \begin{equation*}
    W(T, X) = a(T) \Dist(T, X)^{-(3i + 4k - \ell + 3)}
  \end{equation*}
  for some $a \colon (-\infty, T_*) \to \R_+$ and $T_* < 0$.
  Define the nonlinear operator
  \begin{equation}
    \label{eq:NL-prelim}
    \begin{aligned}
      \op{NL}(V) \coloneqq \partial_T V + \Cub \partial_X V + (i + 1) (\partial_X \Cub) V - \big[b &+ g(V)\big]\partial_X V\\
      &- \big[c + h(V)\big]V - \partial_X^2 V - G.
    \end{aligned}
  \end{equation}
  We choose $a$ so that $\op{NL}(W) \geq 0$.

  Following \eqref{eq:dynamic-deriv}, define $\bar{\partial}_T \coloneqq \partial_T + \Cub \partial_X$.
  Recall from Definition~\ref{def:cubic} that
  \begin{equation*}
    \partial_X \Cub = -\Cubder = -\Dist^{-2}.
  \end{equation*}
  We can therefore write \eqref{eq:NL-prelim} as
  \begin{equation*}
    \op{NL}(V) \coloneqq \bar\partial_T V - (i + 1) \Dist^{-2} V - \big[b + g(V)\big]\partial_X V - \big[c + h(V)\big]V - \partial_X^2 V - G.
  \end{equation*}
  Setting $p \coloneqq 3i + 4k - \ell + 3$, we can use \eqref{eq:derivs} to compute
  \begin{equation}
    \label{eq:super-derivs}
    \begin{aligned}
      \bar{\partial}_T \Dist^{-p} &= \frac{p}{2} \Dist^{-(p + 2)},\\
      \partial_X \Dist^{-p} &= 3 \coeff p \Cub \Dist^{-(p + 4)},\\
      \partial_X^2 \Dist^{-p} &= -3 \coeff p \Dist^{-(p + 6)} + 9 \coeff^2 p (p + 4) \Cub^2 \Dist^{-(p+8)}.\\
    \end{aligned}
  \end{equation}
  Now, Definition~\ref{def:cubic} implies that $\abs{\Cub} \leq (3\coeff)^{-1/2} \Dist$, so we have
  \begin{equation*}
    \abs{\partial_X \Dist^{-p}} \leq \sqrt{3 \coeff} p \Dist^{-(p + 3)} \And \abs{\partial_X^2 \Dist^{-p}} \leq 3 \coeff p (p + 5) \Dist^{-(p + 6)}.
  \end{equation*}
  Combining these estimates with \eqref{eq:early-cond}, we have
  \begin{equation}
    \label{eq:super-ineq-prelim-1}
    \begin{aligned}
      \op{NL}(W) \geq \dot{a} \Dist^{-p} + &\left[\left(\frac{p}{2} - i - 1\right)a - \const\right]\Dist^{-(p + 2)}- \const \sqrt{3\coeff}p a^2 \Dist^{-(2p + 3)}\\
      &- \left[\const\big(\sqrt{3\coeff} p + 1\big) + 3\coeff p (p + 5)\right] a\Dist^{-(p + 6)} - \const a^2 \Dist^{-2p}.
    \end{aligned}
  \end{equation}
  Now, $p \geq 3(i + 1)$ because $4k \geq \ell$, so $p/2 - i - 1 \geq (i + 1)/2 \geq 1/2$.
  Let us assume that $a \geq 2\const$, so that
  \begin{equation*}
    \left(\frac{p}{2} - i - 1\right)a - \const \geq 0.
  \end{equation*}
  Also, assume that $T < T_* \leq -1$, so that $\Dist \geq \abs{T}^{1/2} > 1$.
  Recalling that $p \geq 3,$ we have
  \begin{equation*}
    \Dist^{-(p + 6)}, \Dist^{-(2p + 3)}, \Dist^{-2p} \leq \abs{T}^{-3/2} \Dist^{-p}.
  \end{equation*}
  With these bounds, we can condense \eqref{eq:super-ineq-prelim-1}:
  \begin{equation}
    \label{eq:super-ineq-prelim-2}
    \begin{aligned}
      \op{NL}(W) \geq \dot{a} \Dist^{-p} - &\left[\const\big(\sqrt{3\coeff} p + 1\big) + 3\coeff p (p + 5)\right] a \abs{T}^{-3/2}\Dist^{-p}\\
      &\hspace{3cm}- \const\big(\sqrt{3\coeff}p + 1\big) a^2 \abs{T}^{-3/2} \Dist^{-p}.
    \end{aligned}
  \end{equation}
  Finally, $a \geq 2\const$ implies that
  \begin{equation*}
    \begin{aligned}
      \Big[\const\big(\sqrt{3\coeff} p + 1\big) + 3\coeff p &(p + 5)\Big] \frac{1}{a} + \const\big(\sqrt{3\coeff}p + 1\big)\\
      &\leq \frac{1}{2}\big(\sqrt{3\coeff} p+ 1\big)+ \const\big(\sqrt{3\coeff}p + 1\big) + \frac{3\coeff p (p + 5)}{2\const} \eqqcolon C_0.
    \end{aligned}
  \end{equation*}
  Thus \eqref{eq:super-ineq-prelim-2} implies that
  \begin{equation*}
    \op{NL}(W) \geq \big(\dot{a} - C_0 a^2\abs{T}^{-3/2}\big) \Dist^{-p}.
  \end{equation*}
  We want $\op{NL}(W) \geq 0$, so let $a$ solve
  \begin{equation*}
    \dot{a} = C_0 a^2 \abs{T}^{-3/2}, \quad a(-\infty) = 2\const.
  \end{equation*}
  Using separation of variables, we compute:
  \begin{equation*}
    a(T) = \left(\frac{1}{2\const} - 2C_0 \abs{T}^{-1/2}\right)^{-1}.
  \end{equation*}
  Define
  \begin{equation*}
    T_{\mathrm{early}} \coloneqq -\max\big\{64 \const^2 C_0^2,\, 1\big\} \And C_{\mathrm{early}} \coloneqq 4\const.
  \end{equation*}
  Then our choice of $a$ ensures that $W_{\mathrm{early}} \coloneqq W$ is a smooth positive supersolution of \eqref{eq:inner-super-sub-early} on $(-\infty, T_{\mathrm{early}}) \times \R$ such that
  \begin{equation*}
    W_{\mathrm{early}} \leq a(T_*) \Dist^{-p} \leq C_{\mathrm{early}} \Dist^{-(3i + 4k - \ell + 3)}.
  \end{equation*}
  Moreover, we can easily check that $\op{NL}(-W_{\mathrm{early}}) \leq 0$, i.e., $-W_{\mathrm{early}}$ is a subsolution.
  This completes the proof of Lemma~\ref{lem:inner-super-sub}\ref{item:inner-super-sub-early}.
  \medskip

  \noindent
  \tbf{Part~\ref{item:inner-super-sub-late}.}~Again, fix $\const > 0$ and $i,k,\ell \in \Z_{\geq 0}$.
  We consider the linear equation \eqref{eq:inner-super-sub-late} under the condition \eqref{eq:late-cond}.
  Define the affine operator
  \begin{align*}
    \op{L}(V) &\coloneqq \partial_T V + \Cub \partial_X V + (i + 1) (\partial_X \Cub) V - b\partial_X V - cV - \partial_X^2 V - G\\
              &\,= \bar{\partial}_T V - (i + 1) \Dist^{-2} V - b\partial_X V - cV - \partial_X^2 V - G.
  \end{align*}
  Again, let $W \coloneqq a \Dist^{-p}$ with $p \coloneqq 3i + 4k - \ell + 3$ and $a \colon (-\const, 0) \to [2\const, \infty)$.
  As above, $a \geq 2\const$, \eqref{eq:super-derivs}, and \eqref{eq:late-cond} imply that
  \begin{equation*}
    \op{L}(W) \geq \dot{a} \Dist^{-p} - \sqrt{3\coeff}p\const a \Dist^{-(p + 3)} - \const a \Dist^{-p} - 3\coeff p(p + 5) \Dist^{-(p + 6)} - \const \Dist^{-(p + 2)}.
  \end{equation*}
  Let $\rho \coloneqq \inf_{\m{Q}^c} \Dist > 0$.
  Then
  \begin{equation*}
    \op{L}(W) \geq (\dot{a} - C_1 a) \Dist^{-p}
  \end{equation*}
  for
  \begin{equation*}
    C_1 \coloneqq \sqrt{3 \coeff} p \const \rho^{-3} + \const + 3\coeff p (p + 5) \rho^{-6} + \rho^{-2},
  \end{equation*}
  where we have used $\const \leq a/2 \leq a$ in the last term.
  Thus it suffices to choose
  \begin{equation*}
    a(T) \coloneqq 2\const \e^{C_1(T + \const)}.
  \end{equation*}
  This ensures that $\op{L}(W) \geq 0$ in $\{-\const < T < 0\} \setminus \m{Q}$ and $a \geq 2\const$ for $T \in (-\const, 0)$.
  Thus $W_{\mathrm{late}} \coloneqq W$ is a positive supersolution of \eqref{eq:inner-super-sub-late} in $\{-\const < T < 0\} \setminus \m{Q}$ such that
  \begin{equation*}
    W_{\mathrm{late}} \geq a(-\const) \Dist^{-p} \geq \const \Dist^{-(3i + 4k - \ell + 3)} \And W_{\mathrm{late}} \leq C_{\mathrm{late}} \Dist^{-(3i + 4k - \ell + 3)}
  \end{equation*}
  for $C_{\mathrm{late}} \coloneqq 2\const \e^{C_1 \const}$.
  Moreover, we can easily check that $-W$ is a subsolution.
  This completes the proof of Lemma~\ref{lem:inner-super-sub}\ref{item:inner-super-sub-late}.
\end{proof}
Given $\ell \in \Z_{\geq 0}$, we take $\Ui{\ell}$ as in Proposition~\ref{prop:Ul} and let
\begin{equation}
  \label{eq:Ui-convert}
  \ui{\ell}(t, x) \coloneqq \nu^{(\ell + 1)/4}\Ui{\ell}\big(\nu^{-1/2}t, \nu^{-3/4}x\big)
\end{equation}
The following is immediate from Corollary~\ref{cor:Ul-size} and \eqref{eq:Ui-convert}:
\begin{corollary}
  \label{cor:ul-size}
  For each $i,\ell \in \Z_{\geq 0}$, $|\partial_x^i \ui{\ell}| \lesssim_{i,\ell} \big(\dist \vee \nu^{1/4}\big)^{-3i + \ell + 1}$ on $[t_0, 0) \times [-1, 1]$.
\end{corollary}
We now define the inner partial sum
\begin{equation}
  \label{eq:inner-partial}
  \ui{[\ell]} \coloneqq \sum_{0 \leq \ell' \leq \ell} \ui{\ell'}.
\end{equation}
Likewise, given $k \in \Z_{\geq 0}$, define the weighted outer partial sum
\begin{equation}
  \label{eq:outer-partial}
  \uo{[k]} \coloneqq \sum_{0 \leq k' \leq k} \nu^{k'}\uo{k'}.
\end{equation}
Finally, define the weighted homogeneous partial sums
\begin{equation*}
  u_{[k],\ell} = \sum_{0 \leq k' \leq k} \nu^{k'} u_{k', \ell}
\end{equation*}
and
\begin{equation*}
  u_{[k, \ell]} \coloneqq \sum_{\substack{0 \leq k' \leq k\\ 0 \leq \ell' \leq \ell}} \nu^{k'} u_{k', \ell'} = \sum_{0 \leq k' \leq k} \nu^{k'} u_{k', [\ell]} = \sum_{0 \leq \ell' \leq \ell} u_{[k],\ell'}.
\end{equation*}
This sum represents the ``overlap'' between the partial sums $\uo{[k]}$ and $\ui{[\ell]}$.
Asymptotically, the difference between the partial sums is composed of ``leftover'' terms:
\begin{equation*}
  \uo{[k]} - \ui{[\ell]} \sim \sum_{\substack{0 \leq k' \leq k\\ \ell' > \ell}} \nu^{k'} u_{k', \ell'} - \sum_{\substack{k' > k\\ 0 \leq \ell' \leq \ell}} \nu^{k'} u_{k', \ell'}.
\end{equation*}
Consider the intermediate region $M$ from Section~\ref{sec:overview}, on which $\nu^{1/4} \ll \dist \ll 1$.
There, the terms $\nu^{k'} u_{k', \ell'}$ become smaller as $k'$ and $\ell'$ increase.
Thus the most significant terms in the difference $\uo{[k]} - \ui{[\ell]}$ are $u_{0,\ell + 1}$ and $\nu^{k + 1} u_{k + 1, 0}$.
We represent these relationships graphically in Figure~\ref{fig:overlap}.
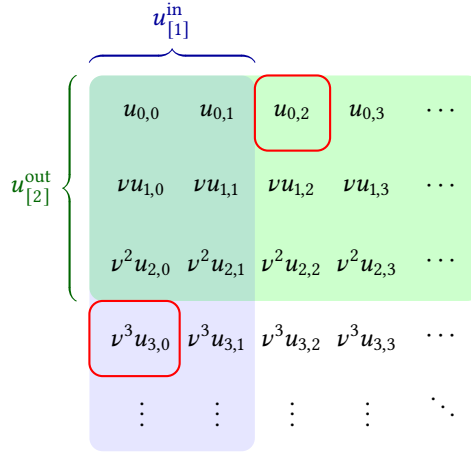
\begin{figure}
  \centering
  \begin{tikzpicture}[scale = 1]
    \tikzstyle{every node}=[font=\small]

    \def\k{3}
    \def\Xeps{0.6}
    \def\Yeps{0.6}

    \fill [green, opacity = 0.2, rounded corners=5pt] (- 0.7, 0.5) rectangle (\k + 1 + 0.5, -2 - 0.5);
    \fill [blue, opacity = 0.1, rounded corners=5pt] (- 0.7, 0.5) rectangle (1 + 0.5, -\k - 1 - 0.5);

    \draw [red, thick, rounded corners=5pt] (- 0.7, -3 - 0.5) rectangle (0.5, -2 - 0.5);
    \draw [red, thick, rounded corners=5pt] (2 - 0.5, - 0.5) rectangle (2 + 0.5, 0.5);

    \foreach \x in {0,1,...,\k}
    {\node at (\x, 0) {$ u_{0, \x}$};
      \node at (\x, -1) {$\nu u_{1, \x}$};
      \foreach \y in {2,...,\k}
      {\node at (\x, -\y) {$\nu^{\y} u_{\y, \x}$};}
      \node at (\x, -\k-0.9) {$\vdots$};}
    
    \foreach \y in {0,...,\k}
    {\node at (\k + 1, -\y) {$\cdots$};}

    \node at (\k + 1, -\k - 0.8) {$\ddots$};

    \draw [pen colour= {green!40!black}, thick, decorate, decoration = {calligraphic brace, raise = 5pt, amplitude = 5pt}] (-0.7, -2 - 0.5) --  (-0.7, 0.5) node[green!40!black, pos = 0.5, left = 10pt]{$\uo{[2]}$};
    \draw [pen colour = {blue!60!black}, thick, decorate, decoration = {calligraphic brace, raise = 5pt, amplitude = 5pt}] (-0.7, 0.5) --  (1 + 0.5, 0.5) node[blue!60!black, pos = 0.5, above = 10pt]{$\ui{[1]}$};
  \end{tikzpicture}
  \caption{
    \tbf{Overlap of partial sums.}
    The partial sums $\uo{[2]}$ and $\ui{[1]}$ represent sums of rows and columns, respectively.
    They overlap in the rectangular blue-green block of terms $u_{[2, 1]}$ (not labeled).
    The singly-shaded blue or green terms comprise the difference ${\uo{[2]} - \ui{[1]}}$.
    These are dominated by the two terms outlined in red, which represent the upper left extremes of the difference set.
  }
  \label{fig:overlap}
\end{figure}

Roughly, the leading error terms $u_{0,\ell + 1}$ and $\nu^{k + 1} u_{k + 1, 0}$ are of size $\dist^{-3i + \ell + 2}$ and $\nu^{k + 1} \dist^{-3i - 4k - 3}$, respectively.
This motivates the following estimate:
\begin{proposition}
  \label{prop:matching}
  For each $i, k, \ell \in \Z_{\geq0}$,
  \begin{equation*}
    \abs{\partial_x^i\left(\uo{[k]} - \ui{[\ell]}\right)} \lesssim_{i,k,\ell} \dist^{-3i + \ell + 2} + \nu^{k + 1} \dist^{-(3i + 4k + 3)}
  \end{equation*}
  on $\{\nu^{1/4} \leq \dist \leq 1\}.$
\end{proposition}
\begin{proof}
  We use the triangle inequality to write
  \begin{equation}
    \label{eq:triangle}
    \abs{\partial_x^i\left(\uo{[k]} - \ui{[\ell]}\right)} \leq \abs{\partial_x^i\left(\uo{[k]} - u_{[k, \ell]}\right)} + \abs{\partial_x^i\left(\ui{[\ell]} - u_{[k,\ell]}\right)}.
  \end{equation}
  We control the first term using Proposition~\ref{prop:uk-homogeneous}:
  \begin{equation*}
    \abs{\partial_x^i\left(\uo{[k]} - u_{[k, \ell]}\right)} \leq \sum_{0 \leq k' \leq k} \nu^{k'} \abs{\partial_x^i\left(\uo{k'} - u_{k', [\ell]}\right)} \lesssim_{i,k,\ell} \dist^{-3i + \ell + 2} \sum_{0 \leq k' \leq k} (\nu \dist^{-4})^{k'}.
  \end{equation*}
  Because $\dist \geq \nu^{1/4}$, $\nu \dist^{-4} \lesssim 1$.
  It follows that
  \begin{equation}
    \label{eq:sum-outer}
    \abs{\partial_x^i\left(\uo{[k]} - u_{[k, \ell]}\right)} \lesssim_{i,k,\ell} \dist^{-3i + \ell + 2}.
  \end{equation}

  To control the second term in \eqref{eq:triangle}, we must convert our bounds on the inner solutions back to the original variables.
  We plug $T = \nu^{-1/2}t$ and $X = \nu^{-3/4}x$ into \eqref{eq:Ul-est} and use $\partial_X = \nu^{3/4} \partial_x$, $\Dist = \nu^{-1/4}\dist$, and
  \begin{equation*}
    U_{[k], \ell'} = \sum_{0 \leq k' \leq k} \nu^{k' - (\ell' + 1)/4} u_{k',\ell'} = \nu^{-(\ell' + 1)/4} u_{[k], \ell'}.
  \end{equation*}
  After a brief calculation, \eqref{eq:Ul-est} becomes
  \begin{equation}
    \label{eq:ul-est}
    \abs{\partial_x^i \left(\ui{\ell'} - u_{[k], \ell'}\right)} \lesssim_{i, k, \ell'} \nu^{k + 1} \dist^{-3i - 4k + \ell' - 3}.
  \end{equation}
  Therefore
  \begin{equation*}
    \abs{\partial_x^i\left(\ui{[\ell]} - u_{[k,\ell]}\right)} \leq \sum_{0 \leq \ell' \leq \ell} \abs{\partial_x^i \left(\ui{\ell'} - u_{[k], \ell'}\right)} \lesssim_{i,k,\ell} \nu^{k + 1} \dist^{-3i - 4k - 3} \sum_{0 \leq \ell' \leq \ell} \dist^{\ell'}.
  \end{equation*}
  Now $\dist \leq 1$, so
  \begin{equation}
    \label{eq:sum-inner}
    \abs{\partial_x^i\left(\ui{[\ell]} - u_{[k,\ell]}\right)} \lesssim_{i,k,\ell} \nu^{k + 1} \dist^{-3i - 4k - 3}.
  \end{equation}
  Combining \eqref{eq:sum-outer} with \eqref{eq:sum-inner}, \eqref{eq:triangle} implies the proposition.
\end{proof}

\section{The approximate solution}
\label{sec:approx}
We have now constructed complete inner and outer expansions.
We can therefore approximate the solution $u^\nu$ of \eqref{eq:Burgers-viscous} on the entire domain $\R \times [t_0, 0)$.
We briefly recall the strategy outlined in Section~\ref{subsec:approx}.
Given $\al \in (0, 1)$, we define the \emph{inner}, \emph{outer}, and \emph{matching} regions as follows:
\begin{equation*}
  I = \{\dist < \nu^\al\}, \quad O = \{\dist > 2 \nu^\al\}, \quad M = \{\nu^\al \leq \dist \leq 2 \nu^\al\}.
\end{equation*}
We will define an approximate solution $\uapp$ by truncating the inner and outer expansions in $I$ and $O$, respectively.
In the matching zone $M$, we interpolate between these truncations.
For simplicity, we use the same number of terms from the inner and outer expansions.
Given $K \in \Z_{\geq 0}$, we construct $\uapp_{[K]}$ from the partial sums $\ui{[K]}$ and $\uo{[K]}$ defined in \eqref{eq:inner-partial} and \eqref{eq:outer-partial}.

Let $\vartheta \in \m{C}_c^\infty([0, \infty))$ be a bump function such that $\vartheta|_{[0, 1]} \equiv 1$, $\vartheta|_{[2, \infty)} \equiv 0$, and $0 \leq \vartheta \leq 1$.
We define the cutoff
\begin{equation*}
  \theta(t, x) \coloneqq \vartheta\left(\dist(t, x) \nu^{-\al}\right)
\end{equation*}
so that $\theta|_I \equiv 1$ and $\theta|_{O} \equiv 0$.
We then define the approximate solution
\begin{equation}
  \label{eq:uapp-def}
  \uapp_{[K]} \coloneqq \theta \ui{[K]} + (1 - \theta) \uo{[K]}.
\end{equation}

Note that
\begin{equation*}
  \uapp_{[K]}(t_0, \anon) = \uo{[K]}(t_0, \anon) = \mr{u} = u^\nu(t_0, \anon).
\end{equation*}
That is, the approximate solution agrees with the true solution at the initial time $t_0$.
However, the functions diverge when $t > t_0$ because $\uapp_{[K]}$ is not an exact solution of \eqref{eq:Burgers-viscous}.
The ``error''
\begin{equation}
  \label{eq:error}
  E \coloneqq -\partial_t \uapp_{[K]} - \uapp_{[K]} \partial_x \uapp_{[K]} + \nu \partial_x^2 \uapp_{[K]}
\end{equation}
captures this inexactness.
The approximate solution $\uapp_{[K]}$ is governed by two parameters: $\al$ and $K$.
We choose these to make $E$ appropriately small.

On $I$ and $O$, the error $E$ is due to our truncation of the inner and outer expansions.
The error can thus be made arbitrarily small by taking $K$ large.
On $M$, there is an additional source of error---the cutoff interpolating $\ui{[K]}$ and $\uo{[K]}$.
To minimize this interpolation error, we want $\ui{[K]}$ and $\uo{[K]}$ to be as similar as possible in $M$.
By Proposition~\ref{prop:matching},
\begin{equation*}
  \big|\ui{[K]} - \uo{[K]}\big| \lesssim_K \dist^{K + 2} + \nu^{K + 1} \dist^{-4K -3}.
\end{equation*}
The right side is minimized when the two terms are of the same order, i.e., when $\dist \asymp \nu^{1/5}$.
We therefore choose $\al = 1/5$ in our definition of $I, O$, and $M$, so that the interpolation error is minimal in the matching zone.
We emphasize that this exponent originates in our choice to use equally many terms from the inner and outer expansions.
If we used unbalanced truncations, a different exponent would be optimal.

We can now state our main result.
\begin{theorem}
  \label{thm:main}
  Let $u^\nu$ solve \eqref{eq:Burgers-viscous} and \eqref{eq:Burgers-viscous-init}.
  Fix $K \in \Z_{\geq 0}$ and take $\uapp_{[K]}$ as in \eqref{eq:uapp-def}.
  Then there exists a constant $C(K, \mr{u}) > 0$ such that for all $\nu \in (0, 1]$, 
  \begin{equation}
    \label{eq:pointwise-v-gen}
    \big\|u^\nu- \uapp_{[K]}\big\|_{L^\infty([t_0, 0) \times \R)} \leq C(K, \mr{u}) \nu^{(K + 2)/5}.
  \end{equation}
  Moreover, for all $K \geq 0$,
  \begin{equation}
    \label{eq:more-regular}
    \sup_{t \in [t_0, 0)}\big\|u^\nu(t, \anon) - \uapp_{[K]}(t, \anon)\big\|_{\mathcal{C}^{1/2}(\R)} \to 0 \quad \textrm{as }\; \nu \to 0.
  \end{equation}
\end{theorem}
We can thus approximate $u^\nu$ to arbitrary precision in $\nu$ by including sufficiently many terms from the inner and outer expansions.
Furthermore, the difference $u^\nu - \uapp_{[K]}$ is uniformly $\frac{1}{2}$-H\"older.
\begin{remark}
  As noted in Remark~\ref{rem:more-regular}, we expect $u^\nu - \uapp_{[K]}$ to remain uniformly bounded in $\m{C}_x^{(K+2)/3}$ for all $K \in \Z_{\geq 0}$.
  This is the (quantitative) regularity of the leading terms $u_{0,K+1}$ and $\nu^{K + 1} u_{K+1, 0}$ in the difference $\uo{[K]} -\ui{[K]}$ on $M$.
\end{remark}
Theorem~\ref{thm:main} implies the informal Theorem~\ref{thm:rough} stated in the introduction.
\begin{proof}[Proof of Theorem~\textup{\ref{thm:rough}}]
  Only the first part of \eqref{eq:1/2-regular} remains to be shown.
  Due to the inverse cubic term $u_{0,0} = \cub$ on $O$, the approximate solution $\uapp_{[0]}$ is only uniformly bounded in $\mathcal{C}^{1/3}$ as $\nu \to 0$.
  That is, if $\beta > 1/3$,
  \begin{equation}
    \label{eq:app-irregular}
    \sup_{t \in [t_0, 0)}\big\|\uapp_{[0]}(t, \anon)\big\|_{\mathcal{C}^\beta(\R)} \to \infty \quad \textrm{as }\; \nu \to 0.
  \end{equation}
  Moreover, the higher-order terms in $\uapp_{[K]}$ are all smoother---it is straightforward to check that Corollaries~\ref{cor:uk-der} and \ref{cor:ul-size} as well as Lemma~\ref{lem:size-der-theta} below imply that
  \begin{equation*}
    \sup_{t \in [t_0, 0)}\big\|\big(\uapp_{[K]} - \uapp_{[0]}\big)(t, \anon)\big\|_{\mathcal{C}^{2/3}(\R)} \lesssim 1.
  \end{equation*}
  Hence \eqref{eq:app-irregular} holds with $\uapp_{[K]}$ in place of $\uapp_{[0]}$.
  Thus the first part of \eqref{eq:1/2-regular} follows from the second.
\end{proof}
We can also prove Corollary~\ref{cor:inviscid-limit}.
\begin{proof}[Proof of Corollary~\textup{\ref{cor:inviscid-limit}}]
  By Theorem~\ref{thm:main} and the triangle inequality,
  \begin{equation}
    \label{eq:true-app-diff}
    \big\|u^\nu - \uapp_{[0]}\big\|_{L^\infty([t_0, 0) \times \R)} \lesssim \nu^{2/5}.
  \end{equation}
  Recall the inviscid solution $u^0 = \uo{0}$.
  We claim that
  \begin{equation}
    \label{eq:app-inviscid-diff}
    \big\|\uapp_{[0]} - \uo{0}\big\|_{L^\infty([t_0, 0) \times \R)} \asymp \nu^{1/4}.
  \end{equation}
  Indeed, by \eqref{eq:uapp-def},
  \begin{equation*}
    \uapp_{[0]} - u^0 = (\ui{0} - \uo{0})\theta.
  \end{equation*}
  Applying Proposition~\ref{prop:matching},
  \begin{equation*}
    \abs{(\ui{0} - \uo{0})\theta} \lesssim \left(\dist^{2} + \nu \dist^{-3}\right)\theta
  \end{equation*}
  when $\dist \geq \nu^{1/4}$.
  Because $\theta$ is supported on $I \cup M$, the second term dominates, and
  \begin{equation}
    \label{eq:matching-off-innermost}
    \abs{(\ui{0} - \uo{0})\theta} \lesssim \nu^{1/4} \quad \text{when } \dist \geq \nu^{1/4}.
  \end{equation}
  When $\dist \leq \nu^{1/4}$, Proposition~\ref{prop:uk-homogeneous} implies that
  \begin{equation}
    \label{eq:triangle-inner}
    \abs{(\ui{0} - \uo{0})\theta} = \abs{\ui{0} - \uo{0}} = \abs{\ui{0} - \cub} + \m{O}\big(\nu^{1/2}\big).
  \end{equation}
  In the inner coordinates, $\Ui{0}$ and $\Cub$ are independent of $\nu$ and distinct---indeed, the former is smooth on $(-\infty, 0] \times \R$ and the latter is not.
  It follows that $\abs{\Ui{0} - \Cub} \asymp 1$ on $\{\Dist \leq 1\}$.
  When we return to the original coordinates, this difference is scaled by $\nu^{1/4}.$
  Thus \eqref{eq:triangle-inner} yields
    \begin{equation*}
    \abs{(\ui{0} - \uo{0})\theta} = \abs{\ui{0} - \cub} + \m{O}\big(\nu^{1/2}\big) \asymp \nu^{1/4} \quad \text{when } \dist \leq \nu^{1/4}.
  \end{equation*}
  Together with \eqref{eq:matching-off-innermost}, this implies \eqref{eq:app-inviscid-diff}.
  Now \eqref{eq:inviscid-limit} follows from \eqref{eq:true-app-diff} and \eqref{eq:app-inviscid-diff}.
\end{proof}
\begin{remark}
  By \eqref{eq:true-app-diff}, $\uapp_{[0]}$ is a better approximation of $u^\nu$ than $u^0$ alone.
\end{remark}

\subsection{Energy estimates}
We rely on an $L^2$ approach to prove Theorem~\ref{thm:main}.
We begin with the energy estimates used in the proof.
For concision, we often abbreviate $\uapp_{[K]}$ by $\uapp$ when the particular index is insignificant.
We study the evolution equation for the difference $v \coloneqq u^\nu - \uapp$:
\begin{equation}
  \label{eq:diff}
  \partial_t v = - \partial_x(\uapp v) - v\partial_x v + \nu \partial_x^2 v + E, \quad v(t_0, \anon) = 0,
\end{equation}
where $E$ is defined in \eqref{eq:error}.
Before analyzing the forcing $E$, we prove energy estimates for arbitrary $E$.
In the following, we use the shorthand
\begin{equation*}
  \norm{f}_{V_x}\!(t) \coloneqq \norm{f(t, \anon)}_V
\end{equation*}
when $V$ is a norm on functions on $\R$.
\begin{lemma}
  \label{lem:eng-est-v}
  Let $v$ satisfy \eqref{eq:diff}.
  Then for all $\eta\in [0,1)$ and $\mathfrak{t} \in [t_0, 0]$,
  \begin{equation}
    \label{eq:eng-est}
    \norm{v}^2_{L^2_x}\!(\mathfrak{t})\leq \left(\int_{t_0}^{\mathfrak{t}}\!\!\!\int_{\R} |t|^\eta E^2\d x \ds t\right)\exp\left[\int_{t_0}^{\mathfrak{t}}\big(\norm{\partial_x \uapp}_{L^\infty_x}+|t|^{-\eta}\big)\ds t\right]
  \end{equation}
  and
  \begin{equation}
    \label{eq:eng-est-der}
    \begin{aligned}
      \norm{\partial_x v}^2_{L^2_x}\!(\mathfrak{t})&\leq 
      \left(\int_{t_0}^{\mathfrak{t}}\!\!\!\int_{\R} |t|^\eta\big[(\partial_x E)^2+(\partial^2_x\uapp)^2 v^2\big]\ds x \ds t\right)\\
      &\hspace{1cm}\times \exp\left[\int_{t_0}^{\mathfrak{t}}\big(3\norm{\partial_x \uapp}_{L^\infty_x}+\nu^{-1}\norm{v}_{L^\infty_x}^2+2|t|^{-\eta}\big)\ds t\right].
    \end{aligned}
  \end{equation}
\end{lemma}

\begin{proof}
  Fix $\eta \in (0, 1]$ and $\f{t} \in [t_0, 0]$.
  We first prove \eqref{eq:eng-est}.
  Multiplying \eqref{eq:diff} by $v$, we find
  \begin{align*}
    \frac{1}{2}\partial_t(v^2) - \nu v \partial_x^2 v= - v\partial_x(\uapp v)-v^2\partial_x{v} + E v = - v\partial_x(\uapp v) - \frac{1}{3} \partial_x(v^3) + E v.
  \end{align*}
  We integrate over $[t_0,\mathfrak{t}]\times \R$ and integrate by parts in space:
  \begin{equation}
    \label{eq:energy-prelim}
    \frac{1}{2}\norm{v}^2_{L^2_x}\!(\mathfrak{t}) + \nu\int_{[t_0, \f{t}] \times \R} (\partial_x v)^2 \leq \int_{[t_0, \f{t}] \times \R} \left(-\frac{1}{2}v^2 \partial_x\uapp + Ev\right).
  \end{equation}
  We use Young's identity to write $2Ev \leq \abs{t}^{\eta} E^2 + \abs{t}^{-\eta} v^2$.
  Dropping the viscous term, \eqref{eq:energy-prelim} yields
  \begin{equation*}
    \norm{v}^2_{L^2_x}\!(\mathfrak{t}) \leq \int_{[t_0, \f{t}] \times \R} \abs{t}^{\eta} E^2 + \int_{t_0}^{\f{t}} \left(\norm{\partial_x \uapp}_{L_x^\infty} + \abs{t}^{-\eta}\right) \norm{v}_{L_x}^2.
  \end{equation*}
  Gr\"{o}nwall's inequality implies \eqref{eq:eng-est}.
  
  We now turn to \eqref{eq:eng-est-der}.
  Differentiating \eqref{eq:diff}, we find
  \begin{equation*}
    \partial_t \partial_x v - \nu \partial_x^3v = - \partial_x^2(\uapp v) - \partial_x(v \partial_x v) + \partial_x E.
  \end{equation*}
  We multiply by $\partial_x v$, integrate over $[t_0, \f{t}] \times \R$, and integrate by parts in space:
  \begin{equation}
    \label{eq:eng-est-deriv-prelim}
    \begin{aligned}
      \frac{1}{2}\norm{\partial_x v}_{L^2_x}^2(\mathfrak{t}) + \nu\int_{[t_0, \f{t}] \times \R} &(\partial_x^2 v)^2\\
      &= \int_{[t_0, \f{t}] \times \R} \left[ - \partial_x^2(\uapp v) - \partial_x(v \partial_x v) + \partial_x E\right]\partial_x v.
    \end{aligned}
  \end{equation}
  Expanding $\partial^2_x(\uapp v)$, we obtain
  \begin{equation*}
    \begin{aligned}
      \int_{\R}(\partial_x v)\partial^2_x(\uapp v)\d x= \frac{1}{2}\int_{\R}&\uapp\partial_x (\partial_x v)^2\d x\\
      &+ 2\int_{\R}(\partial_x \uapp)(\partial_x v)^2\d x + \int_{\R}(\partial^2_x \uapp)v\partial_x v\d x.
    \end{aligned}
  \end{equation*}
  Integrating the first term on the right by parts, this yields
  \begin{equation}
    \label{eq:eng-est-second}
    \int_{\R}(\partial_x v)\partial^2_x(\uapp v)\d x = \frac{3}{2}\int_{\R}(\partial_x \uapp)(\partial_x v)^2\d x + \int_{\R}(\partial^2_x \uapp)v \partial_x v\d x.
  \end{equation}
  Similarly, we have
  \begin{equation}
    \label{eq:eng-est-cubic}
    -\int_{\R}(\partial_x v)\partial_x(v\partial_x v)\d x = \int_{\R} v (\partial_x v) \partial^2_x v \d x.
  \end{equation}
  Combining \eqref{eq:eng-est-second} and \eqref{eq:eng-est-cubic}, \eqref{eq:eng-est-deriv-prelim} yields
  \begin{equation}
    \label{eq:eng-est-deriv-prelim2}
    \begin{aligned}
      \norm{\partial_x v}_{L^2_x}^2\!(\mathfrak{t}) + 2\nu \int_{[t_0, \f{t}] \times \R} &(\partial_x^2v)^2\leq
      \int_{[t_0, \f{t}] \times \R}\big[2|v (\partial_x v) \partial^2_x v| + 2|\partial_x v||\partial_xE|\big]\\
      &+\int_{[t_0, \f{t}] \times \R} \big[3|\partial_x \uapp| (\partial_x v)^2 + 2 |(\partial^2_x\uapp) v \partial_x v|\big].
    \end{aligned}
  \end{equation}
  Now, Young's inequality implies
  \begin{align*}
    2 |(\partial^2_x\uapp) v \partial_x v| &\leq \abs{t}^{-\eta} (\partial_x v)^2 + \abs{t}^\eta (\partial_x^2 \uapp)^2 v^2,\\
    2|v (\partial_x v) \partial^2_x v| &\leq \nu (\partial_x^2v)^2 + \nu^{-1} (\partial_x v)^2 v^2,\\
    2\abs{\partial_x v \partial_x E} &\leq \abs{t}^\eta (\partial_xE)^2 + \abs{t}^{-\eta} (\partial_xv)^2.
  \end{align*}
  Using these bounds in \eqref{eq:eng-est-deriv-prelim2} and rearranging, we find
  \begin{equation*}
    \begin{aligned}
      \norm{\partial_x v}_{L^2_x}^2\!(\mathfrak{t}) \leq
      &\int_{[t_0, \f{t}] \times \R} \abs{t}^{\eta}\left[(\partial_x^2 \uapp)^2 v^2 + (\partial_xE)^2\right]\\
      &+\int_{[t_0, \f{t}] \times \R} \left[3\abs{\partial_x \uapp} + \nu^{-1} v^2 + 2 \abs{t}^{-\eta}\right](\partial_xv)^2.
    \end{aligned}
  \end{equation*}
  Now \eqref{eq:eng-est-der} follows from Gr\"{o}nwall.
\end{proof}
\begin{remark}
  We use the viscosity to absorb an adverse term in the proof of \eqref{eq:eng-est-der}.
  This is a convenience, rather than a necessity.
  Indeed, we could take more derivatives (two suffice) and close in $L^2$, which would give pointwise control on the offending factor (namely, $\partial_x^2 v$).
  Thus, we do not need the ``good'' viscous term this section.
  We have already captured viscous effects in our construction of the inner and outer expansions; we do not need to use the viscosity further.
\end{remark}

\subsection{Forcing estimates}
In light of Lemma~\ref{lem:eng-est-v}, we need to control $E$ in $H_x^1$.
The error has a different character in the three zones $O, I,$ and $M$.
In $O$ and $I$, $E$ is the error from truncating the inner and outer expansions, respectively.
Using \eqref{eq:ukout-ev}, \eqref{eq:Ul-ev}, and \eqref{eq:Ui-convert}, we can compute
\begin{equation}
  \label{eq:E-out}
  \begin{aligned}
    E^{\text{out}} &\coloneqq -\partial_t \uo{[K]} + \nu \partial_x^2 \uo{[K]} - \uo{[K]} \partial_x\uo{[K]}\\
    &= -\frac{1}{2} \sum_{\substack{1 \leq k',k'' \leq K\\ k' + k'' > K}} \nu^{k' + k''} \partial_x\left(\uo{k'}\uo{k''}\right) + \nu^{K + 1} \partial_x^2 \uo{K}.
  \end{aligned}
\end{equation}
and
\begin{equation}
  \label{eq:E-in}
  E^{\text{in}} \coloneqq -\partial_t \ui{[K]} + \nu \partial_x^2 \ui{[K]} - \ui{[K]} \partial_x\ui{[K]} = -\frac{1}{2} \sum_{\substack{1 \leq \ell',\ell'' \leq K\\ \ell' + \ell'' > K}} \partial_x\left(\ui{\ell'}\ui{\ell''}\right).
\end{equation}

Let $E^{\text{match}}$ denote the error in the matching zone $M$, so that:
\begin{equation*}
  E(t,x) = 
  \begin{cases} 
    E^{\text{out}}& \text{in }O\\
    E^{\text{match}}& \text{in }M\\
    E^{\text{in}}& \text{in }I.
  \end{cases}
\end{equation*}
The matching error is somewhat complicated due to the cutoff.
Define the mismatch
\begin{equation*}
  h \coloneqq \ui{[K]} - \uo{[K]}.
\end{equation*}
Then we can write
\begin{equation*}
  E^{\text{match}}=E_1 + E_2 + E_3,
\end{equation*}
for
\begin{equation}
  \label{eq:E-match-structure}
  \begin{aligned}
    E_1 &\coloneqq \theta E^{\text{in}} + (1 - \theta) E^{\text{out}}\\
    E_2 &\coloneqq \big(-\partial_t \theta + \nu \partial_x^2 \theta - \uo{[K]} \partial_x \theta - \theta h\partial_x\theta\big) h,\\
    E_3 &\coloneqq \big(2 \nu \partial_x\theta + \theta(1 - \theta)h\big) \partial_xh.
  \end{aligned}
\end{equation}

To bound $E$, we must first control the derivatives of $\theta$.
\begin{lemma}
  \label{lem:size-der-theta}
  Define $\theta$ by \eqref{eq:theta-def}.
  Then for all $i, j \in \Z_{\geq 0}$, we have
  \begin{equation}
    \label{eq:theta-derivs}
    |\partial_t^i \partial_x^j \theta| \lesssim_{i, j} \nu^{-(2i + 3j)/5}.
  \end{equation}
\end{lemma}
\begin{proof}
  Since $\theta$ is locally constant outside $M$, we may assume that $\nu^{1/5} \leq \dist \leq 2 \nu^{1/5}$.
  Using \eqref{eq:theta-def} and the $1$-homogeneity of $\dist$, we write
  \begin{equation*}
    \theta(t, x) = \vartheta\big(\dist(\nu^{-2/5}t, \nu^{-3/5}x)\big).
  \end{equation*}
  Therefore
  \begin{equation}
    \label{eq:deriv-commute}
    (\partial_t^i \partial_x^j \theta)(t, x) = \nu^{-(2i + 3j)/5} \big[\partial_t^i \partial_x^j (\vartheta \circ \dist)\big]|_{(\nu^{-2/5}t, \nu^{-3/5}x)}.
  \end{equation}
  Since $\dist(t, x) \asymp \nu^{1/5}$, $1$-homogeneity yields $\dist(\nu^{-2/5}t, \nu^{-3/5}x) = \nu^{-1/5}\dist(t, x) \asymp 1$.
  The derivatives of $\vartheta \circ \dist$ are bounded when $\dist \asymp 1$, so \eqref{eq:theta-derivs} follows from \eqref{eq:deriv-commute}.
\end{proof}
We can now control $E$ and $\partial_x E$ pointwise.
\begin{proposition}
  \label{prop:error-pointwise}
  For all $K \in \Z_{\geq 0}$,
  \begin{equation*}
    \abs{E} \lesssim_K
    \begin{cases}
      \big(\nu \dist^{-4}\big)^{K + 1} \dist^{-1} & \text{in } O \cup M,\\
      \big(\dist \vee \nu^{1/4}\big)^K & \text{in } I \cup M
    \end{cases}
  \end{equation*}
  and
  \begin{equation*}
    \abs{\partial_x E} \lesssim_K
    \begin{cases}
      \big(\nu \dist^{-4}\big)^{K + 1} \dist^{-4} & \text{in } O \cup M,\\
      \big(\dist \vee \nu^{1/4}\big)^{K - 3} & \text{in } I \cup M.
    \end{cases}
  \end{equation*}
\end{proposition}
\begin{proof}
  We first consider $E^{\text{out}}$ given by \eqref{eq:E-out}.
  By Corollary~\ref{cor:uk-der},
  \begin{equation*}
    \abs{\partial_x(\uo{k'}\uo{k''})} \lesssim_{k',k''} \dist^{-4(k' + k'') - 1} \And \abs{\partial_x^2 \uo{K}} \lesssim_K \dist^{-4K - 5}.
  \end{equation*}
  By \eqref{eq:E-out},
  \begin{equation*}
    \abs{E^{\text{out}}} \lesssim_K \sum_{\substack{1 \leq k',k'' \leq K\\ k' + k'' > K}} \big(\nu \dist^{-4}\big)^{k' + k''} \dist^{-1} + \big(\nu \dist^{-4}\big)^{K + 1} \dist^{-1}.
  \end{equation*}
  On $O \cup M$, $\nu \dist^{-4} \leq \nu^{1/5} \ll 1$, so the dominant terms in the sum are those with $k' + k'' = K + 1$.
  Hence
  \begin{equation*}
    \abs{E^{\text{out}}} \lesssim_K \big(\nu \dist^{-4}\big)^{K + 1} \dist^{-1} \quad \text{in } O \cup M.
  \end{equation*}
  After a nearly identical argument for the derivative, we have the pair of estimates
  \begin{equation}
    \label{eq:E-out-est}
    \abs{E^{\text{out}}} \lesssim_K \big(\nu \dist^{-4}\big)^{K + 1} \dist^{-1} \And \abs{\partial_x E^{\text{out}}} \lesssim_K \big(\nu \dist^{-4}\big)^{K + 1} \dist^{-4} \quad \text{in } O \cup M.
  \end{equation}
  Next, consider $E^{\text{in}}$ given by \eqref{eq:E-in}.
  By Corollary~\ref{cor:ul-size}, we have
  \begin{equation*}
    \abs{\partial_x(\ui{\ell'}\ui{\ell''})} \lesssim_{\ell',\ell''} \big(\dist \vee \nu^{1/4}\big)^{\ell' + \ell'' - 1}.
  \end{equation*}
  By \eqref{eq:E-in},
  \begin{equation*}
    \abs{E^{\text{in}}} \lesssim_K \sum_{\substack{1 \leq \ell',\ell'' \leq K\\ \ell' + \ell'' > K}} \big(\dist \vee \nu^{1/4}\big)^{\ell' + \ell'' - 1}.
  \end{equation*}
  On $I \cup M$, $\dist \leq 2 \nu^{1/5} \ll 1$, so the dominant terms in the sum are those with $\ell' + \ell'' = K + 1$.
  Hence
  \begin{equation*}
    \abs{E^{\text{in}}} \lesssim_K \big(\dist \vee \nu^{1/4}\big)^K \quad \text{in } I \cup M.
  \end{equation*}
  Again, an analogous argument for the derivative yields
  \begin{equation}
    \label{eq:E-in-est}
    \abs{E^{\text{in}}} \lesssim_K \big(\dist \vee \nu^{1/4}\big)^K \And \abs{\partial_x E^{\text{in}}} \lesssim_K \big(\dist \vee \nu^{1/4}\big)^{K - 3} \quad \text{in } I \cup M.
  \end{equation}  

  Finally, we consider $E = E_1 + E_2 + E_3$ given by \eqref{eq:E-match-structure} on $M$.
  We observe that the upper bounds in \eqref{eq:E-out-est} and \eqref{eq:E-in-est} agree up to a constant in $M$, where $\dist \asymp \nu^{1/5}$.
  Thus Lemma~\ref{lem:size-der-theta}, \eqref{eq:E-out-est}, and \eqref{eq:E-in-est} imply that
  \begin{equation}
    \label{eq:E1-est}
    \abs{E_1} \lesssim_K \nu^{K/5} \And \abs{\partial_x E_1} \lesssim_K \nu^{(K - 3)/5} \quad \text{in } M.
  \end{equation}
  To control $E_2$ and $E_3$ in \eqref{eq:E-match-structure}, we use Proposition~\ref{prop:matching} to bound the mismatch $h = \ui{[K]} - \uo{[K]}$:
  \begin{equation}
    \label{eq:h-est}
    \abs{\partial_x^i h} \lesssim_{i, K} \nu^{(-3i + K + 2)/5} \quad \text{in } M.
  \end{equation}
  Also, Corollary~\ref{cor:uk-der} implies that
  \begin{equation}
    \label{eq:u-out-partial-size}
    \abs{\partial_x^i\uo{[K]}} \lesssim_{i, K} \nu^{(-3i + 1)/5} \quad \text{in } M.
  \end{equation}
  Combining Lemma~\ref{lem:size-der-theta}, \eqref{eq:h-est}, and \eqref{eq:u-out-partial-size}, we can check that
  \begin{equation*}
    \abs{E_2} \lesssim_K \nu^{K/5}, \quad \abs{\partial_x E_2} \lesssim_K \nu^{(K - 3)/5}, \quad \abs{E_3} \lesssim_K \nu^{(K + 1)/5}, \quad \abs{\partial_x E_3} \lesssim_K \nu^{(K - 2)/5}
  \end{equation*}
  in $M$.
  Incorporating \eqref{eq:E1-est}, we see that $\abs{E} \lesssim_K \nu^{K/5}$ and $\abs{\partial_x E} \lesssim_K \nu^{(K-3)/5}$ in $M$.
  The proposition follows.
\end{proof}
These pointwise bounds imply $L^2$ estimates.
\begin{lemma}
  \label{lem:error-L2}
  For all $K \in \Z_{\geq 2}$ and $t \in [t_0, 0)$,
  \begin{equation*}
    \norm{E}_{L_x^2}^2\!(t) \lesssim_K \nu^{2(K + 1)} \big(\abs{t} \vee \nu^{2/5}\big)^{-(8K + 7)/2}
  \end{equation*}
  and
  \begin{equation*}
    \norm{\partial_x E}_{L_x^2}^2\!(t) \lesssim_K \nu^{2(K + 1)} \big(\abs{t} \vee \nu^{2/5}\big)^{-(8K + 13)/2}.
  \end{equation*}
\end{lemma}
\begin{proof}
  First suppose $t \leq -\nu^{2/5}$.
  Then the time-slice $\{t\} \times \R$ lies in $O \cup M$, so by Proposition~\ref{prop:error-pointwise} we have
  \begin{equation}
    \label{eq:error-early-slice}
    \abs{E} \lesssim_K \big(\nu \dist^{-4}\big)^{K + 1} \dist^{-1} \And \abs{\partial_x E} \lesssim_K \big(\nu \dist^{-4}\big)^{K + 1} \dist^{-4} \quad \text{on } \{t\} \times \R.
  \end{equation}
  Given $p > 3$, Lemma~\ref{lem:cubic} yields
  \begin{equation*}
    \int_{\R} \dist(t, x)^{-p} \d x \lesssim \int_0^{\abs{t}^{3/2}} \abs{t}^{-p/2} \d x + \int_{\abs{t}^{3/2}}^\infty \abs{x}^{-p/3} \d x \lesssim_p \abs{t}^{-(p - 3)/2}.
  \end{equation*}
  With $p = 8K + 10$ or $8K + 16$, \eqref{eq:error-early-slice} yields
  \begin{equation*}
    \norm{E}_{L_x^2}^2\!(t) \lesssim_K \nu^{2(K + 1)} \abs{t}^{-(8K + 7)/2} \And \norm{\partial_xE}_{L_x^2}^2\!(t) \lesssim_K \nu^{2(K + 1)} \abs{t}^{-(8K + 13)/2},
  \end{equation*}
  as desired.

  Now suppose $t \in (-\nu^{2/5}, 0)$.
  Define $x_M(t) > 0$ by the condition
  \begin{equation*}
    \dist\big(t, x_M(t)\big) = 2 \nu^{1/5}.
  \end{equation*}
  Then $(t, x) \in O$ if $\abs{x} > x_M(t)$ and $(t, x) \in I \cup M$ otherwise.
  Using $\abs{t} < \nu^{2/5}$ and Lemma~\ref{lem:cubic}, we can check that
  \begin{equation}
    \label{eq:slice-boundary}
    x_M(t) \asymp \nu^{3/5}.
  \end{equation}
  We write
  \begin{equation*}
    \norm{E}_{L_x^2}^2\!(t) = \int_{\{\abs{x} > x_M\}} \abs{E}^2 \d x + \int_{\{\abs{x} \leq x_M\}} \abs{E}^2 \d x \eqqcolon \m{I}_1 + \m{I}_2.
  \end{equation*}
  The first integral $\m{I}_1$ only involves points in $O$, so \eqref{eq:error-early-slice} holds there as well.
  Thus Lemma~\ref{lem:cubic} and \eqref{eq:slice-boundary} imply
  \begin{equation*}
    \m{I}_1 \lesssim_K \nu^{2(K + 1)}\int_{x_M}^\infty \dist^{-(8K + 10)} \d x \lesssim_K \nu^{2(K + 1)} \int_{x_M}^\infty \abs{x}^{-(8K + 10)/3} \d x \lesssim_K \nu^{(2K + 3)/5}.
  \end{equation*}
  On the other hand, Proposition~\ref{prop:error-pointwise} yields
  \begin{equation*}
    \m{I}_2 \lesssim \int_0^{x_M} (\dist \vee \nu^{1/4})^{2K} \d x \lesssim_K \nu^{K/2}x_M + \int_0^{x_M} \abs{x}^{2K/3} \d x \lesssim_K \nu^{(2K + 3)/5}.
  \end{equation*}
  Collecting these integral bounds, we find
  \begin{equation*}
    \norm{E}_{L_x^2}^2\!(t) \lesssim_K \nu^{(2K + 3)/5}
  \end{equation*}
  when $\abs{t} < \nu^{2/5}$, as desired.
  The argument for $\partial_x E$ is similar, so we omit the details.
  We only note one wrinkle: the inner integral for $\partial_x E$ involves a term of the form
  \begin{equation*}
    \int_0^{x_M} \abs{x}^{2(K - 3)/3} \d x.
  \end{equation*}
  We have assumed that $K \geq 2$ so that the integral is finite.
  When $K = 0$ or $1$, different exponents arise in the lemma.
\end{proof}

\subsection{Derivative bounds}
To deploy the energy estimates in Lemma~\ref{lem:eng-est-v}, we must also control derivatives of the approximate solution.
\begin{proposition}
  \label{prop:est-der-u0}
  For all $K \in \Z_{\geq 0}$ and $t \in [t_0, 0)$,
  \begin{equation}
    \label{eq:approx-slope}
    \big\|\partial_x \uapp_{[K]}\big\|_{L^\infty_x}(t) =
    \begin{cases}
      \abs{t}^{-1} + \m{O}_K\left(\abs{t}^{-1/2} + \nu \abs{t}^{-3}\right) & \text{if } \abs{t} \geq \nu^{1/2},\\
      \m{O}_K(\nu^{-1/2}) & \text{if } \abs{t} < \nu^{1/2}
    \end{cases}
  \end{equation}
  and
  \begin{equation}
    \label{eq:approx-second}
    \big\|\partial_x^2 \uapp_{[K]}\big\|_{L^\infty_x}(t) \lesssim_K \left(\abs{t} \vee \nu^{1/2}\right)^{-5/2}.
  \end{equation}
\end{proposition}
\begin{proof}
  First, we claim that
  \begin{equation}
    \label{eq:uapp-der-est}
    \big|\partial_x \uapp_{[K]} - \partial_x \cub\big| \lesssim_K \dist^{-1} + \nu \dist^{-6} \quad \text{if } \dist \geq \nu^{1/4}.
  \end{equation}
  By \eqref{eq:u0-out-precise}, we have
  \begin{equation}
    \label{eq:u0-out-est}
    \big|\partial_x \uo{0} - \partial_x\cub\big| \lesssim \dist^{-1}.
  \end{equation}
  Also, Corollary~\ref{cor:uk-der} implies that
  \begin{equation*}
    \big|\nu^k \partial_x \uo{k}\big| \lesssim_k \nu^k \dist^{-4k - 2} = (\nu^k \dist^{-4k - 1})\dist^{-1}.
  \end{equation*}
  When $\dist \gtrsim \nu^{1/5}$ and $k \geq 1$, we have $\nu^k \dist^{-4k - 1} \lesssim 1$.
  Hence
  \begin{equation}
    \label{eq:outer-der-est}
    \big|\nu^k \partial_x \uo{k}\big| \lesssim_k \dist^{-1} \quad \text{on } O \cup M.
  \end{equation}
  Combining this with \eqref{eq:u0-out-est}, we find
  \begin{equation*}
    \big|\partial_x \uo{[K]} - \partial_x \cub\big| \lesssim_K \dist^{-1} \quad \text{on } O \cup M.
  \end{equation*}
  In particular, this confirms \eqref{eq:uapp-der-est} on $O$.

  Next, we apply \eqref{eq:ul-est} and find
  \begin{equation*}
    \abs{\partial_x\ui{0} - \partial_x \cub} \lesssim \nu \dist^{-6} \quad \text{if } \dist \geq \nu^{1/4}.
  \end{equation*}
  Also, Corollary~\ref{cor:ul-size} yields $\abs{\partial_x \ui{\ell}} \lesssim_\ell \dist^{-1}$ when $\dist \lesssim 1$ and $\ell \geq 1$.
  Combining these observations, we have
  \begin{equation}
    \label{eq:inner-der-est}
    \big|\partial_x\ui{[K]} - \partial_x \cub\big| \lesssim_K \dist^{-1} + \nu \dist^{-6} \quad \text{if } \nu^{1/4} \leq \dist \leq 1.
  \end{equation}
  This confirms \eqref{eq:uapp-der-est} on $I$.
  Finally, on $M$ we have
  \begin{equation}
    \label{eq:matching-der-structure}
    \partial_x \uapp_{[K]} = \theta \partial_x \ui{[K]} + (1 - \theta) \partial_x\uo{[K]} + (\partial_x\theta) \big(\ui{[K]} - \uo{[K]}\big).
  \end{equation}
  By \eqref{eq:outer-der-est} and \eqref{eq:inner-der-est},
  \begin{equation}
    \label{eq:convex-combo}
    \big|\theta \partial_x \ui{[K]} + (1 - \theta) \uo{[K]} - \partial_x \cub\big| \lesssim_K \dist^{-1} + \nu \dist^{-6}.
  \end{equation}
  Moreover, Proposition~\ref{prop:matching} and Lemma~\ref{lem:size-der-theta} imply that
  \begin{equation}
    \label{eq:cutoff-der}
    \abs{(\partial_x\theta) \big(\ui{[K]} - \uo{[K]}\big)} \lesssim_K \nu^{(K - 1)/5} \leq \nu^{-1/5} \lesssim \dist^{-1}.
  \end{equation}
  Combining \eqref{eq:convex-combo} and \eqref{eq:cutoff-der}, \eqref{eq:matching-der-structure} yields
  \begin{equation*}
    \big|\partial_x \uapp_{[K]} - \partial_x \cub\big| \lesssim_K \dist^{-1} + \nu \dist^{-6}.
  \end{equation*}
  We have thus verified \eqref{eq:uapp-der-est}.
  Taking a supremum in $x$ and using \eqref{eq:u00-x} and ${\dist^{-1} \leq \abs{t}^{-1/2}}$ we obtain
  \begin{equation*}
    \big\|\partial_x \uapp_{[K]}\big\|_{L_x^\infty}(t) = \abs{t}^{-1} + \m{O}_K\left(\abs{t}^{-1/2} + \nu \abs{t}^{-3}\right) \For \abs{t} \geq \nu^{1/2}.
  \end{equation*}
  Now if $\dist \leq \nu^{1/4}$, $\uapp_{[K]}=\ui{[K]}$ and Corollary~\ref{cor:ul-size} implies that
  \begin{equation}
    \label{eq:der-sanctum}
    \big|\partial_x \ui{[K]}\big| \lesssim_K \nu^{-1/2}.
  \end{equation}
  Also, $\abs{\partial_x \cub} \lesssim \dist^{-2} \leq \nu^{-1/2}$ if $\dist \geq \nu^{1/4}$.
  Combining this observation with \eqref{eq:uapp-der-est} and \eqref{eq:der-sanctum}, we find
  \begin{equation*}
    \big\|\partial_x \uapp_{[K]}\big\|_{L_x^\infty}(t) \lesssim_K \nu^{-1/2} \For \abs{t} \leq \nu^{1/2}.
  \end{equation*}
  We have thus proven \eqref{eq:approx-slope}.

  We now turn to $\partial_x^2 \uapp_{[K]}$.
  Corollary~\ref{cor:uk-der} yields
  \begin{equation}
    \label{eq:outer-second}
    \big|\partial_x^2 \uo{[K]}\big| \lesssim_K \dist^{-5} \quad \text{on } O \cup M,
  \end{equation}
  while Corollary~\ref{cor:ul-size} implies
  \begin{equation}
    \label{eq:inner-second}
    \big|\partial_x^2 \ui{[K]}\big| \lesssim_K \big(\dist \vee \nu^{1/4}\big)^{-5} \quad \text{on } I \cup M.
  \end{equation}
  Finally, Proposition~\ref{prop:matching} and Lemma~\ref{lem:size-der-theta} imply that
  \begin{equation}
    \label{eq:matching-second}
    \big|\partial_x^2 \uapp_{[K]}\big| \lesssim_K \nu^{-1} \quad \text{on } M.
  \end{equation}
  Combining \eqref{eq:outer-second}--\eqref{eq:matching-second} and taking the supremum in $x,$ we obtain \eqref{eq:approx-second}.
\end{proof}

\subsection{Bootstrapping}
We now employ a bootstrap argument to show that the solution $v$ of \eqref{eq:diff} remains suitably small.
Ultimately, we wish to control $\norm{v}_{L_x^\infty}$.
To do so, we will interpolate between \eqref{eq:eng-est} and \eqref{eq:eng-est-der} in Lemma~\ref{lem:eng-est-v}.
However, the exponential in \eqref{eq:eng-est-der} depends on $\norm{v}_{L_x^\infty}$ itself.
We therefore bootstrap on $\norm{v}_{L_x^\infty}$.
Fix ${\eps \in (0, 1/100)}$.
We will work to prove:
\begin{equation}
  \label{eq:boot-assumptions}
  \tag{B}
  \norm{v}_{L^\infty_x}\!(t)\leq \nu^{-\eps + K + 1} \big(\abs{t} \vee \nu^{2/5}\big)^{-2K-3/4} \big(\abs{t} \vee \nu^{1/2}\big)^{-2}.
\end{equation}
More plainly:
\begin{equation*}
  \norm{v}_{L^\infty_x}\!(t)\leq
  \begin{cases}
    \nu^{ - \eps + K + 1} \abs{t}^{-2K-11/4} & \quad\text{if }t\in [t_0,-\nu^{2/5}),\\
    \nu^{ - \eps + (2K + 7)/10} \abs{t}^{-2} &\quad \text{if }t\in (-\nu^{2/5},-\nu^{1/2}],\\
    \nu^{ - \eps + (2K - 3)/10} &\quad \text{if }t\in (-\nu^{1/2},0).
  \end{cases}
\end{equation*}
The following proposition is our essential tool; it states that the bootstrap hypothesis is \emph{self-improving}.
\begin{proposition}
  \label{prop:boot}
  Fix $K \in \Z_{\geq 3}$ and $\eps \in (0, 1)$.
  There exists $\smallbar{\nu}(K, \eps) \in (0, 1/100)$ such that the following holds for all $\nu \in (0, \smallbar\nu]$.
  Suppose there exists $t_{\mathrm{B}} \in (t_0, 0)$ such that the solution $v$ of \eqref{eq:diff} satisfies \eqref{eq:boot-assumptions} on $[t_0, t_{\mathrm{B}}]$.
  Then for all $t\in [t_0,t_{\mathrm{B}})$,
  \begin{equation}
    \label{eq:boot-improved}
    \norm{v}_{L^\infty_x}\!(t)\leq \nu^{-\eps/2 + K + 1} \big(\abs{t} \vee \nu^{2/5}\big)^{-(8K + 3)/4} \big(\abs{t} \vee \nu^{1/2}\big)^{-2}.
  \end{equation}
\end{proposition}
\noindent
That is, we can shave $\nu^{-\eps/2}$ off the bootstrap assumption \eqref{eq:boot-assumptions}.

To prove Proposition~\ref{prop:boot}, we first control $v$ in $L_x^2$ and $H_x^1$.
\begin{lemma}
  \label{lem:L2-est-v}
  Fix $K \in \Z_{\geq 2}$ and $\eps_0 \in (0, 1)$.
  There exists $\nu_0(K, \eps_0) > 0$ such that for all $\nu \in (0, \nu_0]$ and $t \in [t_0,0)$,
  \begin{equation}
    \label{eq:L2-improved}
    \norm{v}^2_{L^2_x}\!(t) \leq \nu^{-\eps_0/2 + 2 (K + 1)} \big(\abs{\f t} \vee \nu^{2/5}\big)^{-(8K + 3)/2} \left(\abs{t} \vee \nu^{1/2}\right)^{-1}.
  \end{equation}
\end{lemma}
\begin{proof}
  Fix a small exponent $\delta \in (0, 1)$ to be determined later.
  Applying Lemma~\ref{lem:eng-est-v} with $\eta = 1 - \delta$, we find
  \begin{equation}
    \label{eq:eng-est-simple}
    \norm{v}^2_{L^2_x}\!(\f t) \lesssim_\delta \left(\int_{t_0}^{\f t} \abs{t}^{1 - \delta} \norm{E}_{L_x^2}^2\!(t) \d t\right) \exp\left(\int_{t_0}^{\f t} \norm{\partial_x \uapp}_{L_x^\infty}\!(t) \d t\right)
  \end{equation}
  for all $\f t \in (t_0, 0)$.
  If we integrate Lemma~\ref{lem:error-L2} in time, we obtain
  \begin{equation}
    \label{eq:forcing-int}
    \int_{t_0}^{\f t} \abs{t}^{1 - \delta} \norm{E}_{L_x^2}^2\!(t) \d t \lesssim_{K} \nu^{2 (K + 1)} \big(\abs{\f t} \vee \nu^{2/5}\big)^{-(8K + 3)/2 - \delta}.
  \end{equation}
  Similarly, Proposition~\ref{prop:est-der-u0} yields
  \begin{equation}
    \label{eq:exp-deriv-int}
    \int_{t_0}^{\f t} \norm{\partial_x \uapp}_{L_x^\infty}\!(t) \d t \leq \log\left(\abs{t} \vee \nu^{1/2}\right)^{-1} + \m{O}_{K, \delta}\left(1\right).
  \end{equation}
  Combining \eqref{eq:eng-est-simple}--\eqref{eq:exp-deriv-int}, we find
  \begin{equation*}
    \norm{v}^2_{L^2_x}\!(\f t) \leq C_0(K, \delta) \nu^{2 (K + 1)} \big(\abs{\f t} \vee \nu^{2/5}\big)^{-(8K + 3)/2 - \delta} \left(\abs{t} \vee \nu^{1/2}\right)^{-1}
  \end{equation*}
  for some $C_0(K, \delta) > 0$.
  Now choose $\delta = 5\eps_0/8$.
  Then
  \begin{equation*}
    \norm{v}^2_{L^2_x}\!(\f t) \leq \nu^{\eps_0/4} C_0(K, \delta) \nu^{-\eps_0/2 + 2 (K + 1)} \big(\abs{\f t} \vee \nu^{2/5}\big)^{-(8K + 3)/2} \left(\abs{t} \vee \nu^{1/2}\right)^{-1}.
  \end{equation*}
  If we define $\nu_0(K, \eps_0) \coloneqq C_0(K, \delta)^{-4/\eps_0}$ and take $\nu \in (0, \nu_0]$, then $\nu^{\eps_0/4} C_0(K, \delta)\leq 1$ and \eqref{eq:L2-improved} follows.
\end{proof}
To control $\partial_x v$ in the same manner, we need \emph{a priori} control on $\norm{v}_{L_x^\infty}$.
This is the role of the bootstrap hypothesis \eqref{eq:boot-assumptions}.
\begin{lemma}
  \label{lem:L2-est-v-der}
  Fix $K \in \Z_{\geq 3}$ and $\eps \in (0, 1/100)$.
  There exists $\nu_1(K, \eps) \in (0, 1]$ such that the following holds for all $\nu \in (0, \nu_1]$.
  Let $t_{\mathrm{B}}$ be as in Proposition~\textup{\ref{prop:boot}}.
  Then for all $t\in [t_0,t_{\mathrm{B}})$,
  \begin{equation}
    \label{eq:H1-improved}
    \norm{\partial_x v}_{L^2_x}^2\!(t)\leq \nu^{-\eps/2 + 2(K + 1)} \big(\abs{\f t} \vee \nu^{2/5}\big)^{-(8K + 3)/2} \big(\abs{\f t} \vee \nu^{1/2}\big)^{-7}.
  \end{equation}
\end{lemma}
\begin{proof}
  Fix $\delta \in (0, 1/200)$ to be determined later.
  By Lemma~\ref{lem:eng-est-v},
  \begin{equation}
    \label{eq:eng-est-der-simple}
    \begin{aligned}
      \norm{\partial_x v}_{L_x^2}^2\!(\f t) \lesssim_\delta &\int_{t_0}^{\f t} \abs{t}^{1 - \delta} \left[\norm{\partial_xE}_{L_x^2}^2 + \norm{\partial_x^2 \uapp}_{L_x^\infty}^2 \norm{v}_{L_x^2}^2\right]\dn t\\
      &\hspace{2cm}\times\exp\left(\int_{t_0}^{\f t} \left[3 \norm{\partial_x \uapp}_{L_x^\infty} + \nu^{-1}\norm{v}_{L_x^\infty}^2\right]\dn t\right)
    \end{aligned}
  \end{equation}
  for all $\f t \in (t_0, 0)$.
  We control the four terms separately.
  Integrating Lemma~\ref{lem:error-L2} in time, we find
  \begin{equation}
    \label{eq:forcing-der-int}
    \int_{t_0}^{\f t} \abs{t}^{1 - \delta} \norm{\partial_xE}_{L_x^2}^2\!(t) \d t \lesssim_K \nu^{2(K + 1)} \big(\abs{\f t} \vee \nu^{2/5}\big)^{-(8K + 9)/2-\delta}.
  \end{equation}
  Next, we take $\eps_0 = 2\delta$ in Lemma~\ref{lem:L2-est-v} to control $\norm{v}_{L_x^2}$.
  In combination with Proposition~\ref{prop:est-der-u0}, we obtain
  \begin{equation*}
    \norm{\partial_x^2 \uapp}_{L_x^\infty}^2\!(t) \norm{v}_{L_x^2}^2\!(t) \lesssim_{K, \delta} \nu^{-\delta} \nu^{2(K + 1)} \big(\abs{t} \vee \nu^{2/5}\big)^{-(8K + 3)/2} \big(\abs{t} \vee \nu^{1/2}\big)^{-6}.
  \end{equation*}
  Therefore
  \begin{equation}
    \label{eq:v-L2-int}
    \begin{aligned}
      \int_{t_0}^{\f t} \abs{t}^{1 - \delta} &\norm{\partial_x^2 \uapp}_{L_x^\infty}^2\!(t) \norm{v}_{L_x^2}^2\!(t) \d t\\
      &\lesssim_{K, \delta} \nu^{-\delta} \nu^{2(K + 1)} \big(\abs{\f t} \vee \nu^{2/5}\big)^{-(8K + 3)/2} \big(\abs{\f t} \vee \nu^{1/2}\big)^{-4-\delta}.
    \end{aligned}
  \end{equation}
  We observe that the right side of \eqref{eq:v-L2-int} is larger than that of \eqref{eq:forcing-der-int}.
  It thus appears that $(\partial_x E)^2$ is negligible compared to $(\partial_x^2 \uapp)^2 v^2$ in \eqref{eq:eng-est-der} and \eqref{eq:eng-est-der-simple}.

  Next, we control the terms in the exponential in \eqref{eq:eng-est-der-simple}.
  Using Proposition~\ref{prop:est-der-u0}, we can check that
  \begin{equation}
    \label{eq:der-exp-deriv}
    \exp\left(3\int_{t_0}^{\f t}\norm{\partial_x \uapp}_{L_x^\infty}\!(t) \d t\right) \lesssim_{K} \big(\abs{\f t} \vee \nu^{1/2}\big)^{-3}.
  \end{equation}
  Finally, the bootstrap hypothesis \eqref{eq:boot-assumptions} yields
  \begin{equation*}
    \nu^{-1}\int_{t_0}^{\f t} \norm{v}_{L_x^\infty}^2\!(t) \d t \leq \nu^{-1}\int_{t_0}^{t_\mathrm{B}} \norm{v}_{L_x^\infty}^2\!(t) \d t \lesssim_{K} \nu^{(4K - 11)/10 - 2 \eps}
  \end{equation*}
  when $\f{t} \in [t_0, t_{\mathrm{B}})$.
  Because $K \geq 3$ and $\eps < 1/100$, we have $(4K - 11)/10 - 2 \eps > 0$.
  Since $\nu \leq 1$,
  \begin{equation}
    \label{eq:der-exp-v}
    \exp\left(\nu^{-1}\int_{t_0}^{\f t} \norm{v}_{L_x^\infty}^2\!(t) \d t \right) \lesssim_{K} 1.
  \end{equation}
  Using \eqref{eq:forcing-der-int}--\eqref{eq:der-exp-v} in \eqref{eq:eng-est-der-simple}, we obtain
  \begin{align*}
    \norm{\partial_x v}_{L_x^2}^2\!(\f t) &\lesssim_{K, \delta} \nu^{-\delta} \nu^{2(K + 1)} \big(\abs{\f t} \vee \nu^{2/5}\big)^{-(8K + 3)/2} \big(\abs{\f t} \vee \nu^{1/2}\big)^{-7-\delta}\\
                                          &\lesssim_{K, \delta} \nu^{-3\delta/2} \nu^{2(K + 1)} \big(\abs{\f t} \vee \nu^{2/5}\big)^{-(8K + 3)/2} \big(\abs{\f t} \vee \nu^{1/2}\big)^{-7}.
  \end{align*}
  Now let $\delta \coloneqq \eps/6$.
  Then there exists a constant $C_1(K, \eps) > 0$ such that
  \begin{equation}
    \label{eq:H1-improved-prelim}
    \norm{\partial_x v}_{L_x^2}^2\!(\f t) \leq \nu^{\eps/4} C_1(K, \eps) \nu^{-\eps/2 + 2(K + 1)} \big(\abs{\f t} \vee \nu^{2/5}\big)^{-(8K + 3)/2} \big(\abs{\f t} \vee \nu^{1/2}\big)^{-7}.
  \end{equation}
  We therefore need $\nu \leq C_1(K, \eps)^{-4/\eps}$.
  Moreover, because we applied Lemma~\ref{lem:L2-est-v} with $\eps_0 = 2\delta$, we need $\nu \leq \nu_0(K, \eps/3)$.
  Therefore let
  \begin{equation*}
    \nu_1(K, \eps) \coloneqq \min\left\{C_1(K, \eps)^{-4/\eps}, \, \nu_0(K, \eps/3), \, 1\right\}.
  \end{equation*}
  Then for all $\nu \in (0, \nu_1]$ and $\f t \in (t_0, t_{\mathrm{B}})$, \eqref{eq:H1-improved-prelim} implies \eqref{eq:H1-improved}.
\end{proof}
Now, the self-improvement of the bootstrap follows from interpolation.
\begin{proof}[Proof of Proposition~\textup{\ref{prop:boot}}]
  Fix $K \in \Z_{\geq 3}$ and $\eps \in (0, 1/100)$.
  Assume there exists ${t_{\mathrm{B}} \in (t_0, 0)}$ such that the solution $v$ of \eqref{eq:diff} satisfies \eqref{eq:boot-assumptions} on $[t_0, t_{\mathrm{B}})$.
  
  By Morrey's inequality, $H^s(\R) \subset L^\infty(\R)$ for all $s > 1/2$.
  Thus if $s \in (1/2, 1)$, we have
  \begin{equation}
    \label{eq:interpolation}
    \norm{v}_{L_x^\infty} \lesssim_s \norm{v}_{L_x^2}^{1 - s}\norm{\partial_x v}_{L_x^2}^s + \norm{v}_{L_x^2}.
  \end{equation}
  We take $\eps_0 = \eps$ in Lemmas~\ref{lem:L2-est-v} and \ref{lem:L2-est-v-der}:
  \begin{equation}
    \label{eq:interpolation-size}
    \norm{v}_{L_x^2}^{1 - s}\!(t)\norm{\partial_x v}_{L_x^2}^s\!(t) \lesssim_{K, \eps} \nu^{-\eps/4 + K + 1} \big(\abs{t} \vee \nu^{2/5}\big)^{-(8K + 3)/4} \big(\abs{t} \vee \nu^{1/2}\big)^{-1/2-3s}
  \end{equation}
  for all $t \in [t_0, t_{\mathrm{B}})$.
  Take $s = 1/2 + \eps/12$.
  Then we have
  \begin{equation}
    \label{eq:exponent-choice}
    \big(\abs{t} \vee \nu^{1/2}\big)^{-1/2-3s} \leq \big(\abs{t} \vee \nu^{1/2}\big)^{-2 - \eps/4} \leq \nu^{-\eps/8} \big(\abs{t} \vee \nu^{1/2}\big)^{-2}.
  \end{equation}
  Combining \eqref{eq:interpolation}--\eqref{eq:exponent-choice}, we find
  \begin{equation}
    \label{eq:boot-improved-prelim}
    \norm{v}_{L_x^\infty}\!(t) \leq \nu^{\eps/8}C(K, \eps) \nu^{-\eps/2 + K + 1} \big(\abs{t} \vee \nu^{2/5}\big)^{-(8K + 3)/4} \big(\abs{t} \vee \nu^{1/2}\big)^{-2}
  \end{equation}
  for some $C(K, \eps) > 0$.
  Define
  \begin{equation*}
    \smallbar{\nu}(K, \eps) = \min \left\{C(K, \eps)^{-8/\eps},\, \nu_0(K, \eps),\, \nu_1(K, \eps)\right\}.
  \end{equation*}
  Then if $\nu \in (0, \smallbar{\nu}]$ and $t \in [t_0, t_{\mathrm{B}})$, \eqref{eq:boot-improved-prelim} implies \eqref{eq:boot-improved}.
\end{proof}
We are nearly in a position to prove our main theorem.
The last ingredient is time continuity in $L^\infty$:
\begin{lemma}
  \label{lem:short-time-existence}
  The norm $\norm{v}_{L_x^\infty}$ is continuous in time.
\end{lemma}
\noindent
This follows from standard parabolic theory; we omit the proof.

We can now verify the bootstrap assumption \eqref{eq:boot-assumptions} and prove Theorem~\ref{thm:main}.
\begin{proof}[Proof of Theorem~\textup{\ref{thm:main}}]
  Fix $K \in \Z_{\geq 0}$.
  We wish to control $u^\nu - \uapp_{[K]}$.
  However, the estimate \eqref{eq:boot-assumptions} is not sharp.
  We therefore work with a more accurate approximation $\uapp_{[K']}$ for some $K' > K$.
  We bound $v \coloneqq u^\nu - \uapp_{[K']}$ via \eqref{eq:boot-assumptions} and deduce \eqref{eq:pointwise-v-gen} from the triangle inequality.
  As we shall see, $K' = K + 4$ suffices.

  Fix $\eps = 1/200$ and note that $K' >  3$.
  Recalling $\smallbar{\nu}$ from Proposition~\ref{prop:boot}, fix $\nu \in (0, \smallbar{\nu}(K', \eps))$.
  In particular, $\nu < \smallbar{\nu} < 1$.
  Let $\text{\eqref{eq:boot-assumptions}}_{K'}$ denote the condition \eqref{eq:boot-assumptions} with $K'$ in place of $K$.
  Then we define
  \begin{equation*}
    t_{\mathrm{max}} \coloneqq \sup\big\{t \in [t_0,0) \mid v\textrm{ satisfies \eqref{eq:boot-assumptions}}_{K'} \text{ on } [t_0, t)\big\}.
  \end{equation*}
  Recall that $v(t_0,\anon) = 0$.
  By the continuity in Lemma~\ref{lem:short-time-existence}, $t_{\mathrm{max}}>t_0$.
  We will prove that $t_{\mathrm{max}} = 0$, so suppose for the sake of contradiction that $t_{\mathrm{max}} \in (t_0, 0)$.

  Define
  \begin{equation*}
    Z(t) \coloneqq \nu^{\eps-K' - 1} \big(\abs{t} \vee \nu^{2/5})^{(8K'+3)/4} \big(\abs{t} \vee \nu^{1/2}\big)^{2} \norm{v}_{L_x^\infty}\!(t),
  \end{equation*}
  so that $\text{\eqref{eq:boot-assumptions}}_{K'}$ is equivalent to $Z \leq 1$.
  By Proposition~\ref{prop:boot}, $Z \leq \nu^{\eps/2} < 1$ on $[t_0, t_{\mathrm{max}})$.
  But Lemma~\ref{lem:short-time-existence} implies that $Z$ is continuous on $[t_0, 0)$, so there exists $\tau > 0$ such that $Z \leq 1$ on $[t_0, t_{\mathrm{max}} + \tau)$.
  That is, $\text{\eqref{eq:boot-assumptions}}_{K'}$ holds on $[t_0, t_{\mathrm{max}} + \tau)$.
  This contradicts the definition of $t_{\mathrm{max}}$, so in fact $t_{\mathrm{max}} = 0$ and $\text{\eqref{eq:boot-assumptions}}_{K'}$ holds on $[t_0, 0)$.
  In particular, using $K' = K + 4$ and $\eps = 1/200$, we obtain
  \begin{equation}
    \label{eq:pointwise-prelim}
    \big\|u^\nu - \uapp_{[K']}\big\|_{L^\infty([t_0, 0) \times \R)} \leq C(K', \eps) \nu^{(2K' - 3)/10 - \eps} \leq C(K) \nu^{(K + 2)/5}.
  \end{equation}
  The difference $\uapp_{[K']} - \uapp_{[K]}$ should be dominated by the terms $u_{0, K+1}$ and $\nu^{K + 1}u_{K+1,0}$, which are of order $\nu^{(K + 2)/5}$ in $M$.
  We therefore expect
  \begin{equation}
    \label{eq:approx-diff}
    \big\|\uapp_{[K']} - \uapp_{[K]}\big\|_{L^\infty([t_0, 0) \times \R)} \lesssim_K \nu^{(K + 2)/5}.
  \end{equation}
  We can use Corollaries~\ref{cor:uk-der} and \ref{cor:ul-size} to confirm this reasoning.
  We omit the routine details.
  Recall that all constants in the paper implicitly depend on $\mr{u}$.
  Hence \eqref{eq:pointwise-v-gen} follows from \eqref{eq:pointwise-prelim} and \eqref{eq:approx-diff} when ${\nu \in (0, \smallbar{\nu})}$.
  
  When $\nu \in [\smallbar{\nu}, 1]$, \eqref{eq:diff} is uniformly parabolic.
  By standard parabolic estimates,
  \begin{equation*}
    \norm{v}_{L^\infty} \leq C(K) \leq C(K) \smallbar{\nu}(K + 4, 1/200)^{-(K + 2)/5} \nu^{(K + 2)/5}.
  \end{equation*}
  Absorbing the power of $\smallbar{\nu}$ in the constant, we obtain \eqref{eq:pointwise-v-gen} for all $\nu \in [\smallbar{\nu}, 1]$, and hence all $\nu \in (0, 1]$.
  \medskip

  We now turn to \eqref{eq:more-regular}.
  Fix $K = 6$ and $\eps = 1/200$, and let $\nu_1 \coloneqq \nu_1(6, 1/200)$.
  Suppose $\nu \in (0, \nu_1]$.
  Because $\eqref{eq:boot-assumptions}_{\text{6}}$ holds on $[t_0, 0)$, Lemmas~\ref{lem:L2-est-v} and \ref{lem:L2-est-v-der} imply that
  \begin{equation*}
    \big\|u^\nu - \uapp_{[6]}\big\|_{H_x^1} \leq \nu^{(4K - 21)/20 - \eps/4} \leq \nu^{1/10} \ForAll t \in [t_0, 0).
  \end{equation*}
  By Morrey's inequality,
  \begin{equation*}
    \sup_{t \in [t_0, 0)} \big\|u^\nu(t, \anon) - \uapp_{[6]}(t, \anon)\big\|_{\m{C}^{1/2}(\R)} \to 0 \quad \text{as } \nu \to 0.
  \end{equation*}
  Now, the difference $\uapp_{[6]} - \uapp_{[0]}$ should be dominated by $u_{0, 1}$ and $\nu u_{1, 0}$ in $M$.
  Using the homogeneity of these terms, it is straightforward to check that
  \begin{equation*}
    \norm{u_{0, 1}}_{\m{C}^{1/2}(M)} + \norm{\nu u_{1, 0}}_{\m{C}^{1/2}(M)} \lesssim \nu^{1/10}.
  \end{equation*}
  Rigorously, we can use Corollaries~\ref{cor:uk-der} and \ref{cor:ul-size} as well as Lemma~\ref{lem:size-der-theta} to show that
  \begin{equation*}
    \sup_{t \in [t_0, 0)}\big\|\uapp_{[6]}(t, \anon) - \uapp_{[0]}(t, \anon)\big\|_{\m{C}^{1/2}(\R)} \to 0 \quad \text{as } \nu \to 0.
  \end{equation*}
  Now, \eqref{eq:more-regular} follows from the triangle inequality.
  This completes the proof of Theorem~\ref{thm:main}.
\end{proof}

\appendix
\section{The leading inner term}
\label{sec:inner-term}
In this appendix, we examine the structure of the leading part of the inner expansion, namely $\Ui{0}$.
We show that $\Ui{0}$ admits an explicit representation via the Cole--Hopf transformation.
Moreover, the behavior of $\Ui{0}$ as $T \to \infty$ hints at a viscous shock emerging from the inner expansion when $t > 0$.
Thus, these calculations shed light on the structure of $u^\nu$ \emph{after} the shock forms.

Recall the inner coordinates $(T, X)$ defined in \eqref{eq:inner-var}.
In the proof of Proposition~\ref{prop:Ul}, we construct $\Ui{0}$ as the limit of solutions $U^{(n)}$ to the problem
\begin{equation}
  \label{eq:Burgers-early-approx}
  \partial_T U^{(n)} = \partial_X^2 U^{(n)} - U^{(n)} \partial_X U^{(n)}, \quad U^{(n)}(-n, \anon) = \Cub(-n, \anon).
\end{equation}
We recall that $\Cub$ denotes the rescaled cubic $\cub$ in the inner coordinates; see \eqref{eq:Cub}.
For convenience, we state this convergence as a lemma.
\begin{lemma}
  \label{lem:construction}
  As $n \to \infty$, $U^{(n)} \to \Ui{0}$ uniformly in $\R_- \times \R$.
\end{lemma}
\begin{proof}
  In the proof of Proposition~\ref{prop:Ul}, we establish locally uniform convergence along a subsequence.
  However, Proposition~\ref{prop:Ul} also states that $\Ui{0}$ is unique, so in fact the entire sequence converges.
  Moreover, the sub- and supersolutions used in the proof of Proposition~\ref{prop:Ul} imply that
  \begin{equation*}
    \big|\Ui{0} - U^{(n)}\big| \lesssim \Dist^{-3},
  \end{equation*}
  where $\Dist$ is the rescaled distance defined by \eqref{eq:Dist}.
  Because $\Dist^{-3} \to 0$ at infinity in $\R_- \times \R$, we can upgrade locally uniform convergence to uniform.
\end{proof}
We observe that $U^{(n)}$ satisfies viscous Burgers \eqref{eq:Burgers-early-approx} with unit viscosity.
We can explicitly solve this equation via the Cole--Hopf transformation.
Thus far, we have avoided Cole--Hopf so that our methods generalize to other scalar conservation laws \eqref{eq:SCL-viscous}.
However, there is no harm in deploying it here, because the leading part of the inner expansion is canonically viscous Burgers.
That is, every strictly convex conservation law becomes viscous Burgers to leading order in the inner expansion.

To see this, let $(w^\nu)_{\nu \geq 0}$ solve the viscous scalar conservation law \eqref{eq:SCL-viscous} with strictly convex flux $f$.
We first make a Galilean transformation to fix the inviscid shock formation at the origin with $w^0(0, 0) = 0$.
Thanks to this transformation, we can assume that $f(0) = f'(0) = 0$.
The inner expansion reflects the local structure of \eqref{eq:SCL-viscous} near the origin, where $\abs{w^\nu} \ll 1$.
Thus \eqref{eq:SCL-viscous} becomes
\begin{equation*}
  \partial_t w^\nu \approx \nu \partial_x^2 w^\nu - f''(0) w^\nu \partial_x w^\nu
\end{equation*}
near the origin.
Since $f''(0) > 0$, this is a reparameterization of viscous Burgers.
We note, however, that higher derivatives of $f$ lead to new terms in the PDE for $\Ui{\ell}$ for $\ell \geq 1$.

Using Cole--Hopf, we derive an explicit integral representation of $\Ui{0}$.
\begin{proposition}
  \label{prop:inner-explicit}
  For all $(T, X) \in \R^2$,
  \begin{equation}
    \label{eq:explicit}
    \Ui{0}(T, X) = -\frac
    {\displaystyle \int_{\R} \zeta \exp\left(\frac{1}{2} X \zeta + \frac{1}{4} T \zeta^2 - \frac{\coeff}{8}\zeta^4\right) \ds \zeta}
    {\displaystyle \int_{\R} \exp\left(\frac{1}{2} X \zeta + \frac{1}{4} T \zeta^2 - \frac{\coeff}{8}\zeta^4\right) \ds \zeta}\,.
  \end{equation}
\end{proposition}
\noindent
Here, $\coeff > 0$ is the constant in the cubic equation \eqref{eq:cubic} for $\cub$.
\begin{proof}
  We write $\Cub_n(X) \coloneqq \Cub(-n, X)$ and define its renormalized integral
  \begin{equation}
    \label{eq:HJE-init}
    \f{H}_n(X) \coloneqq \int_0^X \Cub_n(Z) \d Z + \log n.
  \end{equation}
  Let $H^{(n)}$ solve the following viscous Hamilton--Jacobi equation on $(-n, \infty) \times \R$:
  \begin{equation}
    \label{eq:HJE}
    \partial_T H^{(n)} = \partial_X^2 H^{(n)} - \frac{1}{2}\big(\partial_X H^{(n)}\big)^2, \quad H^{(n)}(-n, \anon) = \f{H}_n.
  \end{equation}
  Then we can differentiate \eqref{eq:HJE} to check that $\partial_X H^{(n)} = U^{(n)}$.
  Note that the renormalization constant $\log n$ in \eqref{eq:HJE-init} disappears under differentiation.
  This is a form of gauge invariance that we use to ensure that $H^{(n)}$ converges as $n \to \infty$.

  We now deploy the Cole--Hopf transform.
  Define
  \begin{equation*}
    \phi^{(n)}(T, X) \coloneqq \exp\left[-\frac{1}{2}H^{(n)}(T, X)\right].
  \end{equation*}
  Then $\phi^{(n)}$ solves the linear heat equation on $(-n, \infty) \times \R$:
  \begin{equation*}
    \partial_T \phi^{(n)} = \partial_X^2 \phi^{(n)}, \quad \phi^{(n)}(-n, \anon) = \e^{-\f{H}_n/2}.
  \end{equation*}
  We can thus express $\phi^{(n)}$ via the heat kernel:
  \begin{equation*}
    \phi^{(n)}(T, X) = \int_{\R} \exp\left[-\frac{(X - Z)^2}{4(T + n)} - \frac{1}{2} \f{H}_n(Z)\right] \frac{\dn Z}{\sqrt{4\pi(T + n)}}.
  \end{equation*}
  
  Fix $(T, X) \in \R^2$ and define
  \begin{equation*}
    \tau \coloneqq \frac{T}{n}, \quad \eta \coloneqq \frac{X}{n}, \And \zeta \coloneqq \frac{Z}{n},
  \end{equation*}
  so that $\tau, \eta = \m{O}_{T, X}\big(n^{-1}\big)$.
  The scaling properties of the inverse cubic \eqref{eq:cubic} imply that
  \begin{equation*}
    \f{H}_n(Z) = n^2 \f{H}_1\big(n^{-1/2}\zeta\big) + \log n,
  \end{equation*}
  so we have
  \begin{equation}
    \label{eq:phi-n-kernel}
    \phi^{(n)}(T, X) = \int_{\R} \exp\left[-n\frac{(\eta - \zeta)^2}{4(1 + \tau)} - \frac{1}{2}n^2\f{H}_1\big(n^{-1/2}\zeta\big)\right] \frac{\dn \zeta}{\sqrt{4\pi(1 + \tau)}}.
  \end{equation}
  Here, we have used the renormalization to produce an additional factor of $n^{-1/2}$.
  
  We now expand $\f{H}_1$ around $0$ by doing the same for its derivative $\Cub_1$.
  Applying the chain rule and Taylor's theorem to \eqref{eq:cubic}, we find
  \begin{equation*}
    \Cub_1(X) = -X + \coeff X^3 + \m{O}\big(\abs{X}^{5}\big).
  \end{equation*}
  Integrating, this becomes
  \begin{equation}
    \label{eq:HJE-quartic}
    \f{H}_1(X) = -\frac{1}{2}X^2 + \frac{\coeff}{4} X^4 + \m{O}\big(X^{6}\big).
  \end{equation}
  In particular, there exists $\delta > 0$ such that
  \begin{equation}
    \label{eq:HJE-quartic-bound}
    \f{H}_1(X) \geq -\frac{1}{2}X^2 + \frac{\coeff}{8} X^4 \quad \text{on } [-\delta, \delta].
  \end{equation}
  Moreover, \eqref{eq:cubic} implies that $\abs{\Cub_1(X)} < \abs{X}$ when $X \neq 0$.
  Since $\abs{\Cub_1(X)} = \smallO(\abs{X})$ as $\abs{X} \to \infty$, we have
  \begin{equation*}
    \sup_{[-\delta/2,\, \delta/2]^c} \frac{\abs{\Cub_1(X)}}{\abs{X}} < 1.
  \end{equation*}
  Integrating from $0$, there exists $\eps \in (0, 1)$ such that
  \begin{equation}
    \label{eq:HJE-quadratic-bound}
    \abs{\f{H}_1(X)} \leq \frac{1 - \eps}{2} X^2 \quad \text{on } [-\delta, \delta]^c.
  \end{equation}

  For convenience, let
  \begin{equation}
    \label{eq:integrand}
    \m{E}(\zeta; n, T, X) \coloneqq \frac{1}{\sqrt{4\pi(1 + \tau)}} \exp\left[-n\frac{(\eta - \zeta)^2}{4(1 + \tau)} - \frac{1}{2}n^2\f{H}_1\big(n^{-1/2}\zeta\big)\right]
  \end{equation}
  denote the integrand of \eqref{eq:phi-n-kernel}, recalling that $(\tau, \eta) = (T/n, X/n)$.
  Also, define
  \begin{equation*}
    I_1 \coloneqq \int_{\R} \m{E}(\zeta; n, T, X) \tbf{1}_{\{\abs{\zeta} \leq \delta n^{1/2}\}} \d \zeta \And I_2 \coloneqq \int_{\R} \m{E}(\zeta; n, T, X) \tbf{1}_{\{\abs{\zeta} > \delta n^{1/2}\}}\d \zeta,
  \end{equation*}
  so that \eqref{eq:phi-n-kernel} becomes $\phi^{(n)} = I_1 + I_2$.
  
  We first consider the main term $I_1$.
  Recall that $\eta$ and $\tau$ are small, so the first term in the exponential in \eqref{eq:integrand} becomes
  \begin{equation}
    \label{eq:Gaussian}
    -n\frac{(\eta - \zeta)^2}{4(1 + \tau)} = -\frac{n}{4}\zeta^2 + \frac{1}{2} X \zeta + \frac{1}{4} T \zeta^2 + \m{O}_{T, X}\big(n^{-1}(\zeta^2+1)\big).
  \end{equation}
  By \eqref{eq:HJE-quartic-bound},
  \begin{equation*}
    \frac{1}{2}n^2\f{H}_1\big(n^{-1/2}\zeta\big) \geq -\frac{n}{4}\zeta^2 + \frac{\coeff}{16}\zeta^4
  \end{equation*}
  when $\abs{\zeta} \leq \delta n^{1/2}$.
  Therefore
  \begin{equation*}
    -n\frac{(\eta - \zeta)^2}{4(1 + \tau)} - \frac{1}{2}n^2\f{H}_1\big(n^{-1/2}\zeta\big) \leq C - \frac{\coeff}{32} \zeta^4
  \end{equation*}
  for some constant $C(T, X) > 0$ independent of $\zeta$ and $n$ that we allow to change from line to line.
  That is,
  \begin{equation}
    \label{eq:majorant}
    \m{E}(\zeta; n, T, X) \tbf{1}_{\{\abs{\zeta} \leq \delta n^{1/2}\}} \lesssim_{T,X} \e^{-\coeff\zeta^4/32}.
  \end{equation}
  We use this as a majorant for dominated convergence.
  We can write \eqref{eq:HJE-quartic} as
  \begin{equation*}
    \frac{1}{2}n^2\f{H}_1\big(n^{-1/2}\zeta\big) = -\frac{n}{4}\zeta^2 + \frac{\coeff}{8}\zeta^4 + \m{O}_\zeta\big(n^{-1}\big).
  \end{equation*}  
  By \eqref{eq:Gaussian}, the leading term of the Gaussian kernel cancels that of the initial data and
  \begin{equation*}
    \lim_{n \to \infty} \m{E}(\zeta; n, T, X) = (4\pi)^{-1/2} \exp\left(\frac{1}{2} X \zeta + \frac{1}{4} T \zeta^2 - \frac{\coeff}{8}\zeta^4\right).
  \end{equation*}
  In light of \eqref{eq:majorant}, dominated convergence implies that
  \begin{equation}
    \label{eq:first-integral}
    \lim_{n \to \infty} I_1 = \int_{\R} \exp\left(\frac{1}{2} X \zeta + \frac{1}{4} T \zeta^2 - \frac{\coeff}{8}\zeta^4\right) \frac{\dn \zeta}{\sqrt{4\pi}}.
  \end{equation}

  Next, we control the remainder $I_2$.
  By \eqref{eq:HJE-quadratic-bound},
  \begin{equation*}
    -\frac{1}{2}n^2\f{H}_1\big(n^{-1/2}\zeta\big) \leq \frac{1 - \eps}{4} n\zeta^2
  \end{equation*}
  when $\abs{\zeta} > \delta n^{1/2}$.
  On the other hand, \eqref{eq:Gaussian} implies
  \begin{equation*}
    n\frac{(\eta - \zeta)^2}{4(1 + \tau)} \geq \frac{1 - \eps/2}{4} n\zeta^2 - Cn^{-1}
  \end{equation*}
  once $n \geq C$ for $C(T, X)$ sufficiently large.
  Therefore
  \begin{equation*}
    \m{E}(\zeta; n, T, X) \lesssim_{T, X} \exp\left[-\frac{\eps}{8}n\zeta^2\right]
  \end{equation*}
  when $\abs{\zeta} > \delta n^{1/2}$ and $n \geq C$.
  It follows that
  \begin{equation}
    \label{eq:second-integral}
    \lim_{n \to \infty} I_2 = 0.
  \end{equation}

  Combining \eqref{eq:first-integral} and \eqref{eq:second-integral}, we find
  \begin{equation}
    \label{eq:phi}
    \lim_{n \to \infty} \phi^{(n)}(T, X) = \int_{\R} \exp\left(\frac{1}{2} X \zeta + \frac{1}{4} T \zeta^2 - \frac{\coeff}{8}\zeta^4\right) \frac{\dn \zeta}{\sqrt{4\pi}} \eqqcolon \phi(T, X)
  \end{equation}
  pointwise in $(T, X)$.
  It is straightforward to check that this convergence is locally uniform in $(T, X)$.
  By parabolic regularity, we also have $\partial_X \phi^{(n)} \to \partial_X \phi$ locally uniformly as $n \to \infty$.
  Inverting Cole--Hopf, Lemma~\ref{lem:construction} yields
  \begin{equation*}
    \Ui{0} = \lim_{n \to \infty} U^{(n)} = -2 \lim_{n \to \infty} \frac{\partial_X \phi^{(n)}}{\phi^{(n)}} = -\frac{2 \partial_X \phi}{\phi}.
  \end{equation*}
  This is \eqref{eq:explicit}.
\end{proof}
We now briefly consider times after the shock forms.
Given $t > 0$, let $s(t)$ denote the position of the shock so long as it remains isolated.
Once the shock has positive strength, Goodman and Xin~\cite{GoodXin_1992} show that the viscous effects are confined to the narrow strip $\abs{x - s(t)} \lesssim \nu$.
There are thus three relevant coordinate systems in the problem: the original system $(u, t, x)$, the inner variables $(U, T, X)$, and the ``shock variables'' $(u, t, \xi)$ for $\xi \coloneqq (x-s)/\nu$.
We include the solution values $u$ and $U$ as ``coordinates'' to track their relative scaling.
These coordinates are linked through the following relations:
\begin{equation}
  \label{eq:coords}
  (U, T, X) = \big(\nu^{-1/4} u, \nu^{-1/2}t, \nu^{-3/4}x\big) \And (u, t, \xi) = \big(u, t, \nu^{-1}(x - s)\big).
\end{equation}

Here, we consider the manner in which our inner expansion about $(t,x) = (0, 0)$ gives way to the shock expansion about $x = s(t)$ as $t$ increases.
Using the Hopf--Lax formula~\cite{Hopf_1950,Lax_1957}, one can verify that $\abs{s(t)} \lesssim t^2$ and $\abs{\dot{s}(t)} \lesssim t$.
When $0 < t \ll 1$, this is sufficiently small that we can safely approximate $s(t)$ by $0$.
If we let
\begin{equation*}
  h(t) \coloneqq \frac{u^0(t, s(t)-) - u^0(t, s(t)+)}{2}
\end{equation*}
denote (half) the shock strength, we can check that
\begin{equation}
  \label{eq:height}
  h(t) \sim \sqrt{\frac{t}{\coeff}} \quad \text{as } t \to 0^+.
\end{equation}
Indeed, to leading order, $u^0$ should still resemble a solution of the cubic equation $tu^0 - \coeff (u^0)^3 = x$.
The height \eqref{eq:height} is the positive root at $x = 0$.
Thus when $t \ll 1$, $u^0$ has a weak, nearly stationary shock.

Near the shock, the viscous solution $u^\nu$ should resemble a ``viscous shock profile'' $\varphi(t, \xi)$ solving
\begin{equation*}
  \partial_\xi^2\varphi - \varphi \partial_\xi\varphi = 0, \quad \varphi(t,-\infty) = u^0(t, s(t)-), \enspace  \varphi(t,+\infty) = u^0(t, s(t)+).
\end{equation*}
At fixed $t$, all such profiles are reparameterizations of
\begin{equation}
  \label{eq:viscous-shock}
  \sigma(\xi) \coloneqq -\tanh \frac{\xi}{2}.
\end{equation}
The stationary shock of strength $h$ is $h \sigma(h\xi)$.
By \eqref{eq:height}, we therefore expect
\begin{equation*}
  u^\nu(t, \nu \xi) \sim \varphi(t, \xi) \sim \sqrt{\frac{t}{\coeff}} \sigma\left(\sqrt{\frac{t}{\coeff}}\xi\right)
\end{equation*}
when $\nu \ll 1$, $t \ll 1$, and $\abs{\xi} \lesssim 1$.
On the other hand, if our inner expansion remains valid at small positive times, we should have
\begin{equation*}
  u^\nu(t, x) \sim \ui{0}(t, x) = \nu^{1/4}\Ui{0}\left(\nu^{-1/2}t, \nu^{-3/4} x\right)
\end{equation*}
for $\nu \ll 1$, $t \ll 1$, $\abs{x} \ll 1$.
If these asymptotics for $u^\nu$ are simultaneously valid, there must be a relationship between the inner solution $\Ui{0}$ and the viscous shock profile $\sigma$.
Changing to inner coordinates via \eqref{eq:coords}, we expect the following:
\begin{proposition}
  For all $\xi \in \R$,
  \begin{equation*}
    \Ui{0}\big(T, T^{-1/2} \xi\big) = \sqrt{\frac{T}{\coeff}}\sigma\left(\frac{\xi}{\sqrt{\coeff}}\right) \left[1 + \m{O}_\xi\big(T^{-2}\big)\right] \quad \text{as } T \to \infty.
  \end{equation*}
\end{proposition}
\begin{proof}
  Fix $\xi \in \R$ and consider $T > 0$.
  By Proposition~\ref{prop:inner-explicit}, $\Ui{0} = -\frac{2 \partial_X \phi}{\phi}$ with $\phi$ given in \eqref{eq:phi}.
  Using \eqref{eq:phi}, we can write
  \begin{equation*}
    \phi\big(T, T^{-1/2}\xi\big) = \sqrt{\frac{T}{4 \pi}} \int_{\R} \e^{\xi\theta/2} \exp\left[T^2 \left(\frac{1}{4} \theta^2 - \frac{\coeff}{8} \theta^4\right)\right] \ds \theta
  \end{equation*}
  for $\theta \coloneqq T^{-1/2} \zeta$.
  We view this as the Laplace transform of an unnormalized measure that is growing ever more concentrated around the maxima of
  \begin{equation*}
    \Lambda(\theta) \coloneqq \frac{1}{4} \theta^2 - \frac{\coeff}{8} \theta^4.
  \end{equation*}
  The function $\Lambda$ is simultaneously maximized at $\theta = \pm \theta_*$ for $\theta_* \coloneqq \coeff^{-1/2}$.
  Moreover,
  \begin{equation}
    \label{eq:Lambda-stats}
    \Lambda(\theta_*) = \frac{1}{8 \coeff} \And \Lambda''(\theta_*) = -1.
  \end{equation}
  Using Laplace's method, we obtain
  \begin{equation*}
    \phi\big(T, T^{-1/2}\xi\big) = \sqrt{\frac{T}{4 \pi}} \sqrt{\frac{2\pi}{T^2\abs{\Lambda''(\theta_*)}}}\left(\e^{\xi\theta_*/2} + \e^{-\xi \theta_*/2}\right)\exp\left[T^2 \Lambda(\theta_*)\right]\left[1 + \m{O}_\xi\big(T^{-2}\big)\right]
  \end{equation*}
  as $T \to \infty$.
  Simplifying, \eqref{eq:Lambda-stats} yields
  \begin{equation}
    \label{eq:phi-asymp}
    \phi\big(T, T^{-1/2}\xi\big) = \sqrt{\frac{2}{T}} \cosh\left(\frac{\xi}{2 \sqrt{\coeff}}\right)\exp\left(\frac{T^2}{8\coeff}\right) \left[1 + \m{O}_\xi\big(T^{-2}\big)\right].
  \end{equation}
  Similarly,
  \begin{equation*}
    \partial_X \phi\big(T, T^{-1/2}\xi\big) = \frac{T}{2\sqrt{4 \pi}} \int_{\R} \theta \e^{\xi \theta/2} \exp\left[T^2 \Lambda(\theta)\right] \ds \theta
  \end{equation*}
  and
  \begin{equation}
    \label{eq:phi-deriv-asymp}
    \partial_X \phi\big(T, T^{-1/2}\xi\big) = \frac{1}{\sqrt{2 \coeff}} \sinh\left(\frac{\xi}{2 \sqrt{\coeff}}\right) \exp\left(\frac{T^2}{8\coeff}\right) \left[1 + \m{O}_\xi\big(T^{-2}\big)\right]
  \end{equation}
  as $T \to \infty$.
  Since $\Ui{0} = -2(\partial_x \phi)/\phi$, \eqref{eq:phi-asymp} and \eqref{eq:phi-deriv-asymp} imply that
  \begin{equation*}
    \Ui{0}\big(T, T^{-1/2} \xi\big) = -\sqrt{\frac{T}{\coeff}} \tanh\left(\frac{\xi}{2 \sqrt{\coeff}}\right)\left[1 + \m{O}_\xi\big(T^{-2}\big)\right]
  \end{equation*}
  as $T \to \infty$.
  In light of \eqref{eq:viscous-shock}, the proposition follows.
\end{proof}

\printbibliography
\end{document}